\documentclass[final,onefignum,onetabnum]{siamart171218}

\usepackage{lipsum}
\usepackage{amsfonts}
\usepackage{graphicx}
\usepackage{epstopdf}
\usepackage{algorithmic}
\usepackage{lineno,hyperref}
\usepackage[a4paper,text={6.0in,9.0in},centering, includefoot,foot=0.6in]{geometry}
\usepackage{setspace} 
\usepackage{mathtools} 
\usepackage{caption} 
\usepackage{subcaption}
\usepackage{amsmath}
\usepackage{float}
\usepackage{amssymb}

\ifpdf
\DeclareGraphicsExtensions{.eps,.pdf,.png,.jpg}
\else
\DeclareGraphicsExtensions{.eps}
\fi

\newsiamremark{remark}{Remark}
\newsiamthm{example}{Example}
\newsiamthm{textalgorithm}{Algorithm}
\newsiamremark{hypothesis}{Hypothesis}
\crefname{hypothesis}{Hypothesis}{Hypotheses}
\newsiamthm{claim}{Claim}

\headers{Bifurcation study of the Ginzburg-Landau equation}{M. Wouters and W. Vanroose}

\title{Automatic exploration techniques for the numerical bifurcation study of the Ginzburg-Landau equation\thanks{Submitted to SIAM Journal on Applied Dynamical Systems on 6 march 2019.
		\funding{This work was funded by the University of Antwerpen through the GOA project "Emergent Phenomena in Multicomponent Quantum Condensates".}}}

\author{Michiel Wouters\thanks{Applied Mathematics Research Group, Department of Mathematics and Computer Science, University of Antwerpen, Middelheimlaan 1, 2020 Antwerp, Belgium
		(\email{michiel.wouters2@uantwerpen.be},\email{wim.vanroose@uantwerpen.be}).}
	\and Wim Vanroose\footnotemark[2]}

\usepackage{amsopn}

\ifpdf
\hypersetup{
  pdftitle={Article},
  pdfauthor={M. Wouters, W. Vanroose}
}
\fi

\begin{document}

%

\maketitle

\begin{abstract}
This paper considers the extreme type-II Ginzburg-Landau equations, a nonlinear PDE model that describes the states of a wide range of superconductors. For two-dimensional grids, a robust method is developed that performs a numerical continuation of the equations, automatically exploring the whole solution landscape. The strength of the applied magnetic field is used as the bifurcation parameter. Our branch switching algorithm is based on Lyapunov-Schmidt reduction, but we will show that for an important class of grids an alternative method based on the equivariant branching lemma can be applied as well. The complete algorithm has been implemented in Python and tested for multiple examples. For each example a complete solution landscape was constructed, showing the robustness of the algorithm.
\end{abstract}



\section{Introduction}
Superconductors are materials that, below a certain
characteristic temperature ($T_c$), exhibit a complete loss of
electrical resistivity \cite{Du1992} and, as a result, can generate
currents that expel applied magnetic fields.

In certain types, such as bulk type II superconductors, the material
is not homogenously superconducting but there are vortices through which
magnetic fields can penetrate  the material. These vortices organise
themselves in regular patterns known as an Abrikosov lattice
\cite{Abrikosov1957}.

For small nano devices the emerging patterns of vortices depend in an
intricate way on the system parameters and the geometry of the sample.
Scientists and engineers are designing devices that have an improved
critical field and temperatures by exploiting geometrical properties
and engineering the material parameters \cite{Cordoba2013}.  It is
important to understand the effect of the geometry on the vortex
dynamics \cite{Embon2017}. The dynamics of the superconducting
materials is modelled by the Ginzburg-Landau equation, a non-linear
Schr\"odinger equation, and the emerging patterns are steady states of
this equation. Transitions between different patterns are marked by
bifurcation points.

In this paper we aim to develop a numerical tool that  generates
automatically \textit{all} connected patterns that appear in a given
superconductor geometry and indentifies \textit{all} the bifurcation
points that define the transitions between patterns as parameters of
the system change.  There is wide interest from material scientists
and device engineers to understand what parameters are determining the
stability of superconducting states.  In particular, the aim is to
understand the dynamics of the transitions between these patterns and
what needs to be changed in the geometry or parameter settings to
prevent the system from making a transition that destroys a given pattern.

The challenge is that sparse linear algebra is required to solve the
large systems that describe the Ginzburg-Landau equation.  Existing
tools such as Auto \cite{Auto} and Matcont \cite{MatCont} can automatically
generate bifurcation diagrams for small systems of coupled ordinary
differential equations.  However, these tools are based on dense
linear algebra and they cannot scale to the large sparse systems that
appear in the description of superconductors.  In this paper we
develop the required tools that can solve the large systems of
equations solely based on sparse linear algebra.

\subsection*{Describing the state of a superconductor}
The state of superconductors is described by the Ginzburg-Landau
partial differential equation. We will consider samples of
superconducting material that occupy an open, bounded region $\Omega$
of the two-dimensional Euclidean space, subject to an external
magnetic field $\mathbf{H}_0$ (see Figure \ref{fig - states}).

\begin{figure}
	\centering
	\begin{subfigure}[h]{0.35\textwidth}
		\includegraphics[trim = 0mm 0mm 0mm 0mm,clip,scale=0.5]{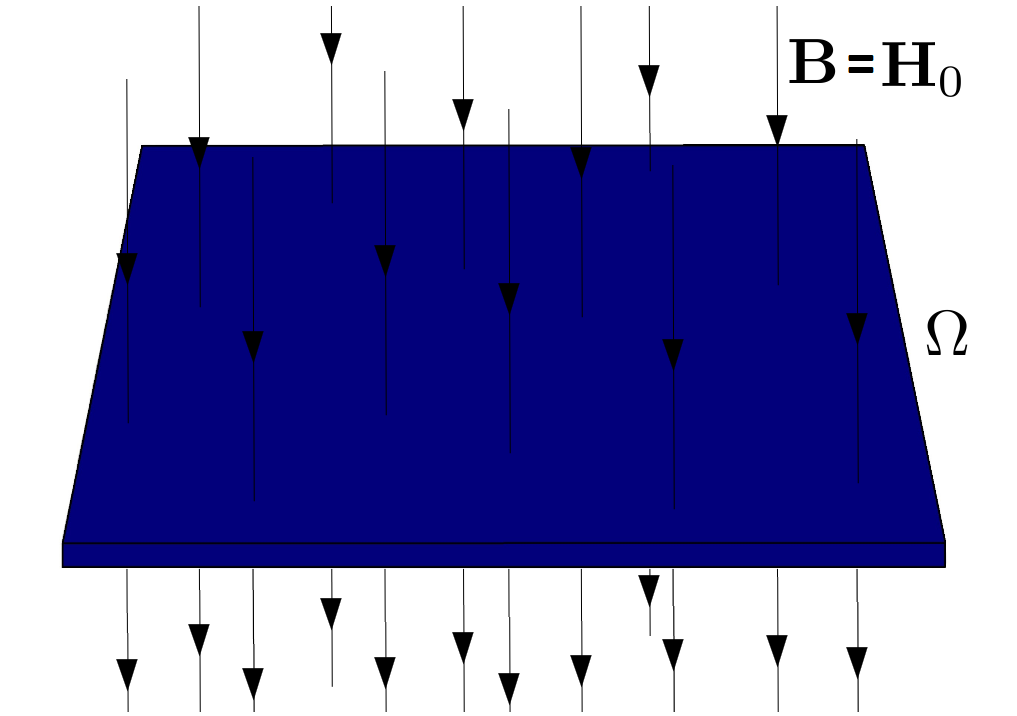}
		\caption{The normal state}
		\label{fig - normal state}
	\end{subfigure}
	\hfil
	\begin{subfigure}[h]{0.38\textwidth}
		\includegraphics[trim = 0mm 0mm 0mm 0mm,clip,scale=0.5]{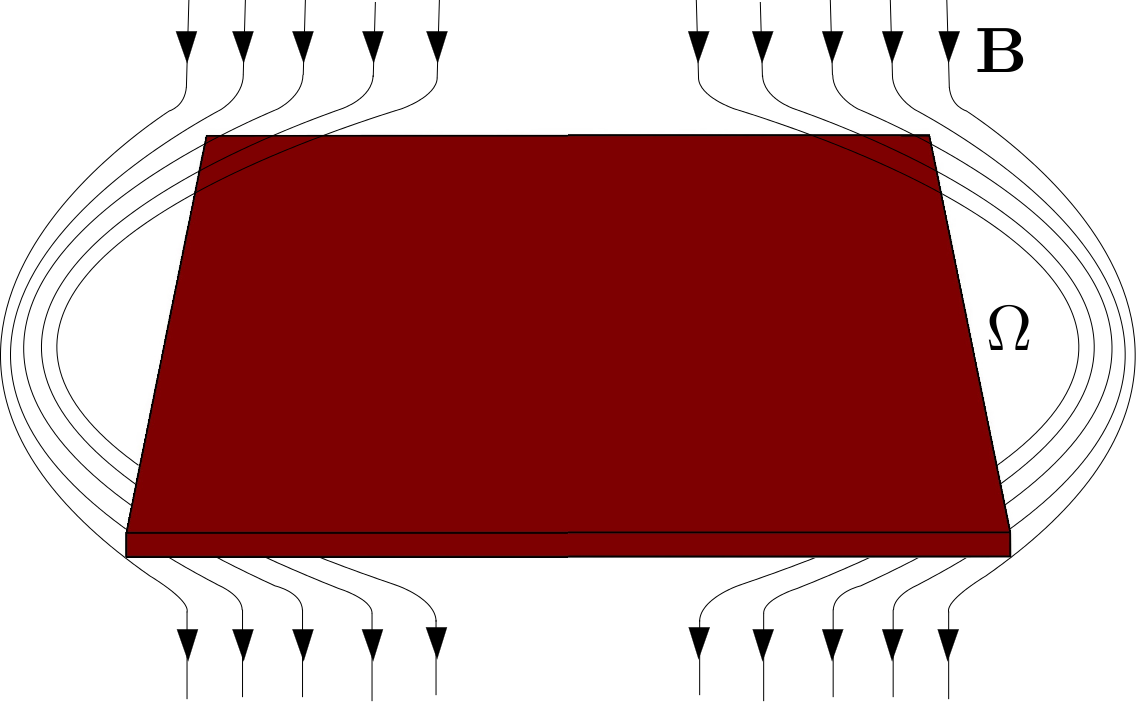}
		\caption{The superconducting state}
		\label{fig - superconducting state}
	\end{subfigure}
	
	\begin{subfigure}[h]{0.37\textwidth}
		\includegraphics[trim = 0mm 0mm 0mm 0mm,clip,scale=0.7]{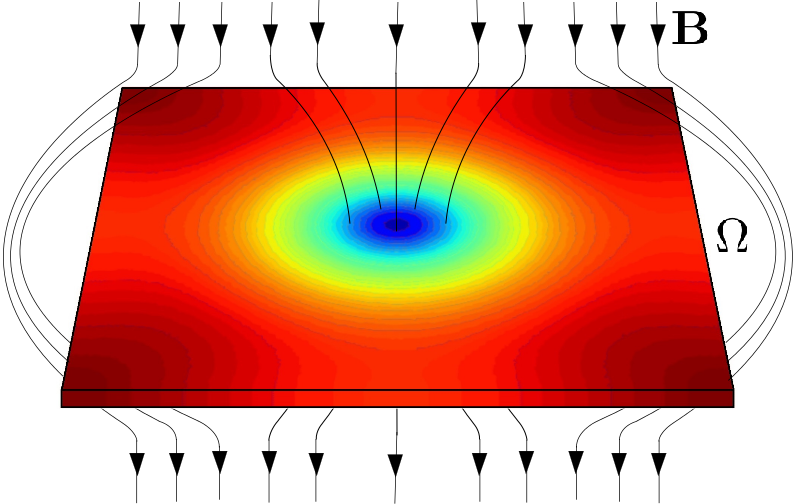}
		\caption{The mixed state}
		\label{fig - mixed state}
	\end{subfigure}
	\caption{Possible states of a superconducting material subject to an external magnetic field $\mathbf{H}_0$.} \label{fig - states}
\end{figure}

There are multiple types of superconductors. They are classified as
type-I and type-II. A type-I superconducting material can only be
in the normal or the superconducting state. In the normal state, or
homogeneously non-superconducting, the material behaves like a normal
conductor, i.e. the total magnetic field $\mathbf{B}$ entirely
penetrates the sample (Figure \ref{fig - normal state}). The
electrical current has a resistance. This state occurs for high
temperatures ($T>T_c$) or for magnetic field strengths above a certain
critical value $\mu_c$.  In contrast, in the homogeneously
superconducting state, the total magnetic field $\mathbf{B}$ is
entirely expelled from the interior of the sample and the material
exhibits zero electrical resistance (Figure \ref{fig - superconducting
  state}). This state occurs for low temperatures ($T<T_c$) provided
that the external magnetic field's strength is sufficiently low
($\mu<\mu_c$).

In type-II superconductors, an additional third state occurs for
temperatures and magnetic field strengths between two critical values
($T_{c_1}<T<T_{c_2}$ and $\mu_{c_1}<\mu<\mu_{c_2}$). In this case, the
magnetic field $\mathbf{B}$ only locally penetrates the material. This gives
rise to vortices of non-superconducting material around which
superconducting currents appear in the material (Figure \ref{fig -
  mixed state}). This extra state that only occurs for type-II
superconductors is called the \textit{mixed} state \cite{Goodman1966}.
It is the vortex patterns found in these type-II materials that are
the focus of the current paper.

The state of a superconductor is described by two quantities: the
total magnetic field $\mathbf{B}:\mathbb{R}^3\rightarrow \mathbb{R}^3$
and the density of Cooper pairs $\rho_C:\Omega \cup
\partial\Omega\rightarrow\mathbb{R}$. Cooper pairs are pairs of
electrons that constitute superconductivity: a high density
corresponds with the superconducting state. Both quantities are
determined by the Ginzburg-Landau system \cite{Du1992}, which includes
a partial differential equation for an order parameter
$\psi:\Omega\cup\partial\Omega \rightarrow \mathbb{C}$ with
$|\psi|^2=\rho_C$. In the current paper we will only consider
\textit{extreme} type-II superconductors, for which the
Ginzburg-Landau system decouples in a system for $\psi$ and the
magnetic field is independent of $\mathbf{A}$.

\subsection*{Motivation for the current paper}
The understanding how vortex patterns are formed and how they lose
their stability is important for multiple applications of extreme
type-II superconductors.  This is theoretically investigated by
numerically solving the steady states of the Ginzburg-Landau
system. After discretization it results in a large sparse non-linear
system. An efficient numerical solver for this system is described in
\cite{Schlomer_solver}, based on an AMG-preconditioned Newton-Krylov
method. Unlike other popular methods (e.g. time-stepping the
Ginzburg-Landau system via Gauss-Seidel iterations \cite{Baelus2002}),
this solver can find both physically stable and unstable solutions.
This solver is only one piece of the machinery that is required to
automatically generate bifurcation diagrams.

Though the direct use of unstable steady states is limited in real
situations, they still contain important information on transitions
between the stable patterns. This is further motivated by example
\ref{example - intro} below.

Numerical evidence in
\cite{Schlomer_solver} shows an independency between the number of
unknowns and the number of Krylov iterations required to converge,
demonstrating the optimality of the presented solver. To achieve this,
the solver makes full use of the sparsity properties of the system.

\begin{example} \label{example - intro}
	We consider a square material of side length $3$ subject to a
        homogeneous magnetic field with strength $\mu=1.3$. For these
        values, the material is in the mixed state, but can adopt two
        different stable patterns, shown in figure \ref{fig - intro
          example patterns} \cite{Schlomer_square}. Both patterns 1
        and 2 are stable, their respective energies are given by $-0.23387$
        and $-0.13271$.  We keep the value for the magnetic field
        strength fixed and want to perform a transition from pattern 1
        to pattern 2 through a temporary external perturbation, by supplying additional energy.
	
	It is, however, not sufficient to add the difference in
        energy between patterns 1 and 2 to the system. Instead, the minimal amount of energy that needs to be
        supplied is given by the energy difference between pattern 1 and
        the solution presented by pattern 3 (figure \ref{fig - intro
          example pattern 3}). This last pattern represents an
        unstable solution for the same magnetic field strength, with an energy of $-0.14007$. To
        perform a transition between patterns 1 and 2, we first need to supply energy to the system to reach the barrier implied by pattern
        3. Afterwards energy is released to reach the steady state
        described by pattern 2. This is further indicated in figure
        \ref{fig - intro example schematic transition}.
\end{example}

\begin{figure}
	\centering
	\begin{subfigure}[h]{0.35\textwidth}
		\centering
		\includegraphics[trim = 0mm 0mm 0mm 0mm,clip,scale=0.4]{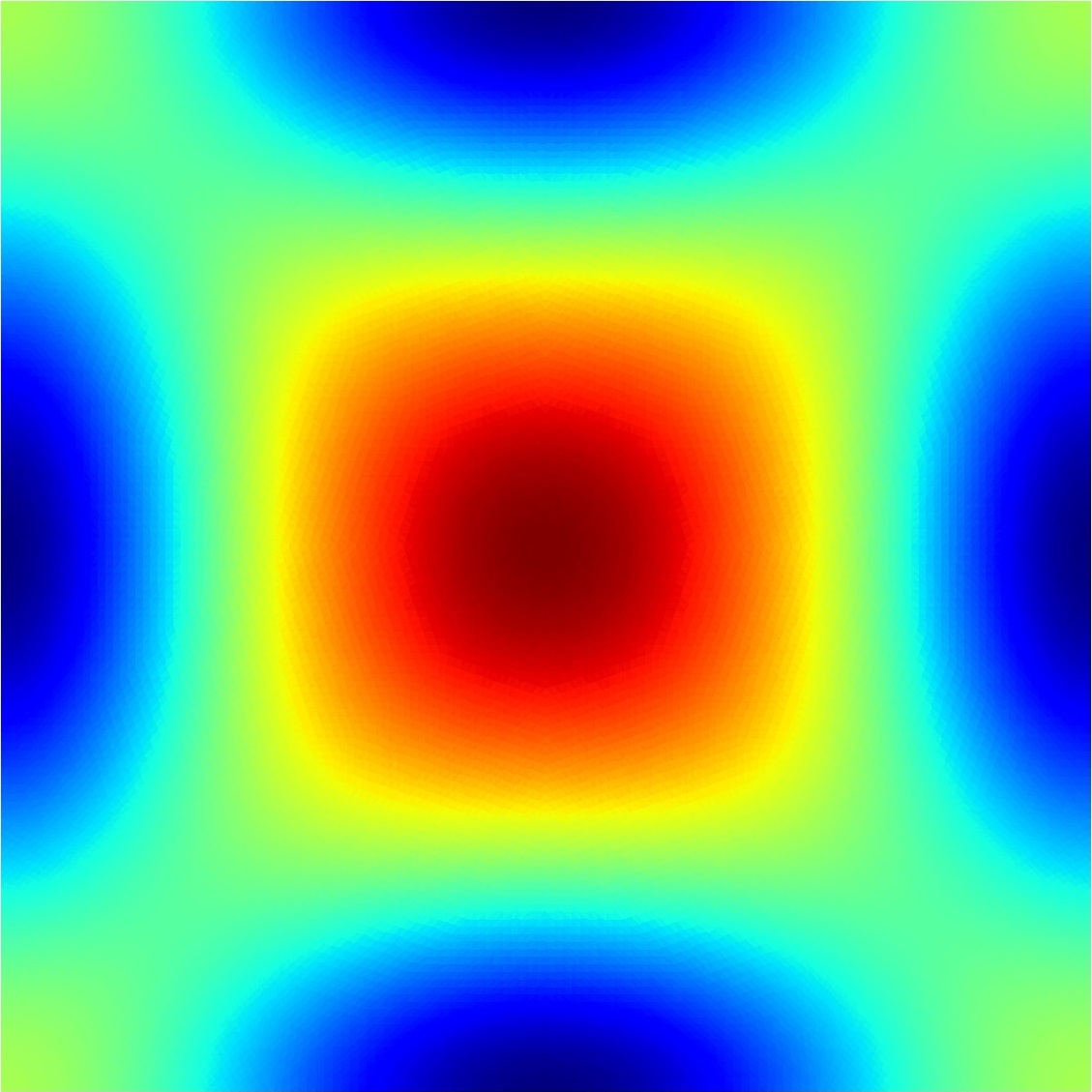}
		\caption{Pattern 1}
		\label{fig - intro example pattern 1}
	\end{subfigure} 
	\hfil
	\begin{subfigure}[h]{0.38\textwidth}
		\centering
		\includegraphics[trim = 0mm 0mm 0mm 0mm,clip,scale=0.4]{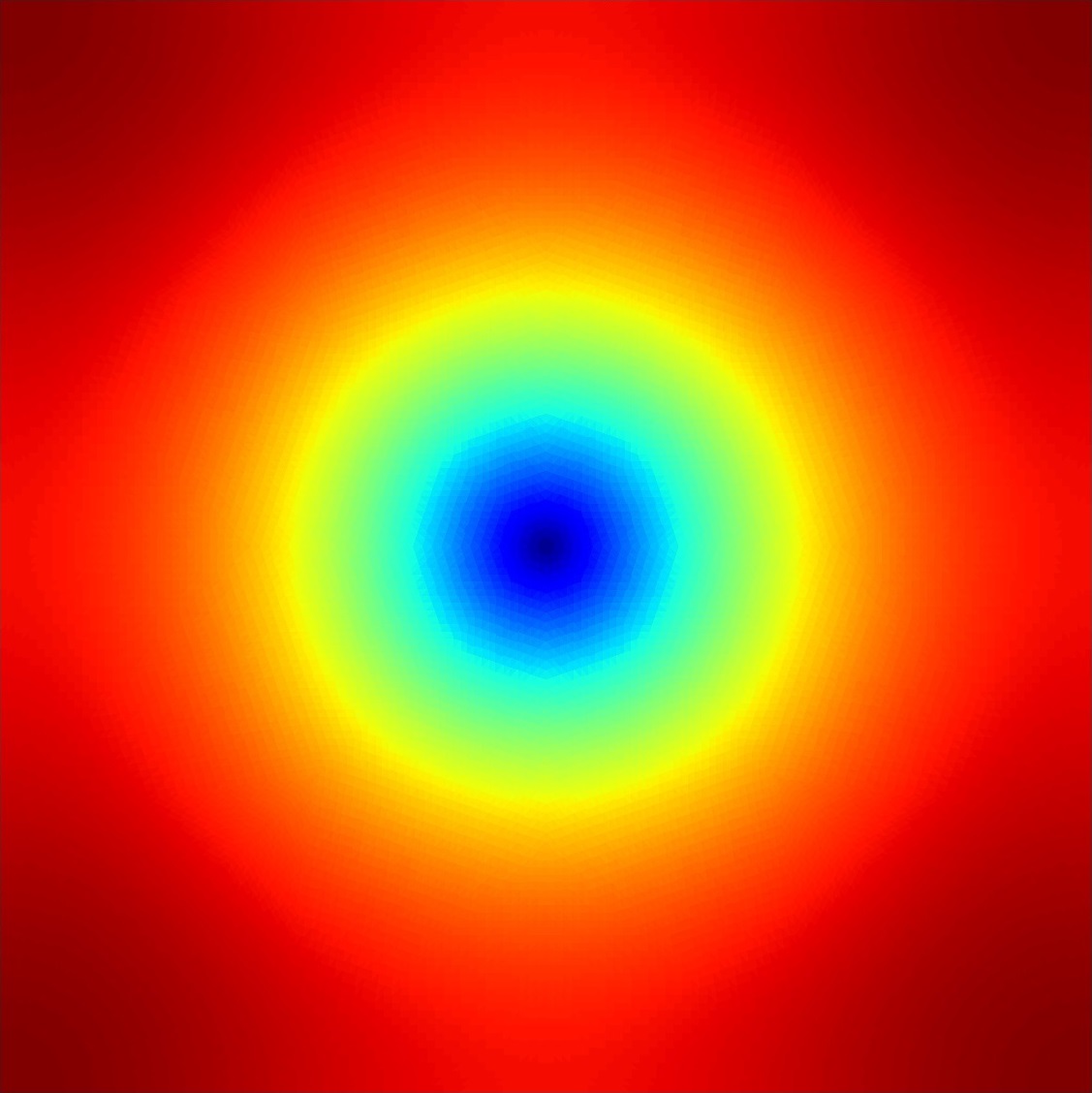}
		\caption{Pattern 2}
		\label{fig - intro example pattern 2}
	\end{subfigure} 
	
	\begin{subfigure}[h]{0.37\textwidth}
		\centering
		\includegraphics[trim = 0mm 0mm 0mm 0mm,clip,scale=0.4]{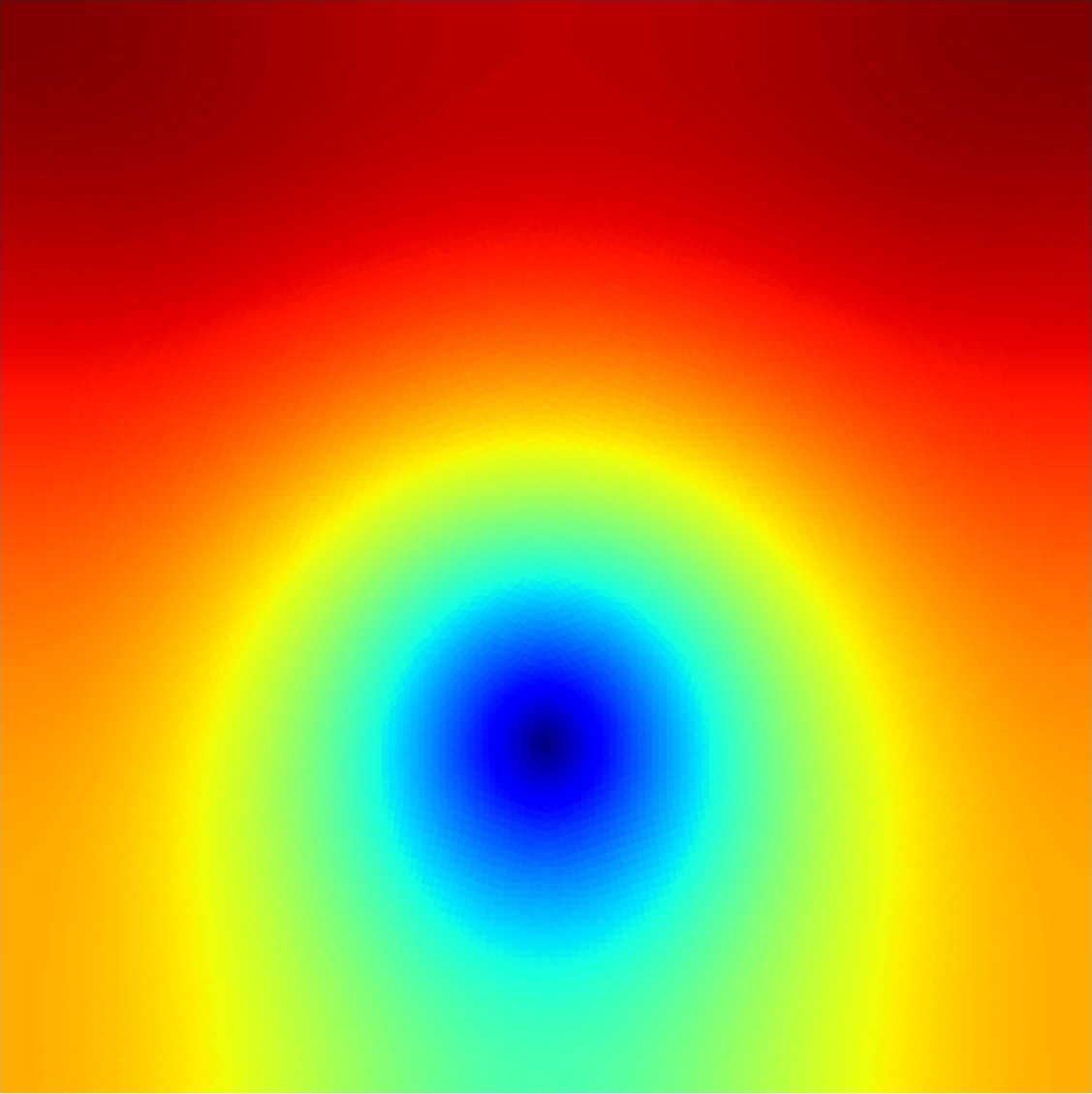}
		\caption{Pattern 3}
		\label{fig - intro example pattern 3}
	\end{subfigure} 

	\caption{Possible vortex patterns for the problem described in example \ref{example - intro}} \label{fig - intro example patterns}
\end{figure}

\begin{figure}
	\centering
	\includegraphics[trim = 0mm 0mm 0mm 0mm,clip,scale=0.4]{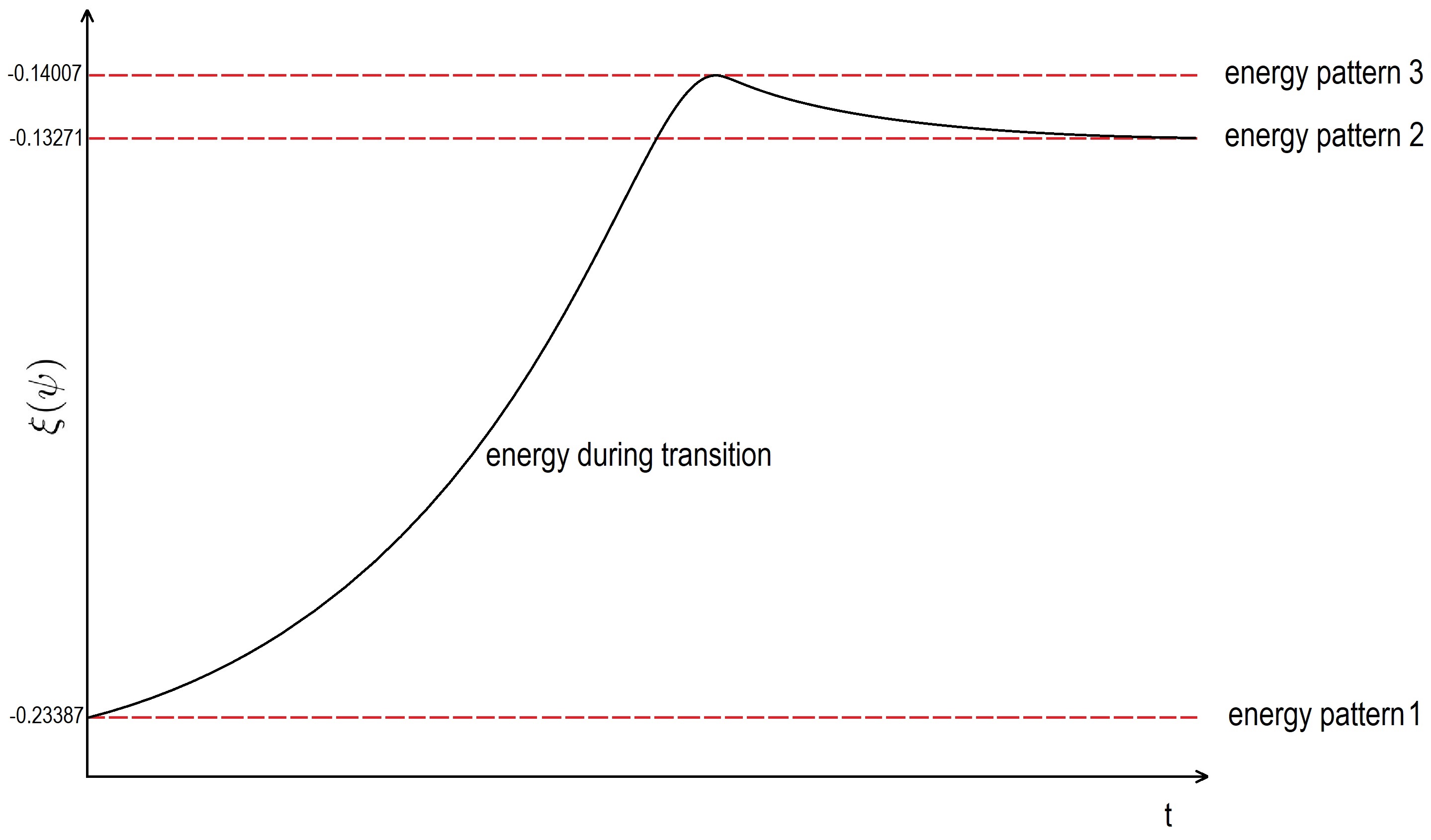}
	\caption{Schematic transition between the two stable patterns presented in example \ref{example - intro}. $t$ represents the time, $\xi(\psi)$ the energy of the state. In order to perform a transition between patterns 1 and 2, the energy barrier implied by pattern 3 needs to be crossed.}  \label{fig - intro example schematic transition}
\end{figure}

To investigate how vortex patterns change when certain parameters
(e.g. the strength of the applied magnetic field) are altered, a
numerical continuation of the Ginzburg-Landau system is
performed. This has already been the topic of some papers as well. In
\cite{Schlomer_square} two different sized square samples subject to a
homogeneous external magnetic field are concerned and their solution
landscapes are explored. In \cite{Schlomer_disc}, square samples in
the field of a magnetic disc are studied. Such bifurcation diagrams
are helpful when facing problems like the one described in example
\ref{example - intro} as well. This is indicated by example
\ref{example - intro cont}.

\begin{example} \label{example - intro cont}
	We again consider the set-up of example \ref{example -
          intro}. Figure \ref{fig - intro example bifurcation diagram}
        was recreated from \cite{Schlomer_square} and shows the
        bifurcation diagram of the problem. At parameter value
        $\mu=1.3$ we have multiple solutions: two stable ones that lie
        on curves A (pattern 1 of figure \ref{fig - intro example
          patterns}) and curve D (pattern 2), and two unstable ones on
        curves B and C (pattern 3). In order to perform a transition
        from the solution on A to the one on D, we need to add enough
        energy to bypass the barrier induced by the unstable solution
        on curve C - the unstable solution with the lowest energy that
        lies on a curve connecting A and D.
\end{example}

\begin{figure}
	\centering
	\includegraphics[trim = 0mm 0mm 0mm 0mm,clip,scale=0.4]{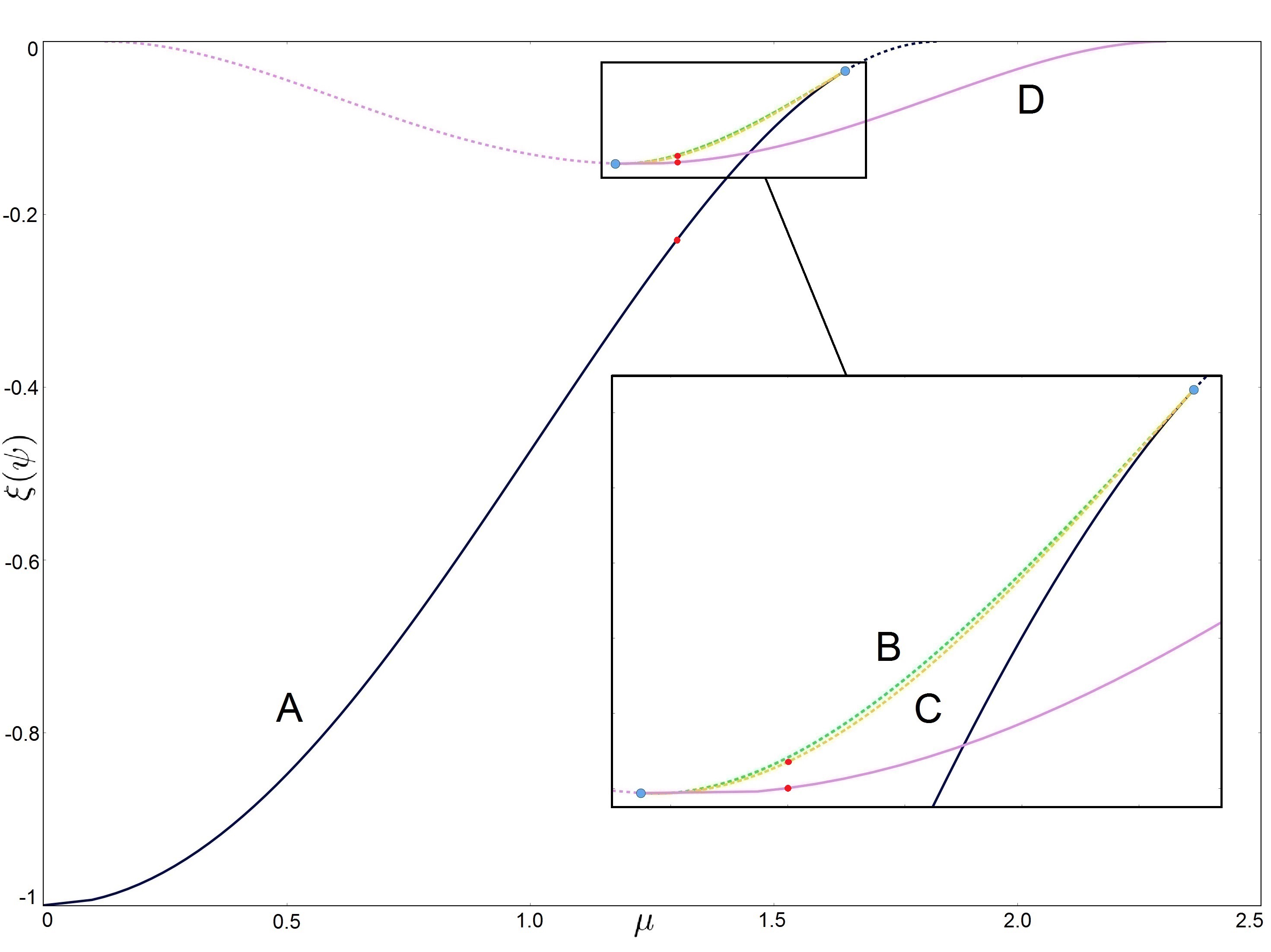}
	\caption{Recreated from \cite{Schlomer_square}. Bifurcation diagram connected to example \ref{example - intro}. Blue dots indicate bifurcation points, red dots the points associated with the patterns given in figure \ref{fig - intro example patterns}.} \label{fig - intro example bifurcation diagram}
\end{figure}

Though the solution landscapes provided by \cite{Schlomer_square} and
\cite{Schlomer_disc} already contain a lot of information, they were
not automatically explored and are possibly incomplete. To fully
investigate the possible steady states and their transitions, it is
essential to use automatic exploration techniques in order to receive
a complete solution landscape.

Besides the tools already mentioned earlier (such as AUTO \cite{Auto}
and MatCont \cite{MatCont}) there is LOCA \cite{Loca} that is
developed around sparse linear algebra, but is less easy to use and
requires expert knowledge in HPC hardware. Furthermore it does not
include a branch switching functionality.

The goal of the current paper is to extend the analysis done in
\cite{Schlomer_square} and \cite{Schlomer_disc} by providing a robust
method that does not only perform a numerical continuation, but
automatically explores the solution landscape as well. We will use the
Python package PyNCT \cite{Draelants2015} as a basis for the
implementation, and extend it with automatic exploration
techniques. The solver described in \cite{Schlomer_solver} is an
important building block for this purpose, as it fully exploits the
sparsity of the equations. To the best of our knowledge, an automatic
exploration of the solution landscape of the Ginzburg-Landau problem
has never been provided before.

To emphasize the details of the Ginzburg-Landau problem, in the
current paper we will specifically derive the methods required for
automatic exploration using this problem. The algorithms are however
easily extended for general problems, and a general version was
implemented in PyNCT as well.

\subsection*{Automatic exploration techniques}
Automatic exploration of the bifurcation diagram is realized in two
steps. First, the bifurcation points that are encountered during the
continuation of the patterns are identified. In these points the
Jacobian of the system has one or multiple zero eigenvalues. A method
suggested in \cite{Mei2000} uses this property to identifiy the
bifurcations; after each numerical continuation step the lowest
magnitude Ritz values of the Jacobian are monitored to indicate a
nearby bifurcation point. If a nearby bifurcation point is indicated
an extended system of non-linear equations is solved that identifies
the solution and the parameter value of the bifurcation point,
simultaniously. It is solved by a Newton-Krylov solver.

The second step of this automatic exploration constructs for each
bifurcation point, identified in the first step, the tangent
directions of the solution curves emerging from these points.  For
this purpose we propose two different methods in the current paper.

The first method is based on a Lyapunov-Schmidt reduction
\cite{Mei2000}. Here, the problem is reduced to a low dimensional
system of algebraic equations, that contain all of the information
about the bifurcation's behaviour and are easier to analyse. These
systems define the tangent directions to the intersecting solution
branches. We propose an iterative algorithm for the
determination of these directions that is inpired by the
analysis done in \cite{Mei1996}. The derived algorithm can always be
applied to the bifurcations that are encountered in the
Ginzburg-Landau equation, but requires multiple solves of a linear
system. The amount of these solves is linked to the symmetry group
of the problem, and this considerably reduces the speed of the algorithm
when highly symmetrical materials (e.g. a decagon) are considered.

We also propose a second method that is based on the equivariant
branching lemma. Here, the lemma links the symmetry group of the
problem to predict the symmetries of the branches that emerge at
bifurcation points \cite{Golubitsky2002,Hoyle2006}. In
\cite{Schlomer_square} the equivariant branching lemma was used for
the prediction of the symmetry groups of emerging branches, but it was
not used for the calculation of the tangent directions
themselves. This alternative to the first approach only requires a
single linear system solve and is, hence, a good alternative. However,
its use is limited to the class of samples of materials with dihedral
symmetry.

These two methods allow us to determine the tangent directions of the
solution curves that emerge from a given branch point. For each new
tangent direction, a new numerical continuation is started, and the
processes of calculating branch points and constructing tangent
directions are repeated. Eventually no new (connected) branch points
can be found and the algorithm stops, providing us with a complete solution
diagram of multiple interconnected solution curves.  The Python
implementation of the full algorithm was tested for multiple examples
of two-dimensional superconductors, each with different
symmetries. For each example a complete solution landscape was found
(in the sense of interconnecting curves, the ones not in any way
connected to the starting point cannot be found), emphasizing the
robustness of the algorithm.

\subsection*{Outline}
The remainder of the article is organized as follows. Section
\ref{sectie Ginzburg-Landau systeem} reviews the Ginzburg-Landau
system for extreme-type-II superconductors and its symmetries. The
properties of its Jacobian and other derivatives are also
discussed. Details of numerical continuation can be found in section
\ref{sectie numerieke continuatie}: the solver described by
\cite{Schlomer_solver} is shortly discussed and its role in our method
is indicated, as well as some details of challenges we encountered for
the numerical continuation itself. The main contribution of the
article is contained in sections \ref{sectie bifurcaties} and
\ref{sectie tangents}: while section \ref{sectie bifurcaties}
discusses the detection and determination of bifurcation points,
section \ref{sectie tangents} concerns the construction of the tangent
directions to new solution branches. Numerical computations of
solution landscapes for multiple two-dimensional samples are included
in section \ref{sectie resulaten}. For these examples the strength of
the applied magnetic field is used as a bifurcation parameter. Section
\ref{sectie conclusie} concludes with a discussion of the obtained
results. Appendix \ref{app proof} contains the proof of an algorithm
provided in section \ref{sectie tangents}, appendix \ref{app calc}
contains some details on how to calculate certain terms in this
algorithm.

\subsection*{Notations}
For $z\in\mathbb{C}$, we denote $\mathfrak{Re}(z)$ and
$\mathfrak{Im}(z)$ as respectively its real and imaginary
part. $\bar{z}$ is used for complex conjugation. Similarly, for
complex valued functions $\varphi:\Omega\rightarrow\mathbb{C}$ we
define $\bar{\varphi}:\Omega\rightarrow\mathbb{C}$ such that $\forall
x\in\Omega: \bar{\varphi}(x) = \overline{\varphi(x)}$. Adjoints of
linear operators $\mathcal{A}$ are denoted by $\mathcal{A}^*$. For the
symmetry groups under consideration, the symbol $S^1$ is used to
denote the circle group $\{z\in\mathbb{C}: |z|=1\}$, $D_m$ ($m\geq 2$)
for dihedral groups and $C_m$ ($m\geq 2$) for cyclic groups.

\section{The Ginzburg-Landau system and its properties} \label{sectie Ginzburg-Landau systeem}
\subsection*{Description of the system}
We will only consider the Ginzburg-Landau problem for extreme type-II superconductors. In \cite{Schlomer_square} it is shown that the Ginzburg-Landau equations decouple for this case, simplifying the problem. A short overview of the resulting equation is given below, based on the analysis done in \cite{Schlomer_square}. \\

For an open, bounded domain $\Omega\subset\mathbb{R}^3$, with a piecewise smooth boundary $\partial \Omega$, the Ginzburg-Landau problem is derived by minimizing the Gibbs free energy functional \cite{Du1992}
\begin{align} \label{Gibbs free energy functional}
	\begin{aligned}
		G(\psi,\mathbf{A}) - G_n = \xi \frac{|\alpha|^2}{\beta} \int_{\Omega} \bigg[ &-|\psi|^2 + \frac{1}{2} |\psi|^4 + ||-\i \nabla \psi - \mathbf{A} \psi ||^2 \\
		&+ \kappa^2 (\nabla\times \mathbf{A})^2 -2\kappa^2 (\nabla\times \mathbf{A}) \cdot \mathbf{H}_0 \bigg] d\Omega.
	\end{aligned}
\end{align}
The state $(\psi,\mathbf{A})$ is in the natural energy space such that the integral is well-defined. $\psi$ represents a scalar-valued function and is commonly referred to as the \textit{order parameter}, and the magnetic vector potential corresponding to the total magnetic field is given by $\mathbf{A}$. The physical observables associated with the state $(\psi,\mathbf{A})$ are the density $\rho_C=|\psi|^2$ of the Cooper pairs and the total magnetic field $\mathbf{B} = \nabla \times \mathbf{A}$. The constant $G_n$ represents the energy associated with the normal (non-superconducting) state.
The energy \eqref{Gibbs free energy functional} depends upon the impinging magnetic field $\mathbf{H}_0$ and the material parameters $\alpha,\beta,\lambda,\xi\in\mathbb{R}$. It is presented in its dimensionless form, where the domain $\Omega$ is scaled in units of the coherence length $\xi$. The ratio $\kappa=\lambda/\xi$ of the penetration depth $\lambda$ and the coherence length $\xi$ completely determines the type of the superconductor: for $\kappa<1/\sqrt{2}$ it is said to be of type I, otherwise the superconductor is said to be of type II.

Using standard calculus of variations, minimization of the Gibbs free energy functional gives rise to the Ginzburg-Landau equations \cite{Du1992}: a boundary-value problem in the unknowns $\psi$ and $\mathbf{A}$. For extreme type-II superconductors the limit $\kappa\rightarrow\infty$ is considered, in this case the Ginzburg-Landau problem decouples for $\psi$ and $\mathbf{A}$ \cite{Schlomer_square}. The magnetic vector potential $\mathbf{A}$ is completely determined by the applied magnetic field $\mathbf{H}_0$ through the system
\begin{equation} \label{formule voor magnetische vector potentiaal}
	\begin{cases}
		\nabla \times (\nabla \times \mathbf{A}) = 0 &\text{in } \Omega, \\
		\mathbf{n} \times (\nabla \times \mathbf{A}) = \mathbf{n} \times \mathbf{H}_0 &\text{on } \partial\Omega.
	\end{cases}
\end{equation}

With the magnetic vector potential determined, the order parameter $\psi$ is derived by solving the equation
\begin{align}
	\begin{aligned} \label{Ginzburg-Landau vergelijking}
		&\mathcal{GL}:X\times\mathbb{R} \rightarrow Y, \\
		&0 = \mathcal{GL}(\psi,\mu) =
		\begin{cases}
			(-\i\nabla -\mathbf{A}(\mu))^2\psi - \psi(1-|\psi|^2) &\text{in } \Omega, \\
			\mathbf{n} \cdot (-\i\nabla -\mathbf{A}(\mu))\psi &\text{on } \partial\Omega.
		\end{cases}
	\end{aligned}
\end{align}
The space $X$ corresponds to the natural energy space over $\Omega$ associated with the Gibbs energy \eqref{Gibbs free energy functional} and $Y$ to its dual space. The dependence on the strength of the applied magnetic field ($\mu$) is explicitly given since it will be used as a bifurcation parameter. For the remainder of the paper we will write $\mathbf{A}$ instead of $\mathbf{A}(\mu)$, dropping the explicit dependency on $\mu$. \\

In the present paper, we will consider two-dimensional grids $\Omega\subset \mathbb{R}^2$ and numerically solve \eqref{Ginzburg-Landau vergelijking} for the order parameter. A bifurcation analysis will be carried out with the strength of the applied magnetic field ($\mu$) as the bifurcation parameter.

\subsection*{Symmetries}
The equivariant branching lemma is crucial for one of the automatic exploration methods. To apply this lemma, the symmetry group of \eqref{Ginzburg-Landau vergelijking} needs to be determined. In \cite{Schlomer_square} it is shown that, for square samples subject to a perpendicular, homogeneous magnetic field, this symmetry group is given by $S^1\times D_4$, with $S^1$ and $D_4$ respectively the circle group and symmetry group of the square. A more general result is given in proposition \ref{prop - symmetrie}. \\

\begin{proposition} \label{prop - symmetrie}
	Let the sample $\Omega\subset\mathbb{R}^2$ and the applied magnetic field $\mathbf{H}_0$ both be invariant under the actions of a dihedral group $D_m=\langle\tau_\omega,\sigma\rangle$ (with $\omega=2\pi/m$, $m\in\mathbb{N}$) defined by
	\begin{align*}
		&\tau_\omega: \mathbb{R}^2 \rightarrow \mathbb{R}^2: \begin{pmatrix}
			x \\ y
		\end{pmatrix} \rightarrow \begin{pmatrix}
			\cos(\omega) & -\sin(\omega) \\
			\sin(\omega) & \cos(\omega)
		\end{pmatrix}\begin{pmatrix}
			x \\ y
		\end{pmatrix}, \\
		&\sigma: \mathbb{R}^2 \rightarrow \mathbb{R}^2: \begin{pmatrix}
			x \\ y
		\end{pmatrix} \rightarrow \begin{pmatrix}
			-x \\ y
		\end{pmatrix},
	\end{align*}
	then \eqref{Ginzburg-Landau vergelijking} is invariant under the actions of $S^1 \times D_m = \langle\theta_\eta,\tau_\omega,\sigma\rangle$, given by
	\begin{align*}
		&\theta_\eta: X\rightarrow X: \psi \rightarrow e^{\i\eta}\psi \quad (\forall \eta \in [0,2\pi]), \\
		&\tau_\omega: X\rightarrow X: \psi(x,y) \rightarrow \psi(\tau_\omega(x,y)), \\
		&\sigma: X\rightarrow X: \psi(x,y) \rightarrow \overline{\psi(-x,y)}.
	\end{align*}
	A similar result holds for invariance under the actions of the cyclic group $C_m = \langle\tau_\omega\rangle$.
\end{proposition}
Note that proposition \ref{prop - symmetrie} is only valid for 2-dimensional materials. For 3-dimensional ones an analogue result can be derived, but this is beyond the scope of the current paper.

Without any sample or magnetic field restrictions, the symmetry group of \eqref{Ginzburg-Landau vergelijking} is
\begin{displaymath}
	\Gamma = S^1 \times D_\infty.
\end{displaymath}
The shape of the material (and possibly its discretization) and magnetic field give rise to restrictions on this group. As example, for a triangular sample subject to a homogeneous magnetic field, the symmetry group is given by $S^1\times D_3$. A circular sample subject to a (centered) magnetic square would restrict the group to $S^1\times D_4$.

The continuous $S^1$ symmetry persists for every choice of material and magnetic field: every solution $\psi$ of \eqref{Ginzburg-Landau vergelijking} is actually a representative of a whole family of solutions $\{\theta_\eta \psi | \theta_\eta\in S^1\}$. Since we are mainly interested in the density of Cooper pairs $\rho_c=|\psi|^2$ and the equation $|\psi|=|\theta_\eta\psi|$ holds for each $\theta_\eta\in S^1$, it is sufficient to consider a single representative for each family \cite{Schlomer_square}.

\subsection*{Properties of the Jacobian operator}
The solver presented in \cite{Schlomer_solver} is based on a Newton-Krylov algorithm. In order to apply this algorithm to \eqref{Ginzburg-Landau vergelijking}, the (partial) Jacobian operator (to $\psi$) associated with this equation is required. In \cite{Schlomer_square} an expression for this operator is derived:
\begin{equation} \label{Jacobiaan naar psi}
	\mathcal{J}_\psi(\psi,\mu): X\rightarrow Y: \varphi \rightarrow \left((-\i\nabla-\mathbf{A})^2-1+2|\psi|^2\right)\varphi + \psi^2 \overline{\varphi}.
\end{equation}
This operator is only linear when defined over $X$ and $Y$ as $\mathbb{R}$-vector spaces \cite{Schlomer_square}.
The Jacobian operator is self-adjoint with respect to the inner product
\begin{equation} \label{inproduct}
	\langle\cdot,\cdot\rangle_\mathbb{R} = \mathfrak{Re}\langle\cdot,\cdot\rangle_{L^2_{\mathbb{C}}(\Omega)},
\end{equation}
which coincides with the natural inner product in $(L^2_{\mathbb{R}}(\Omega))^2$ \cite{Schlomer_square}.
The spectrum of \eqref{Jacobiaan naar psi} has been investigated in both \cite{Schlomer_square} and \cite{Schlomer_solver}. These papers show that this spectrum is a subset of $\mathbb{R}$, but $\mathcal{J}_\psi(\psi,\mu)$ is generally not definite. Definiteness is completely determined by the state ($\psi,\mathbf{A}$). Solutions $\psi$ for which the Jacobian operator does not have positive eigenvalues are said to be physically stable, and physically unstable otherwise.

Another result that's discussed in \cite{Schlomer_square} and \cite{Schlomer_solver} is the zero eigenvalue induced by the continuous $S^1$ symmetry. For every $\psi_s\in X,\mu_s\in\mathbb{R}$ we have
\begin{displaymath}
	\mathcal{GL}(\psi_s,\mu_s)=0 \Rightarrow \mathcal{J}_\psi(\psi_s,\mu_s)(\i\psi_s) = 0,
\end{displaymath}
implying that the kernel of the Jacobian operator is never empty in a solution ($\psi_s,\mu_s$) since $\{\i\psi_s\}\subset \ker(\mathcal{J}_\psi(\psi_s,\mu_s))$.

Aside from the partial Jacobian operator to $\psi$, the one to $\mu$ is also required for numerical continuation. We will approximate the Jacobian operator $\mathcal{J}_\mu(\psi,\mu)$ by a second-order finite difference scheme. Finally, the automatic exploration methods require the full Jacobian of \eqref{Ginzburg-Landau vergelijking}, defined by
$$\mathcal{J}(\psi,\mu):X\times\mathbb{R}\rightarrow Y: \begin{pmatrix}
\varphi \\
\kappa
\end{pmatrix}\rightarrow \mathcal{J}_\psi(\psi,\mu)\varphi + \kappa\mathcal{J}_\mu(\psi,\mu).$$

\subsection*{Other derivatives of the Ginzburg-Landau equations}
The partial Hessian operators ($\mathcal{H}_{\psi\psi}(\psi,\mu)$, $\mathcal{H}_{\psi\mu}(\psi,\mu)$ and $\mathcal{H}_{\mu\mu}(\psi,\mu)$) of \eqref{Ginzburg-Landau vergelijking} are crucial for the determination of bifurcation points and new tangent directions. We will again apply second-order finite difference schemes to approximate the operators $\mathcal{H}_{\psi\mu}(\psi,\mu)$ and $\mathcal{H}_{\mu\mu}(\psi,\mu)$.
Instead of approximating $\mathcal{H}_{\psi\psi}(\psi,\mu)$ by finite differences as well, we derive an exact expression for this operator. This is done by applying a similar technique as in \cite{Schlomer_square} for the derivation of the Jacobian operator. Let $\psi,\varphi,\delta \psi\in X$ and $\mu\in\mathbb{R}$. We have
\begin{align*}
	\mathcal{J}_{\psi}(\psi+\delta\psi,\mu)\varphi-\mathcal{J}_\psi(\psi,\mu)\varphi =&\left((-i\nabla-\mathbf{A})^2-1+2|\psi+\delta\psi|^2\right)\varphi + (\psi+\delta\psi)^2\overline{\varphi} \\
	&-\left((-i\nabla-\mathbf{A})^2-1 +2|\psi|^2\right)\varphi -\psi^2\overline{\varphi} \\
	=& 2\left(\overline{\psi}\delta\psi\varphi+\psi\overline{\delta\psi}\varphi+\psi\delta\psi\overline{\varphi}\right) + 2|\delta\psi|^2\varphi + \left(\delta\psi\right)^{2}\bar{\varphi}.
\end{align*}
The Hessian operator is obtained from this equation by neglecting the higher-order terms in $\delta\psi$:
\begin{equation}
	\mathcal{H}_{\psi\psi}(\psi,\mu): X\times X \rightarrow Y: (\varphi_1,\varphi_2) \rightarrow 2\left(\overline{\psi}\varphi_1\varphi_2+\psi\overline{\varphi_1}\varphi_2+\psi\varphi_1\overline{\varphi_2}\right).
\end{equation}
Note that this operator is independent of the strength $\mu$ of the applied magnetic field. 
The full Hessian operator is defined by
\begin{align*}
	&\mathcal{H}(\psi,\mu): (X\times\mathbb{R})\times(X\times\mathbb{R})\rightarrow Y: \\
	&\hspace{20pt}\left(\begin{pmatrix}
		\varphi_1 \\
		\kappa_1
	\end{pmatrix}\begin{pmatrix}
		\varphi_2 \\
		\kappa_2
	\end{pmatrix}\right) \rightarrow \mathcal{H}_{\psi\psi}(\psi,\mu)\varphi_1\varphi_2 + \kappa_1\mathcal{H}_{\psi\mu}(\psi,\mu)\varphi_2 + \kappa_2\mathcal{H}_{\psi\mu}(\psi,\mu)\varphi_1 + \kappa_1\kappa_2\mathcal{H}_{\mu\mu}(\psi,\mu)
\end{align*}
and appears in the Lyapunov-Schmidt reduction-based method for constructing the tangent directions. Higher-order derivatives appear in this method as well. Using the same technique as before, one can show that the third partial derivative is given by
\begin{align}
	\frac{\partial^3\mathcal{GL}}{\partial\psi^3}(\psi,\mu):& X\times X \times X \rightarrow Y: \\
	&(\varphi_1,\varphi_2, \varphi_3) \rightarrow 2\left(\overline{\varphi_1}\varphi_2\varphi_3+\varphi_1\overline{\varphi_2}\varphi_3+\varphi_1\varphi_2\overline{\varphi_3}\right). \nonumber
\end{align}
This operator is independent of both the strength $\mu$ of the applied magnetic field and the order parameter $\psi$. These independencies yield
\begin{align*}
	&\forall k\geq 4: \frac{\partial^k\mathcal{GL}}{\partial\psi^k}(\psi,\mu) = 0,
	&& \forall k\geq 1: \frac{\partial^k}{\partial \mu^k}\frac{\partial^2\mathcal{GL}}{\partial\psi^2}(\psi,\mu) = 0.
\end{align*}
Other partial derivatives are of the form
\begin{align*}
	&\frac{\partial^k\mathcal{GL}}{\partial\mu^k}(\psi,\mu), 
	&& \frac{\partial^k\mathcal{J}}{\partial \mu^k}(\psi,\mu) && \text{for } k\geq 1.
\end{align*}
These operators will be approximated by applying a second-order finite difference scheme. See \cite{Fornberg1988} for more information on these schemes for general derivatives.

\subsection*{The discretized problem}
In practice the grid and equations need to be discretized for numerical continuation to be efficiently performed. Discretization of the problem was discussed in detail in \cite{Schlomer_square} and \cite{Schlomer_solver}. We will only provide the main results here.

Let $(x_1,y_1),\dots,(x_n,y_n)$ be a set of discretization points of the sample $\Omega$. If possible, this set should be chosen in such a way that any symmetries of $\Omega$ are preserved. States $\psi\in X$ will be approximated by $\psi^{(h)}\in\mathbb{C}^n$, with

$$
\psi^{(h)} = \begin{pmatrix}
\psi_1^{(h)} \\ \vdots \\ \psi_n^{(h)}
\end{pmatrix} = \begin{pmatrix}
\psi(x_1,y_1) \\ \vdots \\ \psi(x_n,y_n)
\end{pmatrix}.
$$

Let $\{T_i\}_{i=1}^m$ be the Delauny triangulation of $(x_1,y_1),\dots,(x_n,y_n)$ and $\{V_j\}_{j=1}^n$ its corresponding Voronoi tessellation. An edge $(x_j,y_j)-(x_k,y_k)$ of a triangle $T_i$ is denoted as $e_{j,k}$.
We will first discretize the operator $(-\i\nabla -\mathbf{A}(\mu))^2$, this discretization is given by $\mathcal{K}^{(h)}(\mu)$, defined by the property \cite{Schlomer_solver}

\begin{align}
	\forall \psi^{(h)},\phi^{(h)} \in\mathbb{C}^n: \sum_{i=1}^{n}|V_i| \overline{\phi_i^{(h)}}\left(K^{(h)}(\mu)\psi^{(h)}\right)_i = \sum_{i=1}^{m} \sum_{\text{edges } e_{j,k} \text{ of } T_i} & \alpha_{j,k}^{(i)}\Big(\big(\psi_j^{(h)}-U_{j,k}(\mu)\psi_k^{(h)}\big)\overline{\phi_j^{(h)}} \nonumber \\
	&+  \big(\psi_k^{(h)}-\overline{U_{j,k}(\mu)}\psi_j^{(h)}\big)\overline{\phi_k^{(h)}}\Big).
\end{align}

For a triangle $T_i$ consisting of the edges $e_{j,k}$, $e_{k,l}$ and $e_{l,j}$, the coefficient $\alpha_{j,k}^{(i)}\in\mathbb{R}$ is given by the formula

$$
\alpha_{j,k}^{(i)} = \frac{1}{2} \frac{t_{j,k}}{\sqrt{1-t_{j,k}^2}} \qquad \text{with } t_{j,k} = \bigg\langle \frac{e_{k,l}}{||e_{k,l}||_2},\frac{e_{l,j}}{||e_{l,j}||_2}\bigg\rangle_2.
$$

The coefficients $\alpha_{k,l}^{(i)}$ and $\alpha_{l,j}^{(i)}$ are defined by a similar formula. The values $U_{j,k}(\mu)\in\mathbb{C}$ are given by

$$
U_{j,k}(\mu) = \exp\left(-\i\int_{(x_k,y_k)}^{(x_j,y_j)}\langle e_{j,k} , \mathbf{A}(\mu,\omega)\rangle_2 d\omega\right)
$$

with $\mathbf{A}(\mu,\omega)$ the magnetic vector potential evaluated at the location $\omega\in\Omega$.
Using the discretization $K^{(h)}$ of $(-\i\nabla -\mathbf{A}(\mu))^2$, we discretize the Ginzburg-Landau equation \eqref{Ginzburg-Landau vergelijking} as \cite{Schlomer_solver}

\begin{align} 
	\begin{aligned} \label{discretized GL}
		&\mathcal{F}:\mathbb{C}^n\times\mathbb{R} \rightarrow \mathbb{C}^n, \\
		&0 = \mathcal{F}(\psi^{(h)},\mu) =
		\begin{cases}
			\left(K^{(h)}(\mu)\psi^{(h)}\right)_1 - \psi_1^{(h)}\left(1-|\psi_1^{(h)}|^2\right) = 0, \\
			\hspace{70pt}\vdots \\
			\left(K^{(h)}(\mu)\psi^{(h)}\right)_n - \psi_n^{(h)}\left(1-|\psi_n^{(h)}|^2\right) = 0.
		\end{cases}
	\end{aligned}
\end{align}

Other derivatives are discretized in a similar way. The discrete Jacobian is self-adjoint with respect to the inner product

\begin{equation}
\langle \cdot, \cdot \rangle_\mathbb{R}^{(h)} = \mathfrak{Re}\langle\cdot,\cdot\rangle_2
\end{equation}

which is the discrete version of \eqref{inproduct}. In the remainder of the article we will derive methods and algorithms for the continuous problem \eqref{Ginzburg-Landau vergelijking}. The extension of the results to the discrete case is straightforward.

\section{Numerical Continuation} \label{sectie numerieke continuatie}
\subsection*{Pseudo-arclength continuation}
To study the influence of system parameters like the magnetic field strength on the behaviour of patterns in superconductors, numerical continuation is performed. The most natural method for this purpose is considered to be pseudo-arclength continuation \cite{Beyn2002,Keller1986}. This algorithm is a predictor-corrector method: from a given solution ($\psi^{(0)},\mu^{(0)}$) a prediction ($\tilde{\psi}^{(1)},\tilde{\mu}^{(1)}$) for the next point of the solution branch is made, which is then corrected by a Newton-Krylov algorithm.
The prediction ($\tilde{\psi}^{(1)},\tilde{\mu}^{(1)}$) is made by perturbing the given solution in the direction of the tangent ($\dot{\psi}^{(0)},\dot{\mu}^{(0)}$) of the solution curve:
\begin{align*}
	&\tilde{\psi}^{(1)} = \psi^{(0)} + \Delta s \dot{\psi}^{(0)}, \\
	&\tilde{\mu}^{(1)} = \mu^{(0)} + \Delta s \dot{\mu}^{(0)},
\end{align*}
with $\Delta s\in\mathbb{R}$ a sufficiently small step size.
In practice an approximation to the tangent ($\dot{\psi}^{(0)},\dot{\mu}^{(0)}$) is used, which speeds up the algorithm.
After its construction the (approximated) tangent is orthogonalized towards $\begin{pmatrix}
\i\psi^{(0)} &0
\end{pmatrix}^T$, the null vector of the (full) Jacobian induced by the continuous $S^1$ symmetry. Without this orthogonalization numerical continuation might eventually fail due to solution branches falsely reverting. The correction step now consists of solving the system
\begin{equation} \label{Newton stelsel}
	\begin{cases}
		&\mathcal{GL}(\psi^{(1)},\mu^{(1)}) = 0, \\
		& P(\psi^{(1)},\mu^{(1)}) = \langle\dot{\psi}^{(0)},\psi^{(1)}-\tilde{\psi}^{(1)}\rangle_\mathbb{R} + \langle\dot{\mu}^{(0)},\mu^{(1)}-\tilde{\mu}^{(1)}\rangle_\mathbb{R} =0
	\end{cases}
\end{equation}
for the unknowns $\psi^{(1)}$ and $\mu^{(1)}$ by a Newton-Krylov algorithm. Note that perpendicularity with respect to the inner product \eqref{inproduct} is used.

\subsection*{A solver for the Jacobian system}
Application of a Newton-Krylov algorithm to \eqref{Newton stelsel} requires a solver for linear systems of the form
\begin{equation} \label{Jacobian system Newton}
	\begin{pmatrix}
		\mathcal{J}_\psi(\psi^{(1)},\mu^{(1)}) & \mathcal{J}_\mu(\psi^{(1)},\mu^{(1)}) \\
		\langle\dot{\psi}^{(0)},\cdot\rangle_\mathbb{R} & \langle\dot{\mu}^{(0)},\cdot\rangle_\mathbb{R}
	\end{pmatrix} \begin{pmatrix}
		\Delta \psi \\
		\Delta \mu
	\end{pmatrix} = \begin{pmatrix}
		-\mathcal{GL}(\psi^{(1)},\mu^{(1)}) \\
		-P(\psi^{(1)},\mu^{(1)})
	\end{pmatrix}.
\end{equation}
Since we want to trace solution branches for both physically stable and unstable solutions, the solver constructed in \cite{Schlomer_solver} will be used as a base for solving these linear systems. A summary of the construction and properties of the solver is given below, based on the analysis done in \cite{Schlomer_solver}. \\

In \cite{Schlomer_solver} the problem 
\begin{equation} \label{vergelijking opgelost in pynosh}
	\mathcal{J}_\psi(\psi,\mu)x = y \quad \text{ with } x\in X, y\in Y.
\end{equation}
is considered: a solver is constructed for linear systems with the partial Jacobian operator to $\psi$. Though this operator is singular in solutions of \eqref{Ginzburg-Landau vergelijking} and ill-conditioned close to such solutions, Krylov methods can be applied given that the right-hand side $y$ lies in the range of the operator $\mathcal{J}_\psi(\psi,\mu)$. The sparsity structure of $\mathcal{J}_\psi(\psi,\mu)$ promotes the use of a Krylov method as well, since only the application of an operator to given vectors needs to be known. Since the Jacobian $\mathcal{J}_\psi(\psi,\mu)$ is self-adjoint with respect to the inner product \eqref{inproduct}, but is generally indefinite, the Krylov method MINRES (adapted for the inner product \eqref{inproduct}) is chosen.

Preconditioning is also discussed in \cite{Schlomer_solver}: approximate inverses of the operator
\begin{displaymath}
	\mathcal{R}(\psi,\mu): X\rightarrow Y: \varphi \rightarrow (-\i\nabla - \mathbf{A})^2\varphi + 2|\psi|^2\varphi
\end{displaymath}
are used as a preconditioner. Approximate inversion is realized by an algebraic multigrid (AMG) strategy. Numerical evidence of \cite{Schlomer_solver} shows that the approximate inversion with a single V-cycle yields the fastest solver: in this case the number of Krylov iterations is independent of the dimension of the solution space. \\

An implementation of the complete solver is provided in Python (package PyNosh, see \cite{pynosh}), and forms an important part of our automatic exploration algorithm.
Though this solver is constructed for systems of the form \eqref{vergelijking opgelost in pynosh}, it will be combined with block elimination techniques for application to the bordered linear system \eqref{Jacobian system Newton}. For more information about these techniques, see e.g. \cite{Chan1986,Govaerts1991}.

\section{Detection and determination of bifurcation points} \label{sectie bifurcaties}
Bifurcation points of \eqref{Ginzburg-Landau vergelijking} are points $(\psi_b,\mu_b)$ for which $\mathcal{J}_\psi(\psi_b,\mu_b)$ has a zero eigenvalue \cite{Beyn2002,Keller1986} (we ignore zero eigenvalues induced by the $S^1$ symmetry in this definition). Both intersections of solution branches (branch points) and changes in solutions physical stability always correspond with a bifurcation, making their determination an essential part of our algorithm. In this section a condition that indicates the proximity of bifurcations is described first, then a non-linear system of equations is constructed to determine the points.

\subsection*{Detection}
We use a method based on the analysis done in \cite{Mei2000} to detect bifurcation points: The lowest magnitude Ritz values of $\mathcal{J}_\psi(\psi,\mu)$ are calculated after each numerical continuation step. When a point $(\psi,\mu)$ is close to a bifurcation, one of these Ritz values will approximate zero. Furthermore, if the eigenvalue associated with the bifurcation changes its sign after passing this point, the same happens with the corresponding Ritz value. These properties give rise to the following definition:

\begin{definition} \textbf{\textup{Near bifurcation condition}} \\ \label{def - nearbif}
	Let $\zeta_1^{(0)},\zeta_2^{(0)},\dots, \zeta_k^{(0)}$ and $\zeta_1^{(1)},\zeta_2^{(1)},\dots, \zeta_k^{(1)}$ be the $k$ lowest magnitude Ritz values of the partial Jacobian operator in respectively $(\psi^{(0)},\mu^{(0)})$ and $(\psi^{(1)},\mu^{(1)})$, two consecutive points of the solution branch calculated by pseudo-arclength continuation. Assume the Ritz values are ordered according to the same eigenvectors. A bifurcation point is close to $(\psi^{(1)},\mu^{(1)})$ if either
	\begin{align*}
		&\exists j\in\{1,\dots, k\}: |\zeta_j^{(1)}|< \epsilon_1 \text{ for a certain treshold } \epsilon_1\in\mathbb{R}^+_0,\epsilon_1\ll 1, \\
		&\exists j\in\{1,\dots, k\}: \begin{cases}
			|\zeta_j^{(1)}|< \epsilon_2 \text{ for a certain treshold } \epsilon_2\in\mathbb{R}^+_0,\epsilon_1\ll \epsilon_2\ll 1, \\
			\zeta_j^{(0)}\zeta_j^{(1)}<0.
		\end{cases}
	\end{align*}
	In this case we say that the point $(\psi^{(1)},\mu^{(1)})$ satisfies the near bifurcation condition.
\end{definition}

For the calculation of the examples in section \ref{sectie resulaten} we used $k=5$ Ritz values. Note that values induced by the $S^1$ symmetry are ignored in definition \ref{def - nearbif}.

Except for the detection process, Ritz values are also used when determining physical stability: if the Jacobian in a point has a (high) positive Ritz value, this point is assumed to be unstable. Another use of Ritz values and Ritz vectors is deflation: multiplying the system with a certain projection (chosen to eliminate eigenvectors corresponding to low magnitude eigenvalues) accelerates Krylov algorithms \cite{Gaul2011,Gaul2013,Gaul2015}.

\subsection*{Determination}
Though the near bifurcation condition described in the previous section gives us an indication of the proximity of bifurcations, we still need to determine these points more precisely. For this purpose we construct a non-linear system of equations that's only satisfied for bifurcation points $(\psi_b,\mu_b)$ of \eqref{Ginzburg-Landau vergelijking}, this system is based on the property
\begin{align*}
	&(\psi_b,\mu_b) \text{ is a bifurcation point} \iff \exists \phi\in X \text{ } (\phi \neq \i\psi_b, \phi \neq 0): \mathcal{J}_\psi(\psi_b,\mu_b)\phi = 0.
\end{align*}
The system we consider is given by 
\begin{equation} \label{bifurcatiepunt stelsel}
	\begin{cases}
		\mathcal{GL}(\psi_b,\mu_b) = 0, \\
		\mathcal{J}_\psi(\psi_b,\mu_b)\phi = 0, \\
		\langle\i\psi_b,\phi\rangle_\mathbb{R} = 0,
	\end{cases}
\end{equation}
and is similar to the one described in \cite{Mei2000}. When we find a point that satisfies the near bifurcation condition, we use it as an initial guess to solve \eqref{bifurcatiepunt stelsel} for $(\psi_b,\mu_b,\phi)$ by a Newton-Krylov algorithm. The approximate eigenvector corresponding to the near-zero Ritz value is used as an initial guess for the null vector $\phi$. Note that the guess for $\phi$ should be normalized after each Newton step, to prevent the solver from converging to the trivial solution $\phi=0$.
In each Newton step the Jacobian system of \eqref{bifurcatiepunt stelsel} has to be solved, this system is given by
\begin{equation} \label{Jacobian system bifurcatiepunt}
	\begin{pmatrix}
		\mathcal{J}_\psi(\psi_b,\mu_b) & 0 & \mathcal{J}_\mu(\psi_b,\mu_b) \\
		\mathcal{H}_{\psi\psi}(\psi_b,\mu_b)\phi & \mathcal{J}_\psi(\psi_b,\mu_b) & \mathcal{H}_{\psi\mu}(\psi_b,\mu_b)\phi \\
		\langle\i\phi,\cdot\rangle_\mathbb{R} & \langle\i\psi_b,\cdot\rangle_\mathbb{R} & 0
	\end{pmatrix} \begin{pmatrix}
		\Delta \psi \\
		\Delta \phi \\
		\Delta \mu
	\end{pmatrix} = \begin{pmatrix}
		-\mathcal{GL}(\psi_b,\mu_b) \\
		-\mathcal{J}_\psi(\psi_b,\mu_b)\phi \\
		-\langle\i\psi_b,\phi\rangle_\mathbb{R}
	\end{pmatrix}.
\end{equation}
Together with the application of block elimination techniques \cite{Chan1986,Govaerts1991} and deflation \cite{Gaul2011,Gaul2013,Gaul2015}, this equation is solved by employing the solver described in section \ref{sectie numerieke continuatie}. Note that both (partial) Jacobians and Hessians of \eqref{Ginzburg-Landau vergelijking} appear in \eqref{Jacobian system bifurcatiepunt}. As was shown in section \ref{sectie Ginzburg-Landau systeem}, these expressions are either known, or are approximated by a finite difference scheme.

\section{Construction of new tangent directions} \label{sectie tangents}
In order to calculate new solution curves that emerge from a branch point, it is sufficient to determine the tangent directions to these curves. Given these directions, the curves are calculated by employing the pseudo-arclength continuation algorithm (see section \ref{sectie numerieke continuatie}). We will derive two algorithms for the construction of the tangent directions, one based on Lyapunov-Schmidt reduction, another based on the equivariant branching lemma. In the remainder of this section ($\psi_b,\mu_b$) is assumed to be a bifurcation point of \eqref{Ginzburg-Landau vergelijking} and we denote $\mathcal{GL}^{(b)}$, $\mathcal{J}^{(b)}$, \dots as the evaluations of $\mathcal{GL}$, $\mathcal{J}$, \dots at ($\psi_b,\mu_b$).

\subsection*{Lyapunov-Schmidt reduction}

A popular approach for analysing bifurcations is the Lyapunov-Schmidt method: in a neighbourhood of the bifurcation ($\psi_b,\mu_b$), the problem is reduced to low dimensional systems of algebraic equations that contain all of the information of the bifurcations behaviour. The solutions of these systems are then used to construct the tangent directions. In \cite{Mei1996} the method is applied to general problems, leading to three equations that form the base of our first tangent construction algorithm. A short overview of the deduction of these equations (applied to the Ginzburg-Landau problem) is given below, based on the analysis done in \cite{Mei1996}. \\

The following assumptions are made:
\begin{itemize}
	\item[-] $\mathcal{J}_\psi^{(b)}$ is a Fredholm operator of index $0$ and zero is a semi-simple eigenvalue of it.
	\item[-] Orthonormal bases for the kernels of $\mathcal{J}_\psi^{(b)}$ and $\mathcal{J}_\psi^{(b)^*}$ (the adjoint Jacobian with respect to \eqref{inproduct}, see \cite{Schlomer_square} for more details) have been determined:
	\begin{align*}
		&\ker\left(\mathcal{J}_\psi^{(b)}\right) = \text{span}\left(\phi_1,\phi_2,\dots,\phi_n\right), \\
		&\ker\left(\mathcal{J}_\psi^{(b)^*}\right) = \text{span}\left(\phi_1^*,\phi_2^*, \dots, \phi_n^*\right)
	\end{align*}
	for a certain $n\in\mathbb{N}$. Note that the trivial null vector $\i\psi_b$ induced by the $S^1$ symmetry is ignored in this analysis and does not count towards the dimension of these kernels. In the case of the Ginzburg-Landau equation the Jacobian $\mathcal{J}_\psi^{(b)}$ is self-adjoint, so we actually have $\ker\left(\mathcal{J}_\psi^{(b)}\right) = \ker\left(\mathcal{J}_\psi^{(b)^*}\right)$.
	\item[-] We have
	\begin{displaymath}
		\mathcal{J}_\mu^{(b)} \in \text{im}\left(\mathcal{J}_\psi^{(b)}\right).
	\end{displaymath}
	Note that if this assumption does not hold and $n$ equals $1$, the bifurcation $(\psi_b,\mu_b)$ is in fact a turning point and no new tangent directions need to be calculated.
\end{itemize}
In practice, the kernels of $\mathcal{J}_\psi^{(b)}$ and $\mathcal{J}_\psi^{(b)^*}$ are approximated using numerical methods (e.g. an inverse iteration algorithm \cite{Demmel1997}).

A solution curve ($\psi(s),\mu(s)$) with $(\psi(0),\mu(0))=(\psi_b,\mu_b)$ is considered, where $s$ is the parametrization parameter. The fundamental theorem of linear algebra states that the space $X$ can be decomposed into $\ker\left(\mathcal{J}_\psi^{(b)}\right)$ and $\text{im}\left(\mathcal{J}_\psi^{(b)^*}\right)$, this implies that we can write $\psi(s)$ and $\mu(s)$ as
\begin{align*}
	&\psi(s) = \psi_b + \sum_{i=1}^{n}s\alpha_i\phi_i +sw(\alpha_1,\dots,\alpha_n,s) &&\text{with } w\in\text{im}\left(\mathcal{J}_\psi^{(b)^*}\right), \\
	&\mu(s) = \mu_b + s\beta(s) &&\text{with } \beta\in\mathbb{R}.
\end{align*}
Applying a Taylor expansion, this yields
\begin{align}
	&\psi(s) = \psi_b + \sum_{i=1}^{n}s\alpha_i\phi_i + \sum_{k=1}^{l}s^kw_k + \mathcal{O}(s^{l+1}), \nonumber \\
	&\mu(s) = \mu_b + \sum_{k=1}^{l}s^k\beta_k + \mathcal{O}(s^{l+1}) \label{Taylor expansie psi en mu} \\
	&\text{with } l\in\mathbb{N} \text{, } w_1,\dots, w_l\in\text{im}\left(\mathcal{J}_\psi^{(b)^*}\right) \text{ and } \beta_1,\dots,\beta_{l}\in\mathbb{R}. \nonumber
\end{align}
Note that the tangent ($\dot{\psi}(0),\dot{\mu}(0)$) to the solution curve is given by
\begin{align}
	&\dot{\psi}(0) = \lim\limits_{s\rightarrow0}\frac{\psi(s)-\psi_b}{s} = \sum_{i=1}^{n}\alpha_i\phi_i+w_1 \label{vorm tangents}, \\
	&\dot{\mu}(0) = \lim\limits_{s\rightarrow0}\frac{\mu(s)-\mu_b}{s} = \beta_1. \nonumber
\end{align}
Equation \eqref{Ginzburg-Landau vergelijking} is rewritten (using again a Taylor expansion),
\begin{align}
	\mathcal{GL}(\psi,\mu) =& \mathcal{J}^{(b)}\begin{pmatrix}
		\psi-\psi_b \\
		\mu-\mu_b
	\end{pmatrix} + R(\psi-\psi_b,\mu-\mu_b) \nonumber \\
	=&\sum_{k=1}^{l} s^k\left(\mathcal{J}_\psi^{(b)}w_k+\mathcal{J}_\mu^{(b)}\beta_k+r_k\right) + \mathcal{O}(s^{l+1}), \label{Taylorexpansie GL}
\end{align}
with $R(\psi-\psi_b,\mu-\mu_b)$ given by
\begin{align*}
	&R(\psi-\psi_b,\mu-\mu_b) \\
	&\hspace{10pt}= \sum_{m=2}^{\infty} \frac{1}{m!} D^m\mathcal{GL}^{(b)}
	\begin{pmatrix}
		\sum\limits_{i=1}^{n}s\alpha_i\phi_i+\sum\limits_{k=1}^{l}s^kw_k+\mathcal{O}(s^{l+1}) \\
		\sum\limits_{k=1}^{l}s^k\beta_k+\mathcal{O}(s^{l+1})
	\end{pmatrix}^m.
\end{align*}
The term $r_k$  in \eqref{Taylorexpansie GL} represents the coefficient of $s^k$ in $R(\psi-\psi_b,\mu-\mu_b)$, this is explicitly given by
\begin{align}
	&r_k = \sum_{j=2}^{k}\frac{1}{j!}\hspace{20pt}\sum_{\mathclap{\substack{\sum_{p=1}^{j}k_p=k \\ \forall p=1,\dots, j: \\ k_p\in\{1,\dots, k-j+1\}}}}\hspace{10pt}D^j\mathcal{GL}^{(b)}
	\begin{pmatrix}
		x_{k_1} \\
		\beta_{k_1}
	\end{pmatrix}\begin{pmatrix}
		x_{k_2} \\
		\beta_{k_2}
	\end{pmatrix}\dots\begin{pmatrix}
		x_{k_j} \\
		\beta_{k_j}
	\end{pmatrix} \label{exact expression rk} \\
	&\text{with } x_1=\sum_{i=1}^{n}\alpha_i\phi_i+w_1 \qquad \forall p=2,\dots,k-1: x_p=w_p. \nonumber
\end{align}

\noindent The following three equations are now derived in \cite{Mei1996} (for $k=1,\dots, l$):
\begin{align}
	&\mathcal{J}_\psi^{(b)}w_k+\mathcal{J}_\mu^{(b)}\beta_k = -r_k, \label{vergelijkingen van paper Mei 1} \\
	&\forall j=1,\dots, n: \langle\phi_j,w_k\rangle_\mathbb{R}=0, \label{vergelijkingen van paper Mei 2} \\
	&\forall j=1,\dots, n: \langle\phi_j^*,r_k\rangle_\mathbb{R}=0, \label{vergelijkingen van paper Mei 3}
\end{align}
these equations will form the basis of the upcoming algorithms.
For more details on Lyapunov-Schmidt reduction, see e.g. \cite{Mei2000,Mei1996}.

\subsection*{An algorithm for the construction of the tangent directions}
In \cite{Mei1996} a general algorithm is deduced to construct reduced systems of equations for $\alpha_1,\dots,\alpha_n, \beta_{k-1}$. Together with \eqref{vorm tangents} the solutions of these systems yield the tangent directions. For this method equations \eqref{vergelijkingen van paper Mei 1}, \eqref{vergelijkingen van paper Mei 2} and \eqref{vergelijkingen van paper Mei 3} are used to write the terms $r_k$ and $w_k$ as polynomials in the variables $\alpha_1,\dots,\alpha_n,\beta_1,\dots,\beta_{k-1}$. Under the assumptions
\begin{align}
	&t_1 = \frac{\partial}{\partial\psi}\mathcal{J}^{(b)}
	\begin{pmatrix}
		v^{(0)} \\
		1
	\end{pmatrix}
	\phi_1\notin\text{im}\left(\mathcal{J}_\psi^{(b)}\right), \label{extra aanname in LS reductie 1} \\
	&t_2 = \frac{\partial}{\partial\psi}\mathcal{J}^{(b)}
	\begin{pmatrix}
		v^{(0)} \\
		1
	\end{pmatrix}
	\phi_2\notin\text{im}\left(\mathcal{J}_\psi^{(b)}\right), \label{extra aanname in LS reductie 1.5} \\
	&\langle t_1,\phi_1^*\rangle_\mathbb{R}
	\langle t_2,\phi_2^*\rangle_\mathbb{R} \neq \langle t_1,\phi_2^*\rangle_\mathbb{R}
	\langle t_2,\phi_1^*\rangle_\mathbb{R}, \label{extra aanname in LS reductie 1.8} \\
	&t_3 = \mathcal{H}^{(b)}\begin{pmatrix}
		v^{(0)} \\
		1
	\end{pmatrix}\begin{pmatrix}
		v^{(0)} \\
		1
	\end{pmatrix}\in\text{im}\left(\mathcal{J}_\psi^{(b)}\right), \label{extra aanname in LS reductie 2}
\end{align}
it is possible to write the term $\beta_k$ as a polynomial in the variables $\alpha_1,\dots,\alpha_n$ as well. The vector $v^{(0)}$ is implicitly defined via $w_1=\beta_1 v^{(0)}$. The method described in \cite{Mei1996} is given in algorithm \ref{algorithm Mei} (rewritten to emphasize the construction of the reduced systems).

\begin{textalgorithm} \textbf{\textup{Creation of tangent directions based on \cite{Mei1996}}} \\ \label{algorithm Mei}
	\textbf{Initial:} \\
	For $k=1$, use the equation
	\begin{displaymath}
		\mathcal{J}_\psi^{(b)}w_1 = -\mathcal{J}_\mu^{(b)}\beta_1
	\end{displaymath}
	to write $w_1\in\text{im}\left(\mathcal{J}_\psi^{(b)^*}\right)$ as a polynomial in $\beta_1$. \\
	\textbf{Iteration:} \\
	For $k=2,3,\dots$ do: \\
	\textnormal{Step 1:} Use the polynomial expression of $w_{k-1}$ (and $\beta_{k-2}$ if $k\geq 3$) together with \eqref{exact expression rk} to write $r_k$ as a polynomial in $\alpha_1,\dots,\alpha_n,\beta_{k-1}$. \\
	\textnormal{Step 2:} Substitute this polynomial expression in the equation
	\begin{displaymath}
		\forall j=1,\dots, n: \langle\phi_j^*,r_k\rangle_\mathbb{R}=0.
	\end{displaymath}
	This yields a reduced system of equations for $\alpha_1,\dots,\alpha_n,\beta_{k-1}$. \\
	\textnormal{Step 3:} Check whether the system has real and isolated solutions $\alpha_1(\beta_{k-1}),\dots, \alpha_n(\beta_{k-1})$ for $\beta_{k-1}$ in an open interval. Solve the system for these solutions, each one corresponds to a different tangent direction. \\
	\textnormal{Step 4a:} If only isolated solutions exist in the reduced system, stop the algorithm. \\
	\textnormal{Step 4b:} If non-isolated solutions exist as well, use the system to write $\beta_{k-1}$ as a polynomial of $\alpha_1,\dots, \alpha_n$ and use this result to write $r_k$ as a polynomial in these variables as well. Substitute this expression for $r_k$ into the system
	\begin{align*}
		&\mathcal{J}_\psi^{(b)}w_k+\mathcal{J}_\mu^{(b)}\beta_k = -r_k, \\
		&\forall j=1,\dots, n: \langle\phi_j,w_k\rangle_\mathbb{R}=0
	\end{align*}
	and solve it to find a polynomial expression for $w_k$ in $\alpha_1,\dots,\alpha_n,\beta_k$.
\end{textalgorithm}
The complexity of the polynomial expressions in algorithm \ref{algorithm Mei} strongly depends on the dimension $n$ of $\ker\left(\mathcal{J}_\psi^{(b)}\right)$. If this kernel is one-dimensional, the algorithm only requires a single iteration: the first reduced system that is constructed does not contain non-isolated solutions. In this case the algorithm is actually equivalent to the algebraic branching equation \cite{Beyn2002,Keller1986}, the reduced system for $\alpha_1$ and $\beta_1$ is given by
\begin{align}
	&a\alpha_1^2 + b\alpha_1\beta_1+c\beta_1^2 = 0 \label{Algebraic branching equation} \\
	&\text{with } a = \langle\phi_1^*,\frac{1}{2}\mathcal{H}_{\psi\psi}^{(b)}\phi_1\phi_1\rangle_\mathbb{R},  \nonumber \\
	&\hspace{22pt}b = \langle\phi_1^*,\mathcal{H}_{\psi\psi}^{(b)}\phi_1v^{(0)}+\mathcal{H}_{\psi\mu}^{(b)}\phi_1\rangle_\mathbb{R}, \nonumber \\
	&\hspace{22pt}c = \langle\phi_1^*,\frac{1}{2}\mathcal{H}_{\psi\psi}^{(b)}v^{(0)}v^{(0)}+\mathcal{H}_{\psi\mu}^{(b)}v^{(0)} + \frac{1}{2}\mathcal{H}_{\mu\mu}^{(b)}\rangle_\mathbb{R}, \nonumber
\end{align}
where $v^{(0)}$ is solved from the equation
\begin{displaymath}
	\mathcal{J}_\psi^{(b)}v^{(0)} = -\mathcal{J}_\mu^{(b)} \qquad v^{(0)}\in\text{im}\left(\mathcal{J}_\psi^{(b)^*}\right).
\end{displaymath}
Equation \eqref{Algebraic branching equation} is solved for $\alpha_1$ and $\beta_1$, together with $w_1=\beta_1v^{(0)}$ the tangent directions are constructed by applying \eqref{vorm tangents}. The directions are determined up to multiplication by a constant, and should be normalized before further use.

For problems that exhibit symmetry, bifurcations arise with Jacobian kernel dimension $n\geq 2$ and the polynomials in \ref{algorithm Mei} become more complex. In \cite{Mei1996} no exact expressions for these polynomials were derived and hence only implicit reduced systems of equations are given, which limits the use of the algorithm in its current form by numerical software. Furthermore, it is not specified how these systems can be checked for real and isolated solutions. In the remainder of this section we derive exact polynomial expressions that yield explicit systems for the case $n=2$. A simple rule is used to check these for real and isolated solutions. No other cases than $n=1$ and $n=2$ were encountered for the Ginzburg-Landau equation on two-dimensional samples.
Set $n=2$, assumptions \eqref{extra aanname in LS reductie 1},\eqref{extra aanname in LS reductie 1.5} and \eqref{extra aanname in LS reductie 1.8} allow us to choose $\phi_1^*$ and $\phi_2^*$ such that:
\begin{align*}
	&\frac{\partial}{\partial\psi}\mathcal{J}^{(b)}\begin{pmatrix}
		v^{(0)} \\
		1
	\end{pmatrix}\phi_1 = q_1\phi_1^*+u_1 &&\text{with } u_1\in\text{im}\left(\mathcal{J}_\psi^{(b)}\right),q_1\in\mathbb{R}_0, \\
	&\frac{\partial}{\partial\psi}\mathcal{J}^{(b)}\begin{pmatrix}
		v^{(0)} \\
		1
	\end{pmatrix}\phi_2 = q_2\phi_1^*+q_3\phi_2^*+u_2 &&\text{with } u_2\in\text{im}\left(\mathcal{J}_\psi^{(b)}\right),q_2\in\mathbb{R},q_3\in\mathbb{R}_0.
\end{align*}

A modification of algorithm \ref{algorithm Mei} for the case $n=2$ is given by algorithm \ref{algoritme n=2}. In this last algorithm the coefficients necessary to write the different polynomial expressions are derived in order to determine explicit reduced systems of equations for $\alpha_1,\dots,\alpha_n,\beta_{k-1}$. A thorough derivation of algorithm \ref{algoritme n=2} is provided in appendix \ref{app proof}.
\begin{textalgorithm} \textbf{\textup{Creation of tangent directions for the $n=2$ case}} \\ \label{algoritme n=2}
	\textbf{Initial:} \\
	\textnormal{Step 1:} Solve the equation
	\begin{displaymath}
		\mathcal{J}_\psi^{(b)}v^{(0)} = -\mathcal{J}_\mu^{(b)} \qquad v^{(0)}\in\text{im}\left(\mathcal{J}_\psi^{(b)^*}\right) \subset X
	\end{displaymath}
	by applying the solver described in section \ref{sectie numerieke continuatie}. Note that $w_1$ (defined in \eqref{Taylor expansie psi en mu}) can be written as
	\begin{equation} \label{polynoom w1}
		w_1=\beta_1v^{(0)}.
	\end{equation}
	\textnormal{Step 2:} Calculate the following terms in the space $Y$:
	\begin{align*}
		&y_0^{(2)}=\frac{1}{2}\mathcal{H}_{\psi\psi}^{(b)}\phi_2\phi_2, \qquad y_1^{(2)}=\mathcal{H}_{\psi\psi}^{(b)}\phi_1\phi_2, \qquad y_2^{(2)}=\frac{1}{2}\mathcal{H}_{\psi\psi}^{(b)}\phi_1\phi_1, \\
		&t_1 =\mathcal{H}_{\psi\psi}^{(b)}\phi_1v^{(0)}+\mathcal{H}_{\psi\mu}^{(b)}\phi_1, \qquad t_2 =\mathcal{H}_{\psi\psi}^{(b)}\phi_2v^{(0)}+\mathcal{H}_{\psi\mu}^{(b)}\phi_2, \\
		&t_3 = \frac{1}{2}\mathcal{H}_{\psi\psi}^{(b)}v^{(0)}v^{(0)}+\mathcal{H}_{\psi\mu}^{(b)}v^{(0)} + \frac{1}{2}\mathcal{H}_{\mu\mu}^{(b)}.
	\end{align*}
	Note that $r_2$ (defined in \eqref{exact expression rk}) can be written as
	\begin{equation}
		r_2 = \sum_{i=0}^{2}\alpha_1^i\alpha_2^{2-i}y_i^{(2)} + \beta_1\alpha_1 t_1 + \beta_1\alpha_2 t_2 + \beta_1^2t_3.
	\end{equation}
	\textnormal{Step 3:} Calculate the scalars
	\begin{align*}
		&\forall l=1;2, i=0;1;2: a_i^{(2,l)} = \langle\phi_l^*,y_i^{(2)}\rangle_\mathbb{R}, \\
		&b^{(1)} = \langle\phi_1^*,t_1\rangle_\mathbb{R}, \qquad b^{(2)} = \langle\phi_2^*,t_2\rangle_\mathbb{R}, \qquad b^{(3)} = \langle\phi_1^*,t_2\rangle_\mathbb{R}.
	\end{align*}
	These values form the coefficients of the following reduced system of equations:
	\begin{equation} \label{reduced equation k=2}
		\begin{cases}
			&\sum\limits_{i=0}^{2}\alpha_1^i\alpha_2^{2-i}a_i^{(2,1)}+\beta_1\alpha_1b^{(1)}+\beta_1\alpha_2 b^{(3)}=0, \\
			&\sum\limits_{i=0}^{2}\alpha_1^i\alpha_2^{2-i}a_i^{(2,2)}+\beta_1\alpha_2b^{(2)}=0.
		\end{cases}
	\end{equation}
	\textnormal{Step 4:} Check whether $b^{(2)}a_0^{(2,1)}=b^{(3)}a_0^{(2,2)}$, $a_2^{(2,2)}=0$ and $\forall i=1;2:b^{(2)}a_i^{(2,1)}=b^{(1)}a_{i-1}^{(2,2)}+b^{(3)}a_{i}^{(2,2)}$. \\
	\textnormal{Step 5a:} If this is the case, stop the algorithm: system \eqref{reduced equation k=2} only contains isolated solutions, these correspond to the tangent directions. \\
	\textnormal{Step 5b:} If this is not the case, the only isolated solution of \eqref{reduced equation k=2} is $(\alpha_1,\alpha_2,\beta_1)=(0,0,a)$ with $a\in\mathbb{R}$. The system contains non-isolated solutions as well. Set $\kappa^{(1)}_{-1}=\kappa^{(1)}_{2}=0$ and calculate
	\begin{align*}
		&\forall i=0;1: \kappa_i^{(1)} = -\frac{a_i^{(2,2)}}{b^{(2)}}, \\
		&z_0^{(2)} = y_0^{(2)} + \kappa_0^{(1)}t_2 + \kappa_0^{(1)^2}t_3 \qquad z_1^{(2)} = y_1^{(2)} + \kappa_0^{(1)}t_1 + \kappa_1^{(1)}t_2 + 2\kappa_0^{(1)}\kappa_1^{(1)}t_3, \\
		&z_2^{(2)} = y_2^{(2)} + \kappa_1^{(1)}t_1 + \kappa_1^{(1)^2}t_3, \\
		&q_0^{(1)} = \kappa_0^{(1)}v^{(0)}+\phi_2, \qquad q_1^{(1)} = \kappa_1^{(1)}v^{(0)}+\phi_1.
	\end{align*}
	Note that $\beta_1$, $r_2$ and $x_1$ (defined in \eqref{Taylor expansie psi en mu} and \eqref{exact expression rk}) can respectively be written as
	\begin{align} 
		&\beta_1 = \sum_{i=0}^{1}\alpha_1^i\alpha_2^{1-i}\kappa_i^{(1)}, \label{polynoom beta1} \\
		&r_2 = \sum_{i=0}^{2}\alpha_1^i\alpha_2^{2-i}z_i^{(2)}, \\
		&x_1 = \sum_{i=0}^{1}\alpha_1^i\alpha_2^{1-i}q_i^{(1)}.
	\end{align}
	\textbf{Iteration:} \\
	For $k=3,4,\dots$ do: \\
	\textnormal{Step 1:} Solve the equations
	\begin{displaymath}
		\forall i=0,\dots, k-1: \mathcal{J}_\psi^{(b)}v_i^{(k-1)} = -z_i^{(k-1)} \qquad v^{(k-1)}_i\in\text{im}\left(\mathcal{J}_\psi^{(b)^*}\right) \subset X
	\end{displaymath}
	by applying the solver described in section \ref{sectie numerieke continuatie}. Note that $w_{k-1}$ (defined in \eqref{Taylor expansie psi en mu}) can be written as
	\begin{equation}
		w_{k-1} = \sum_{i=0}^{k-1}\alpha_1^i\alpha_2^{k-1-i}v_i^{(k-1)}+\beta_{k-1}v^{(0)}.
	\end{equation}
	\textnormal{Step 2:} Calculate the following terms in the space $Y$:
	\begin{align*}
		&\forall i=0,\dots, k: y_i^{(k)} = \sum_{\mathclap{\substack{i_1+i_2=i \\ i_1\in\{0;1\} \\ i_2\in\{0,\dots, k-1\}}}} \mathcal{H}^{(b)}
		\begin{pmatrix}
			q^{(1)}_{i_1} \\
			\kappa^{(1)}_{i_1}
		\end{pmatrix}\begin{pmatrix}
			v^{(k-1)}_{i_2} \\
			0
		\end{pmatrix} +\frac{1}{2} \hspace{20pt} \sum_{\mathclap{\substack{k_1+k_2=k \\ k_1,k_2\in\{2,\dots, k-2\}}}} \hspace{40pt} \sum_{\mathclap{\substack{i_1+i_2=i \\ i_1\in\{0,\dots, k_1\} \\ i_2\in\{0,\dots, k_2\}}}} \mathcal{H}^{(b)}
		\begin{pmatrix}
			q^{(k_1)}_{i_1} \\
			\kappa^{(k_1)}_{i_1}
		\end{pmatrix}\begin{pmatrix}
			q^{(k_2)}_{i_2} \\
			\kappa^{(k_2)}_{i_2}
		\end{pmatrix} \\
		&\hspace{90pt}+\sum_{j=3}^{k}\frac{1}{j!} \hspace{20pt} \sum_{\mathclap{\substack{\sum_{p=1}^{j}k_p=k  \\ \forall p=1,\dots, j: \\ k_p\in\{1,\dots, k-j+1\}}}} \hspace{40pt} \sum_{\mathclap{\substack{\sum_{p=1}^{j}i_p=i  \\ \forall p=1\dots j: \\ i_p\in\{0,\dots, k_p\}}}} D^j\mathcal{GL}^{(b)}
		\begin{pmatrix}
			q^{(k_1)}_{i_1} \\
			\kappa^{(k_1)}_{i_1}
		\end{pmatrix}\begin{pmatrix}
			q^{(k_2)}_{i_2} \\
			\kappa^{(k_2)}_{i_2}
		\end{pmatrix}\dots\begin{pmatrix}
			q^{(k_j)}_{i_j} \\
			\kappa^{(k_j)}_{i_j}
		\end{pmatrix}.
	\end{align*}
	An efficient method to calculate these terms is provided in appendix \ref{app calc}. Note that $r_{k}$ (defined in \eqref{exact expression rk}) can be written as
	\begin{equation}
		r_k = \sum_{i=0}^{k}\alpha_1^i\alpha_2^{k-i}y_i^{(k)} + \beta_{k-1}\alpha_1\left(t_1+2\kappa_1^{(1)}t_3\right) + \beta_{k-1}\alpha_2\left(t_2+2\kappa_0^{(1)}t_3\right).
	\end{equation}
	\textnormal{Step 3:} Calculate the scalars
	\begin{displaymath}
		\forall l=1;2, i=0,\dots, k: a_i^{(k,l)} = \langle\phi_l^*,y_i^{(k)}\rangle_\mathbb{R}.
	\end{displaymath}
	These values form the coefficients of the following reduced system of equations:
	\begin{equation} \label{reduced equation k=2 b}
		\begin{cases}
			&\sum\limits_{i=0}^{k}\alpha_1^i\alpha_2^{k-i}a_i^{(k,1)}+\beta_{k-1}\alpha_1b^{(1)}+\beta_{k-1}\alpha_2b^{(3)}=0, \\
			&\sum\limits_{i=0}^{k}\alpha_1^i\alpha_2^{k-i}a_i^{(k,2)}+\beta_{k-1}\alpha_2b^{(2)}=0.
		\end{cases}
	\end{equation}
	\textnormal{Step 4:} Check whether $b^{(2)}a_0^{(k,1)}=b^{(3)}a_0^{(k,2)}$, $a_k^{(k,2)}=0$ and $\forall i=1,\dots, k:b^{(2)}a_i^{(k,1)}=b^{(1)}a_{i-1}^{(k,2)}+b^{(3)}a_{i}^{(k,2)}$. \\
	\textnormal{Step 5a:} If this is the case, stop the algorithm: system \eqref{reduced equation k=2 b} only contains isolated solutions, these correspond to the tangent directions. \\
	\textnormal{Step 5b:} If this is not the case, the only isolated solution of \eqref{reduced equation k=2 b} is $(\alpha_1,\alpha_2,\beta_{k-1})=(0,0,a)$ with $a\in\mathbb{R}$. The system contains non-isolated solutions as well. Set $\kappa^{(k-1)}_{-1}=\kappa^{(k-1)}_{k}=0$ and calculate
	\begin{align*}
		&\forall i=0,\dots, k-1: \kappa_i^{(k-1)} = -\frac{a_i^{(k,2)}}{b^{(2)}}, \\
		&\forall i=0,\dots, k: z_i^{(k)} = y_i^{(k)} + \kappa_{i-1}^{(k-1)}\left(t_1+2\kappa_1^{(1)}t_3\right) + \kappa_{i}^{(k-1)}\left(t_2+2\kappa_0^{(1)}t_3\right), \\
		&\forall i=0,\dots, k-1: q_i^{(k-1)}=v_i^{(k-1)}+\kappa_i^{(k-1)}v^{(0)}.
	\end{align*}
	Note that $\beta_{k-1}$, $r_k$ and $x_{k-1}$ (defined in \eqref{Taylor expansie psi en mu} and \eqref{exact expression rk}) can respectively be written as
	\begin{align} 
		&\beta_{k-1} = \sum_{i=0}^{k-1}\alpha_1^i\alpha_2^{k-1-i}\kappa_i^{(k-1)}, \\
		&r_k = \sum_{i=0}^{k}\alpha_1^i\alpha_2^{k-i}z_i^{(k)}, \\
		&\beta_{k-1} = \sum_{i=0}^{k-1}\alpha_1^i\alpha_2^{k-1-i}q_i^{(k-1)}.
	\end{align}
\end{textalgorithm}

The initial steps of the algorithm always yield a system with isolated solution $(\alpha_1,\alpha_2,\beta_{1})=(0,0,a)$ ($a\in\mathbb{R}$). Application of \eqref{vorm tangents} and \eqref{polynoom w1} leads to a first tangent direction. If this system contains other isolated solutions as well, the corresponding tangent directions are constructed with the same formulas. In the other case the algorithm is continued, possibly yielding further systems with a sole isolated solution $(\alpha_1,\alpha_2,\beta_{k-1})=(0,0,a)$ ($a\in\mathbb{R}$). Relation \eqref{polynoom beta1} results in $\beta_1=0$, consequently \eqref{vorm tangents} equals zero and does not actually correspond to a real tangent direction. Eventually the algorithm will find a system that only contains isolated solutions, the tangent directions are constructed from these by application of \eqref{vorm tangents}, \eqref{polynoom w1} and \eqref{polynoom beta1}.

\subsection*{Solving the reduced system of equations}
Algorithm \ref{algoritme n=2} eventually gives rise to a system of equations of the form
\begin{equation} \label{reduced equation}
	\mathcal{F}(\alpha_1,\alpha_2,\beta_{k-1})=\begin{cases}
		&\sum\limits_{i=0}^{k}\alpha_1^i\alpha_2^{k-i}a_i^{(k,1)}+\beta_{k-1}\alpha_1b^{(1)} + \beta_{k-1}\alpha_2 b^{(3)}=0, \\
		&\sum\limits_{i=0}^{k}\alpha_1^i\alpha_2^{k-i}a_i^{(k,2)}+\beta_{k-1}\alpha_2b^{(2)}=0
	\end{cases}
\end{equation}
with $k\geq 2$ and known coefficients $a_0^{(k,1)},\dots, a_k^{(k,1)}, a_0^{(k,2)},\dots, a_k^{(k,2)}, b^{(1)}, b^{(2)}, b^{(3)}$ such that \eqref{reduced equation} only contains isolated solutions.

To find solutions of \eqref{reduced equation} we fix $\beta_{k-1}$. This does not reduce the amount of solutions that will be found since the tangent directions are determined up to normalization. Though in some cases it is possible to solve \eqref{reduced equation} exactly (for e.g. $k=2$ solving this equation is similar to determining the intersection of two conics, for which multiple algorithms exist), we will approximate its solutions with a Newton algorithm. This requires the Jacobian of \eqref{reduced equation}, given by
\begin{align*}
	&\begin{pmatrix}
		\frac{\partial \mathcal{F}(\alpha_1,\alpha_2,\beta_{k-1})}{\partial \alpha_1} & \frac{\partial \mathcal{F}(\alpha_1,\alpha_2,\beta_{k-1})}{\partial \alpha_2}
	\end{pmatrix} \\
	&\hspace{100pt}= \begin{pmatrix}
		\sum\limits_{i=1}\limits^{k}i\alpha_1^{i-1}\alpha_2^{k-i}a_i^{(k,1)}+\beta_{k-1}b^{(1)} & \sum\limits_{i=0}\limits^{k-1}(k-i)\alpha_1^i\alpha_2^{k-1-i}a_i^{(k,1)} + \beta_{k-1}b^{(3)} \\
		\sum\limits_{i=1}\limits^{k}i\alpha_1^{i-1}\alpha_2^{k-i}a_i^{(k,2)} & \sum\limits_{i=0}\limits^{k-1}(k-i)\alpha_1^i\alpha_2^{k-1-i}a_i^{(k,2)}+\beta_{k-1}b^{(2)}
	\end{pmatrix}.
\end{align*}
In practice the solutions are found by applying the Newton algorithm for a high amount of different initial guesses. Together with this Newton solver, algorithm \ref{algoritme n=2} allows for the construction of the tangent directions arising from bifurcations with Jacobian kernel dimension $2$.

The case $n=2$ only appears when the underlying problem exhibits symmetry and the bifurcation point is symmetric as well \cite{Mei2000,Mei1996}. In \cite{Mei2000} it is proved that for a bifurcation invariant under the actions of $C_m$ or $D_m$ ($m\geq 2$), the reduced system \eqref{reduced equation} has isolated solutions for $k=m-1$ and is invariant under the actions of $C_m$ respectively $D_m$ as well.
Due to the $C_m$ or $D_m$ invariance of \eqref{reduced equation}, a lot of its solutions result in tangent directions to equivalent solution branches: application of the action $\tau_\omega$ (see proposition \ref{prop - symmetrie}) on solutions of one branch lead to solutions of another. To prevent equivalent branches from being calculated, only one solution of \eqref{reduced equation} for each class of emerging branches is selected. For $m$ even/odd this results in respectively $3$ or $2$ different solutions.

Algorithm \ref{algoritme n=2} was used for the results of the triangular and rotation-symmetric materials in section \ref{sectie resulaten}. In both examples, it yielded the required tangent directions for each branch point with Jacobian kernel dimension $2$.

\subsection*{The equivariant branching lemma}
Though algorithm \ref{algoritme n=2} can always be applied for bifurcations with Jacobian kernel dimension $2$, it becomes less practical for strongly symmetrical problems: if the bifurcation is invariant under the actions of $C_m$ or $D_m$, $k\geq m-1$ iterations might be required before a reduced system that only contains isolated solutions is obtained \cite{Mei1996}. Before this system is found, at least $((m-1)m-2)/2$ solves of a linear equation with the Jacobian need to be performed. This is unwanted for $m\gg 1$.
For $m\geq 4$ one can still apply the algorithm to determine an equation for $\beta_1$ and $w_1$ as a polynomial in the variables $\alpha_1$ and $\alpha_2$ (see \eqref{polynoom w1} and \eqref{polynoom beta1}), this requires only one solve of the linear problem. For $D_m$ symmetric bifurcations the unknowns $\alpha_1$ and $\alpha_2$ can then alternatively be determined by application of the equivariant branching lemma. We will describe this second method in the remainder of this section.

We call a point $(\psi,\mu)$ $\gamma$-invariant (for $\gamma\in D_m$) if its transformation under the action $\gamma$ remains in the same solution family:
$$
(\psi,\mu) \text{ is } \gamma\text{-invariant} \iff \exists \eta\in[0,2\pi]: \theta_\eta \psi = \gamma\psi.
$$

Consider a branch point invariant under the actions of $D_m$ (with $m\geq 4$). If the eigenvalue corresponding to the bifurcation crosses the origin with non-zero speed, the equivariant branching lemma guarantees the existence of symmetry-breaking branches: new solution curves with symmetries corresponding to the conjugacy classes of the isotropy subgroups of $D_m$ will emerge from the branch point \cite{Golubitsky2002,Hoyle2006}. These conjugacy classes are $\langle\sigma\rangle$ and $\langle\tau_{\eta_0}\sigma\rangle$ (with $\eta_0=2\pi/m$) if $m$ is even and just $\langle\sigma\rangle$ if $m$ is odd.

Using the knowledge of the symmetry $\gamma$ of a solution curve that emerges from the branch point, it is possible to calculate the remaining unknowns $\alpha_1$ and $\alpha_2$ that construct the tangent direction ($\dot{\psi},\dot{\mu}$) to this curve. First, a representative $(\psi'_b,\mu_b)$ of the bifurcation point for which $\gamma\psi'_b=\psi'_b$ needs to be determined. In practice this is realized by applying a least-squares method: the point
$$
\psi'_b = \theta_{\eta_\gamma}\psi_b \text{ with } 
\eta_\gamma = \frac{1}{2}\arctan\left(\frac{\int\limits_{x\in\Omega}\mathfrak{Re}(\psi_b(x))\mathfrak{Im}(\gamma\psi_b(x))dx - \int\limits_{x\in\Omega}\mathfrak{Im}(\psi_b(x))\mathfrak{Re}(\gamma\psi_b(x))dx}{\int\limits_{x\in\Omega}\mathfrak{Re}(\psi_b(x))\mathfrak{Re}(\gamma\psi_b(x))dx + \int\limits_{x\in\Omega}\mathfrak{Im}(\psi_b(x))\mathfrak{Im}(\gamma\psi_b(x))dx}\right)
$$
satisfies $\gamma\psi'_b\approx\psi'_b$. Let $(\psi,\mu)$ be a point of the $\gamma$-symmetric solution curve that emerges from $(\psi'_b,\mu_b)$: $\gamma\psi=\psi$. If $(\psi,\mu)$ is chosen sufficiently close to $(\psi'_b,\mu_b)$, we have
\begin{equation} \label{benadering solution curve}
	\begin{cases}
		&\psi \approx \psi'_b+\epsilon\theta_{\eta_\gamma}\dot{\psi}, \\
		&\mu \approx \mu_b+\epsilon\dot{\mu}
	\end{cases}  \quad \text{with } \epsilon\in\mathbb{R},\epsilon \ll 1
\end{equation}
with the tangent direction $(\theta_{\eta_\gamma}\dot{\psi},\dot{\mu})$ given by (see \eqref{vorm tangents},\eqref{polynoom w1} and \eqref{polynoom beta1})
\begin{equation} \label{vorm tangent na sub beta en w}
	\begin{cases}
		&\theta_{\eta_\gamma}\dot{\psi} = \alpha_1 (\theta_{\eta_\gamma}\phi_1+\kappa_1^{(1)}\theta_{\eta_\gamma}v_0) + \alpha_2 (\theta_{\eta_\gamma}\phi_2+\kappa_0^{(1)}\theta_{\eta_\gamma}v_0), \\
		&\dot{\mu} = \alpha_1\kappa_1^{(1)} + \alpha_2\kappa_0^{(1)}
	\end{cases}
\end{equation}
with unknowns $\alpha_1$ and $\alpha_2$, and $v_0$, $\kappa_0^{(1)}$ and $\kappa_1^{(1)}$ defined by the initial steps of algorithm \ref*{algoritme n=2}. Since $\gamma\psi=\psi$ and $\gamma\psi'_b=\psi'_b$, equation \eqref{benadering solution curve} implies $\gamma\theta_{\eta_\gamma}\dot{\psi}=\theta_{\eta_\gamma}\dot{\psi}$. This result is combined with a least-squares method to yield the following formulas for $\alpha_1$ and $\alpha_2$:
\begin{align*}
	&\alpha_1 = -\int\limits_{x\in\Omega}\left(\mathfrak{Re}(\gamma\chi_1(x))-\mathfrak{Re}(\chi_1(x)) + \mathfrak{Im}(\gamma\chi_1(x))-\mathfrak{Im}(\chi_1(x))\right) \\
	&\hspace{53pt}*\left(\mathfrak{Re}(\gamma\chi_2(x))-\mathfrak{Re}(\chi_2(x)) + \mathfrak{Im}(\gamma\chi_2(x))-\mathfrak{Im}(\chi_2(x))\right)dx, \\
	&\alpha_2 = \int\limits_{x\in\Omega}\left(\mathfrak{Re}(\gamma\chi_1(x))-\mathfrak{Re}(\chi_1(x)) + \mathfrak{Im}(\gamma\chi_1(x))-\mathfrak{Im}(\chi_1(x))\right)^2dx \\
	&\text{with } \chi_1 = \theta_{\eta_\gamma}\phi_1+\kappa_1^{(1)}\theta_{\eta_\gamma}v_0 \qquad \chi_2 = \theta_{\eta_\gamma}\phi_2+\kappa_0^{(1)}\theta_{\eta_\gamma}v_0.
\end{align*}
Substitution of $\alpha_1$ and $\alpha_2$ in \eqref{vorm tangent na sub beta en w} gives the desired $\gamma$-symmetric tangent.

Though the use of the equivariant branching lemma is an efficient alternative for determining the tangent directions, it can only be applied if the bifurcation point is $D_m$ symmetric (for a certain $m\geq 4$). Such bifurcations cannot arise if the material or the magnetic field only exhibits rotational symmetry ($C_m$). It is however still possible to apply algorithm \ref{algoritme n=2} in this case.

Calculation of the bifurcation diagram for the square material of section \ref{sectie resulaten} made use of the method based on the equivariant branching lemma. For each branch point with Jacobian kernel dimension $2$, it was able to construct the required tangent directions.

\section{Numerical results} \label{sectie resulaten}
This section contains the solution landscapes for multiple shapes of materials, calculated with an extension of the Python package PyNCT \cite{Draelants2015}. We will only consider small-scale (mesoscopic) shapes, which are for example of interest when building nanoscale fluxonics devices to use in e.g. SQUIDs \cite{Kleiner2004}, RSFQ processors \cite{Filippov2012} and supercomputers \cite{Dorojevets2015,Holmes2013,Murphy2017}. The materials used in these devices are typically shaped as a triangle, square or disc. Our methods however allow for general shapes, which will be indicated by the second example. 

The materials are subject to a homogeneous magnetic field, numerical continuation is performed in this fields strength ($\mu$). For each shape, the homogeneous solution $\psi \equiv 1$ in absence of a magnetic field ($\mu=0$) is taken as the starting point. The other points and bifurcations are determined using the tools discussed in this paper, yielding an automatically explored solution landscape.
Note that the bifurcation diagrams are plotted in terms of the expression
$$
\mathcal{E}(\psi) = -|\Omega|^{-1}\int_{\Omega}|\psi|^4d\Omega,
$$
which is part of the Gibbs energy \eqref{Gibbs free energy functional} (see \cite{Schlomer_square} for more details).

\subsection*{Triangular material}

Our first example is a material shaped as a regular triangle, the size of the edges is $6$ (scaled in units of the coherence length $\xi$, see section \ref{sectie Ginzburg-Landau systeem}). The corresponding bifurcation diagram, together with representative solutions for each solution curve, is given in figure \ref{fig - bifurcation diagram triangle}.
\begin{figure}
	\centering
	\includegraphics[trim = 0mm 0mm 0mm 0mm,clip,scale=0.12]{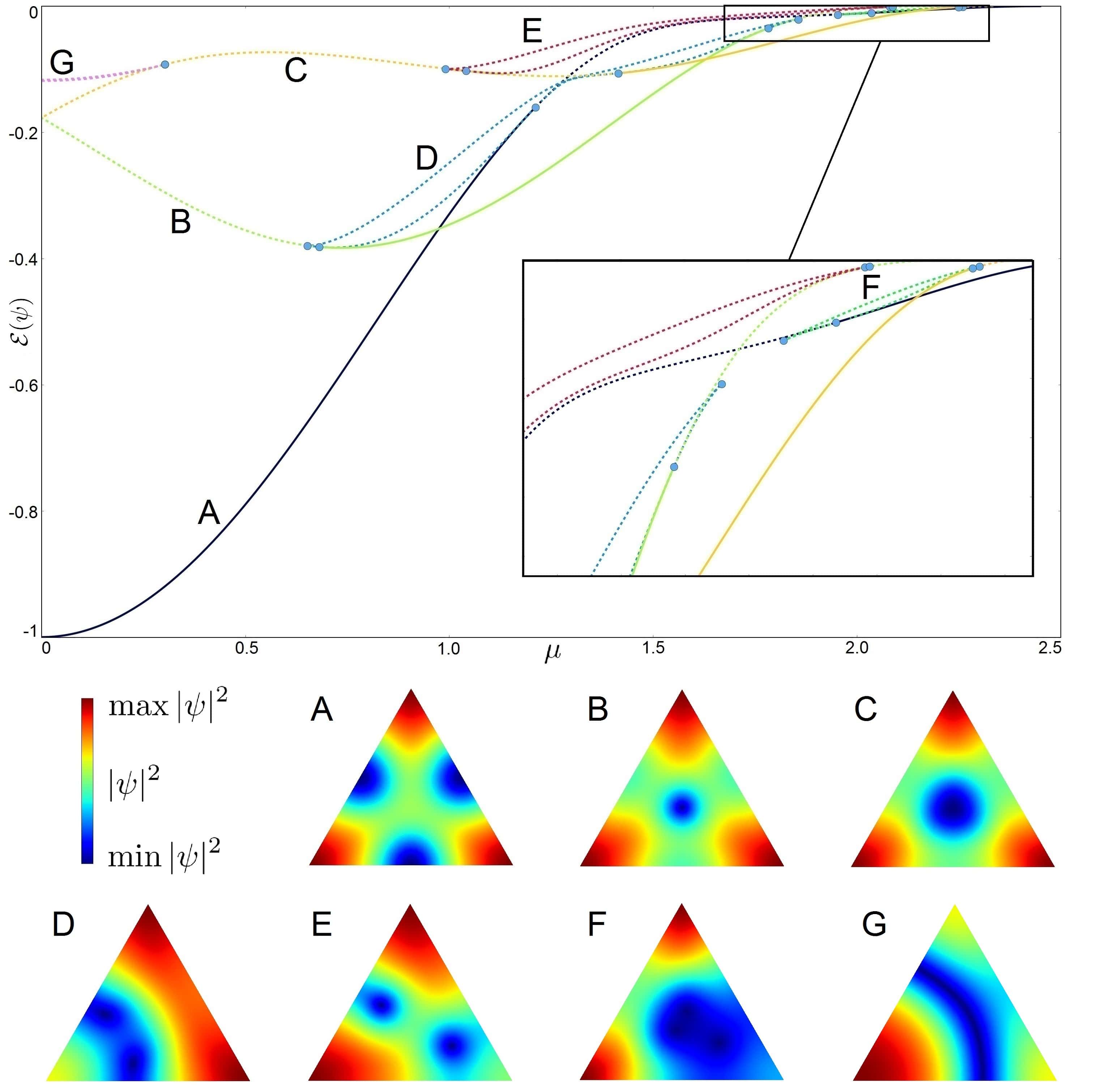}
	\caption{Bifurcation diagram for a triangular material with edges sized $6$ (above) and representative solutions for each solution curve (below). The free energy $\mathcal{E}(\psi)$ of the solutions is shown as a function of the strength of the applied magnetic field $\mu$. Solid (dashed) lines represent stable (unstable) solutions. Bifurcation points are indicated with a blue dot.}
	\label{fig - bifurcation diagram triangle}
\end{figure}

A total of $16$ bifurcation points were found during the continuation: $7$ turning points and $9$ branch points. For all of these branch points, the corresponding Jacobian contained a kernel with dimension $n=2$. Application of algorithm \ref{algoritme n=2} yielded the tangent directions to the emerging curves. Note that only the initial steps of this algorithm had to be executed.

The trivial solution $\psi\equiv 1$, $\mu=0$ is part of solution branch A. For values of $\mu$ between $0$ and approximately $1.21$, branch A is stable. For non-zero field strength, the solutions consist of zones of low supercurrent density near the edges of the domain. At $\mu\approx 1.21$ branch A destabilizes. If $\mu$ is further increased along this branch, three vortices move in from the midpoints of the edges towards the center. At field strength $\mu\approx 1.90$, the three vortices arrive at the center and form a giant vortex with multiplicity $3$. Further increase shows how this giant vortex breaks up again and three separate vortices slowly move away from the center. At $\mu\approx 2.03$ branch A restabilizes, and at a value of $\mu\approx 2.45$ the solution curve connects to the trivial zero solution $\psi\equiv 0$. Solutions between these two values consist of three separate vortices near the centre of the material.

Except for branch A, two more main solution branches (with full $D_3$ symmetry) were discovered: branches B and C. Branch B consists of solutions with a single vortex in the centre of the material. It is stable for values of $\mu$ between approximately $0.68$ and $1.78$. When the fields strength is increased along solution branch B, three vortices start to appear near the edges of the domain. At $\mu\approx2.10$, these vortices start to move towards the single, central vortex. These do however not join together, as at $\mu\approx 2.23$ solution curve B connects to the trivial zero solution. Except for the bifurcations at $\mu\approx0.68$ and $\mu\approx1.78$, branch B contains one at $\mu\approx 2.08$ as well.

For low values of $\mu$, solution branch C also consists of solutions with a single vortex in the centre. As the field strength is increased along this branch, three more vortices, with opposite vorticity to the central one, move from the edges of the domain towards this central vortex. At $\mu\approx 1.26$, these vortices join the central one to form one giant vortex with multiplicity $2$. This vortex remains in the centre of the material during further increase of the fields strength. Solution branch C is unstable for low values of $\mu$, but stabilizes between $\mu\approx1.41$ and $\mu\approx2.25$. For high values of the fields strength, the edges of the domain show zones of low supercurrent density as well. The curve connects to the trivial zero solution at $\mu\approx2.34$. Four bifurcations were found in total on solution curve C: the ones at $\mu\approx1.41$ and $\mu\approx2.25$, where the stability of the solutions change, and two more at  $\mu\approx0.30$ and $\mu\approx1.04$.

Solution branch D connects the three main solution branches. Starting from the bifurcation point of branch A at $\mu\approx1.21$ and decreasing the parameter, a single vortex moves from the edge of the material towards the centre. It arrives there at $\mu\approx0.68$, where the curve connects with branch B. Decreasing the fields strength further shows how the vortex moves towards the opposing corner of the triangle, until a turning point at $\mu\approx0.65$ is encountered. When $\mu$ is increased from this value, the vortex moves further towards the corner until $\mu\approx1.12$. Around this value, a second vortex appears from the corner the original one was moving to, and both vortices start to move towards the centre of the material. At $\mu\approx 1.33$ they are joined into a giant vortex with multiplicity $2$ (this does not happen in the centre). Further increase shows how this giant vortex is split into two vortices again, one of which moves back towards the corner, the other one towards the centre. A second turning point is encountered at $\mu\approx 1.86$. Decreasing the fields strength from this value, both vortices arrive at their destination: at $\mu\approx1.78$ solution branch D connects with branch B, this solution consists of a single vortex in the centre and zones of low supercurrent density at the edges. Decreasing $\mu$ from this bifurcation point shows how a second vortex moves from one edge of the triangle towards the centre, where it joins the central vortex into a giant one with multiplicity $2$. This corresponds with the bifurcation point at $\mu\approx 1.41$. Finally, when the parameter is decreased even further, the giant vortex splits into two vortices, which both move towards the edges of the domain. Solution curve D reconnects with itself at $\mu\approx1.21$.

Except for branch D, branches A and C are connected through solution branch F as well. Starting from the solution of branch A at $\mu\approx2.03$, consisting of three separate vortices near the centre of the material, the vortices move towards one of the materials corners when the field strength is decreased. A turning point is encountered at $\mu\approx1.95$. Increase of $\mu$ from this value shows how one of the vortices completely moves to the corner while the other two return to the centre, until a turning point at $\mu\approx2.26$ is reached. Decreasing $\mu$, zones of low supercurrent density start to appear at the edges, until a value of $\mu\approx2.25$, where solution branch F connects with branch C. When the fields strength is decreased further, a vortex moves from one of the edges towards the centre, eventually leading back to the bifurcation at $\mu\approx2.03$.

Solution branch E connects branches B and C as well. We start from the bifurcation of branch B at $\mu\approx2.08$, consisting of a solution with one central vortex. Decreasing $\mu$, two vortices - one from a triangle's corner, the other from the opposite edge - move towards the central vortex. Before the one moving from the edge reaches the centre, the other two join together into a giant vortex with multiplicity 2. Further decrease shows how this vortex splits back into two smaller ones, after which all of the vortices start to move towards the edges. Eventually two vortices remain: one in a corner, another at the opposite edge. These two vortices move towards each other, until they are joined into a giant one with vorticity 2 at the centre of the material. This happens at $\mu\approx1.04$, where branch E connects with branch C. Decreasing $\mu$ further, this vortex splits again, with both vortices moving towards different corners. A turning point is encountered for $\mu\approx0.99$. Increasing the fields strength from this value, the vortices continue to move towards the corners. Before these corners are reached a new vortex appears at the intermediate edge. This new vortex starts to move towards the centre, while the other two vortices now move towards the remaining edges. After encountering another turning point at $\mu\approx2.09$, solution branch E eventually reconnects with itself at the bifurcation for $\mu\approx2.08$.

Finally, branch G connects branch C with its symmetric counterpart for negative values of $\mu$. The solution branch starts at $\mu\approx0.30$, a bifurcation point of branch C that shows a pattern with one central vortex connected to the edges by zones of low supercurrent density. Decreasing the fields strength, the vortex moves to one of the corners, the zone connecting it to the opposite edge stretching towards it. Increasing the parameter $\mu$ shows how the vortex moves towards one of the edges, the two zones connecting it to the other edges again stretching towards it. Almost immediately a turning point is encountered (approximately at $\mu\approx0.30$). Decreasing $\mu$ from this point shows how the vortex first keeps moving towards the edge, later it is joined with the zones of low supercurrent density, yielding a single zone stretching from one edge to the other, bent away from the intermediate corner.

\subsection*{Rotation-symmetric material}
The second material is shaped as a four-pointed star. It only exhibits rotational symmetry, the corresponding symmetry group of the problem is hence given by $C_4$. The domain is chosen such that it fits into a square with edge $6$: the distance between the central point and outer corners is given by $\sqrt{10}$, the one between the central point and inner corners by $\sqrt{2}$. The size of the short and long edges of the star itself are respectively given by $2$ and $2\sqrt{2}$. The bifurcation diagram, and representative solutions, associated with this material is given in figure \ref{fig - bifurcation diagram star}. Several close-ups of the diagram are given in figure \ref{fig - bifurcation diagram star zooms}.

\begin{figure}
	\centering
	\includegraphics[trim = 0mm 0mm 0mm 0mm,clip,scale=0.12]{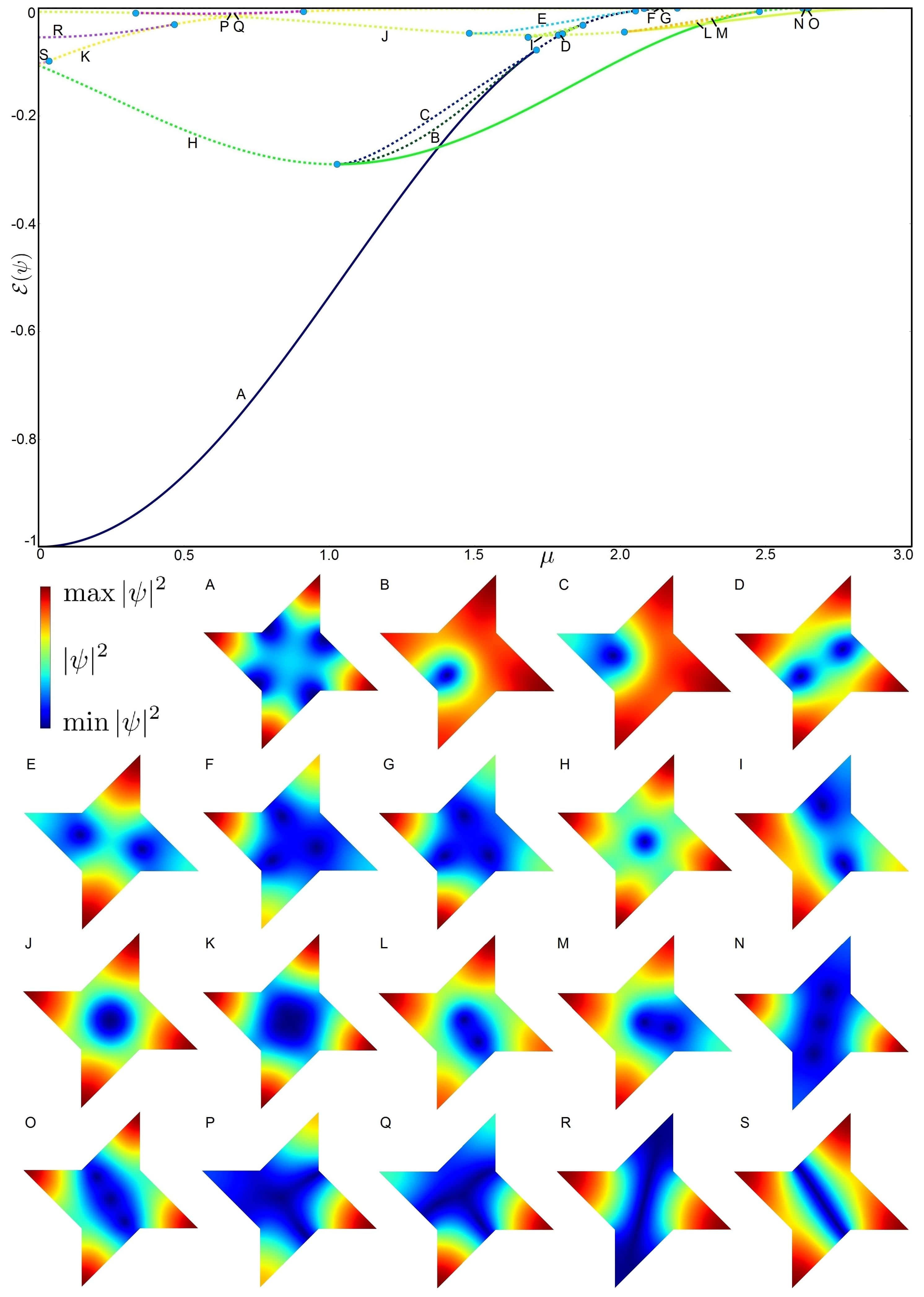}
	\caption{Bifurcation diagram for $C_4$-symmetric, star-shaped material (above) and representative solutions for each solution curve (below). The free energy $\mathcal{E}(\psi)$ of the solutions is shown as a function of the strength of the applied magnetic field $\mu$. Solid (dashed) lines represent stable (unstable) solutions. Bifurcation points are indicated with a blue dot.} \label{fig - bifurcation diagram star}
\end{figure}

\begin{figure}
	\centering
	\begin{subfigure}[t]{0.49\textwidth}
		\includegraphics[trim = 0mm 0mm 0mm 0mm,clip,scale=0.23]{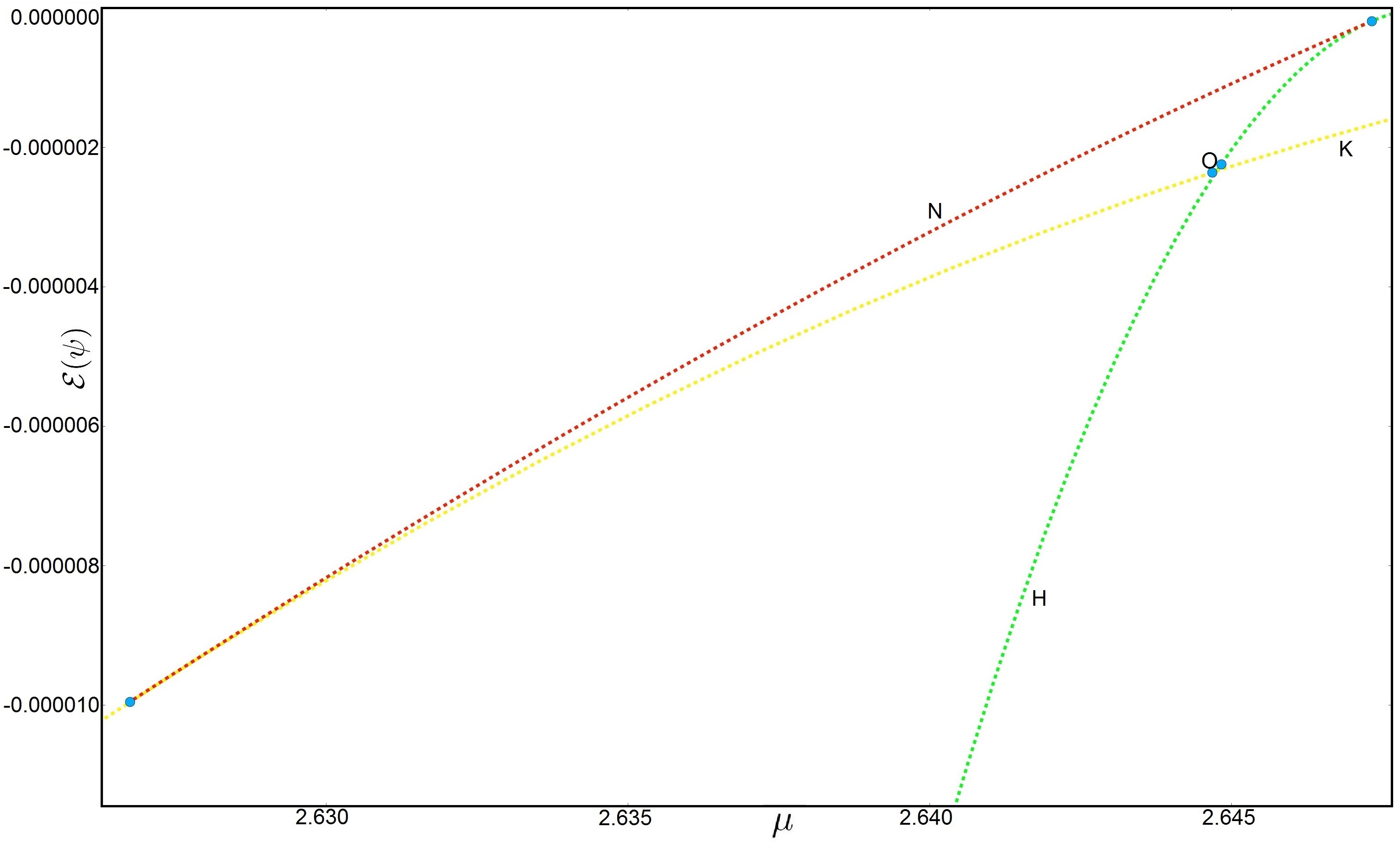}
		\caption{Close-up around branches N and O}
	\end{subfigure}
	\hfil
	\begin{subfigure}[t]{0.49\textwidth}
		\includegraphics[trim = 0mm 0mm 0mm 0mm,clip,scale=0.23]{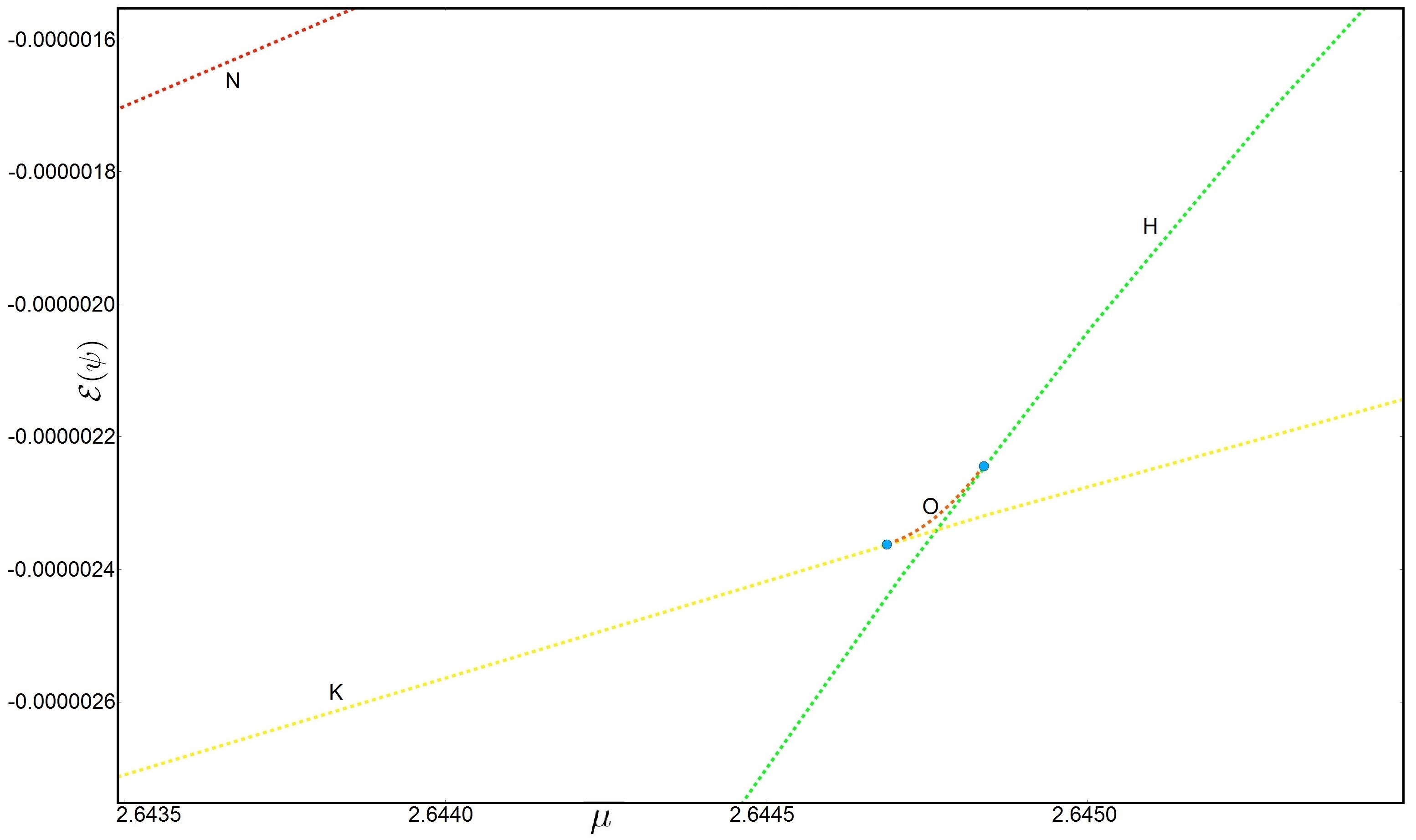}
		\caption{Close-up around branch O}
	\end{subfigure}
	\begin{subfigure}[t]{0.49\textwidth}
		\includegraphics[trim = 0mm 0mm 0mm 0mm,clip,scale=0.23]{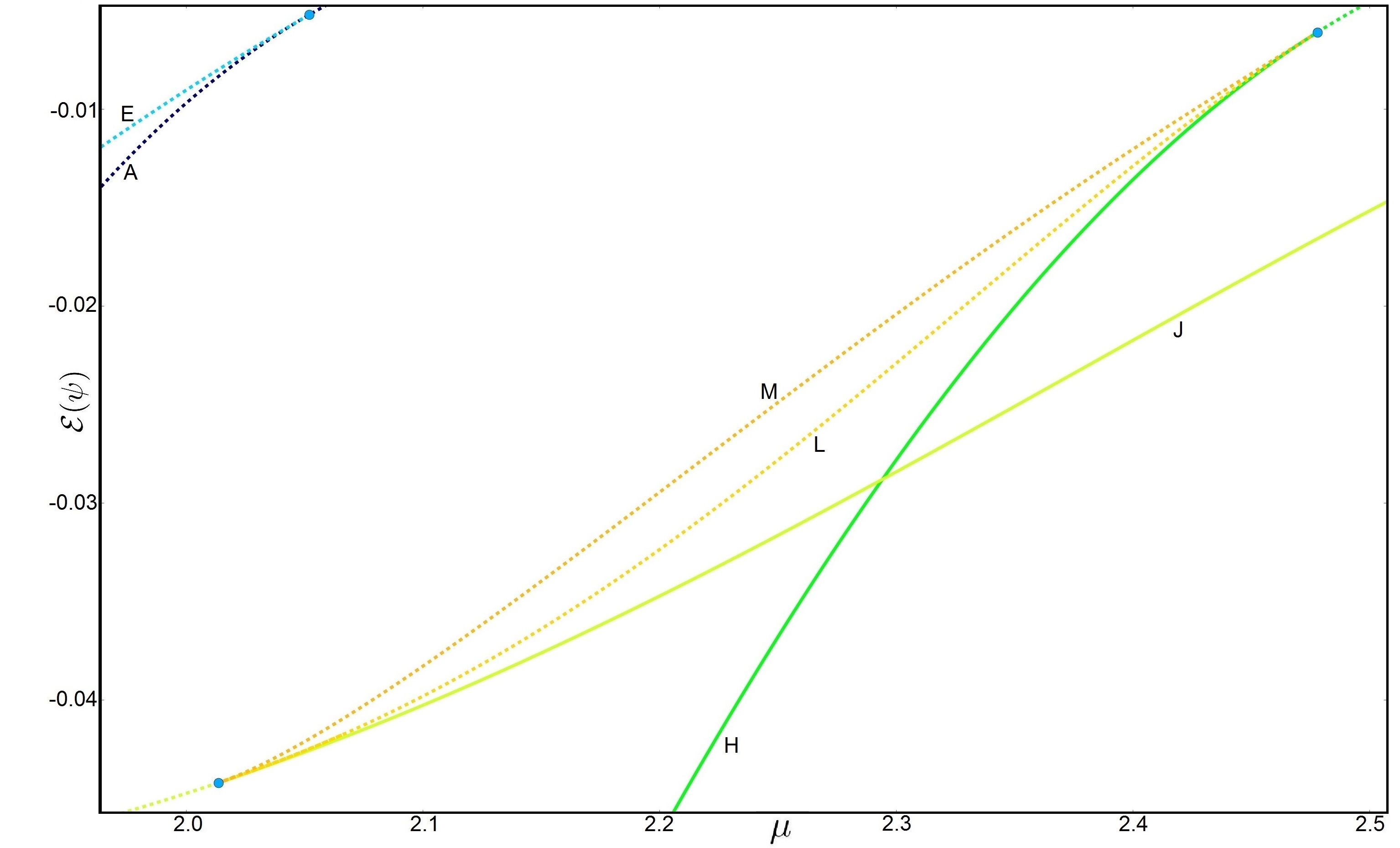}
		\caption{Close-up around branches L and M}
	\end{subfigure}
	\hfil
	\begin{subfigure}[t]{0.49\textwidth}
		\includegraphics[trim = 0mm 0mm 0mm 0mm,clip,scale=0.23]{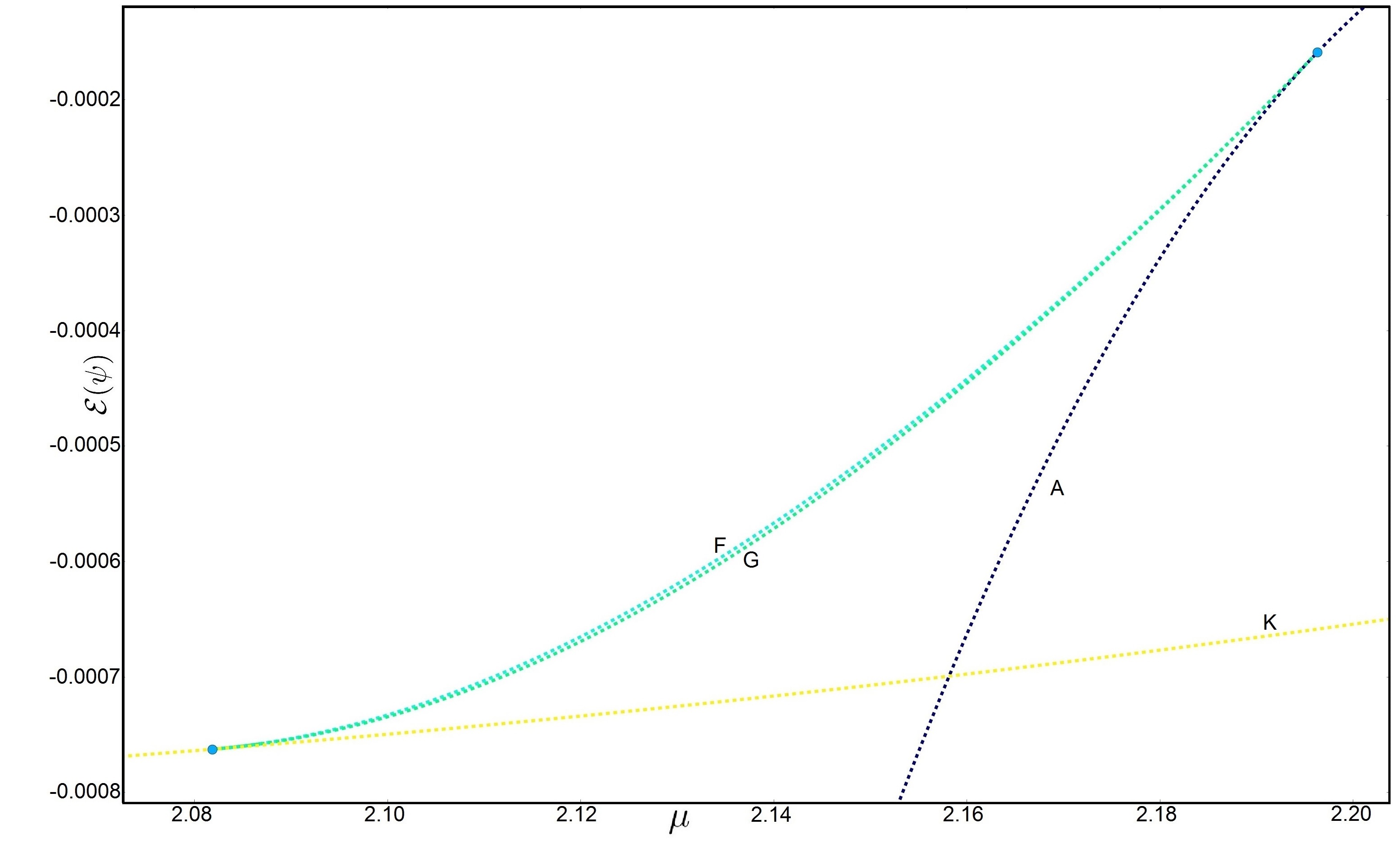}
		\caption{Close-up around branches F and G}
	\end{subfigure}
	\begin{subfigure}[t]{0.49\textwidth}
		\includegraphics[trim = 0mm 0mm 0mm 0mm,clip,scale=0.23]{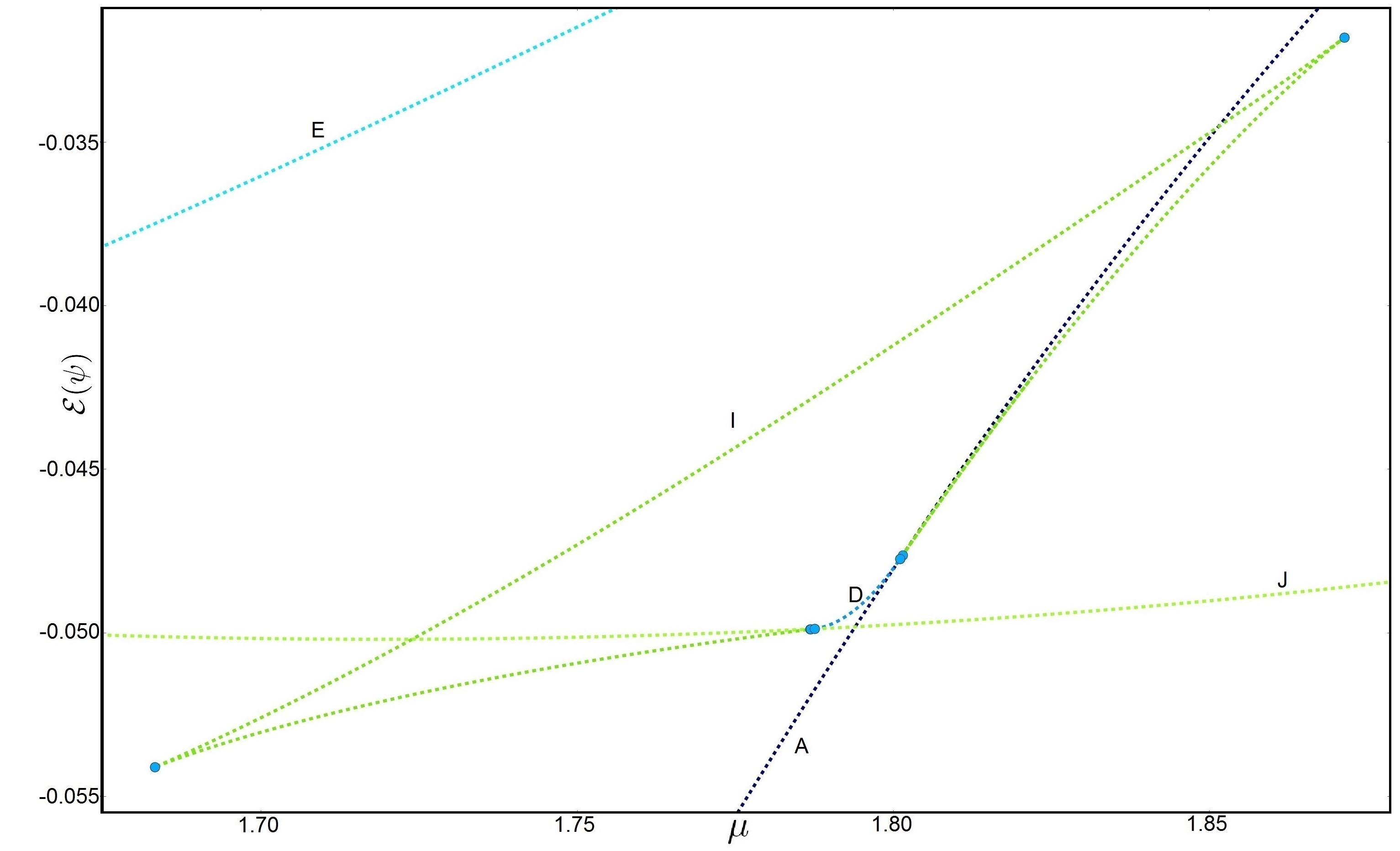}
		\caption{Close-up around branches D and I}
	\end{subfigure}
	\hfil
	\begin{subfigure}[t]{0.49\textwidth}
		\includegraphics[trim = 0mm 0mm 0mm 0mm,clip,scale=0.23]{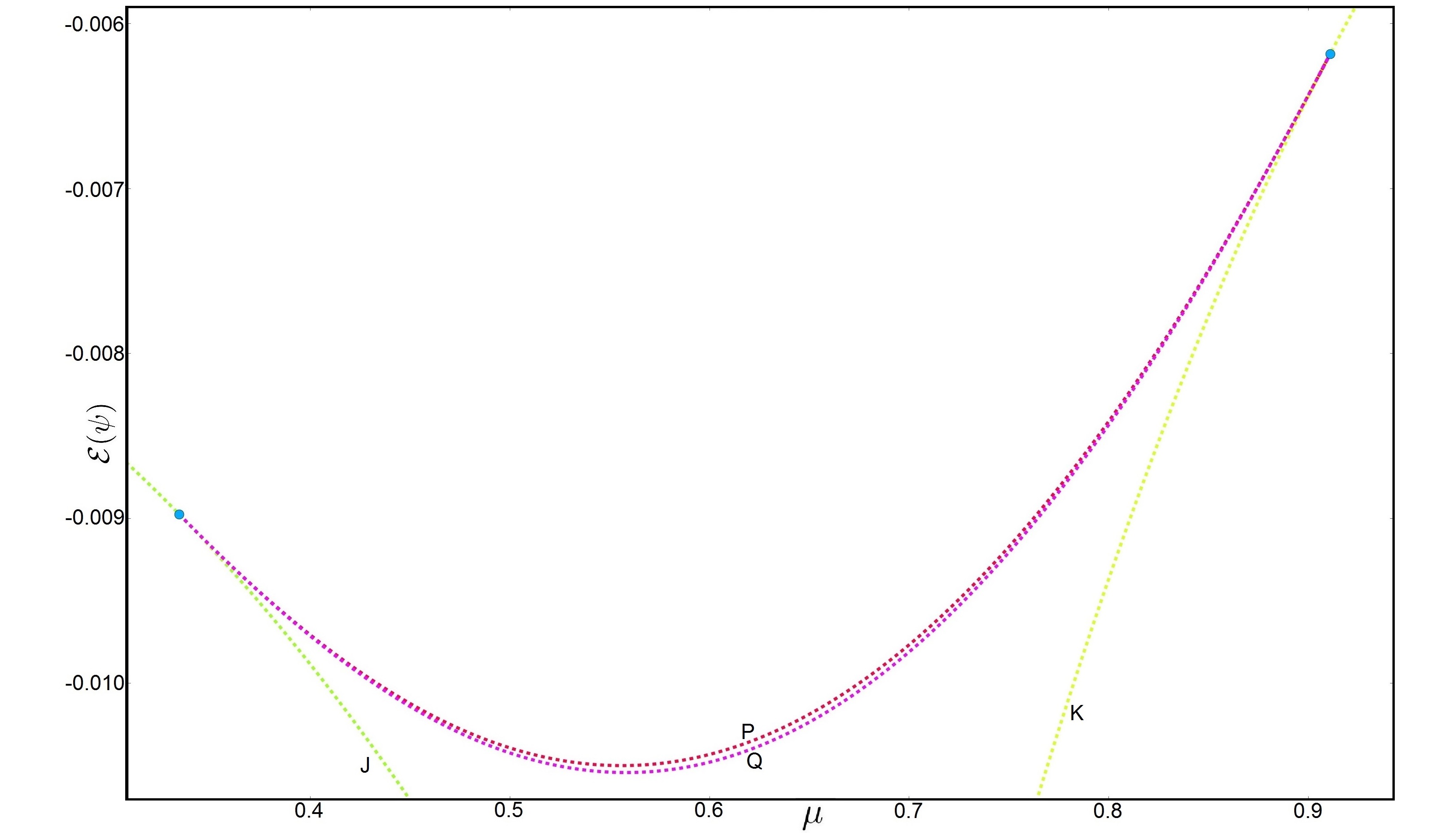}
		\caption{Close-up around branches P and Q}
	\end{subfigure}
	
	\begin{subfigure}[t]{0.49\textwidth}
		\includegraphics[trim = 0mm 0mm 0mm 0mm,clip,scale=0.23]{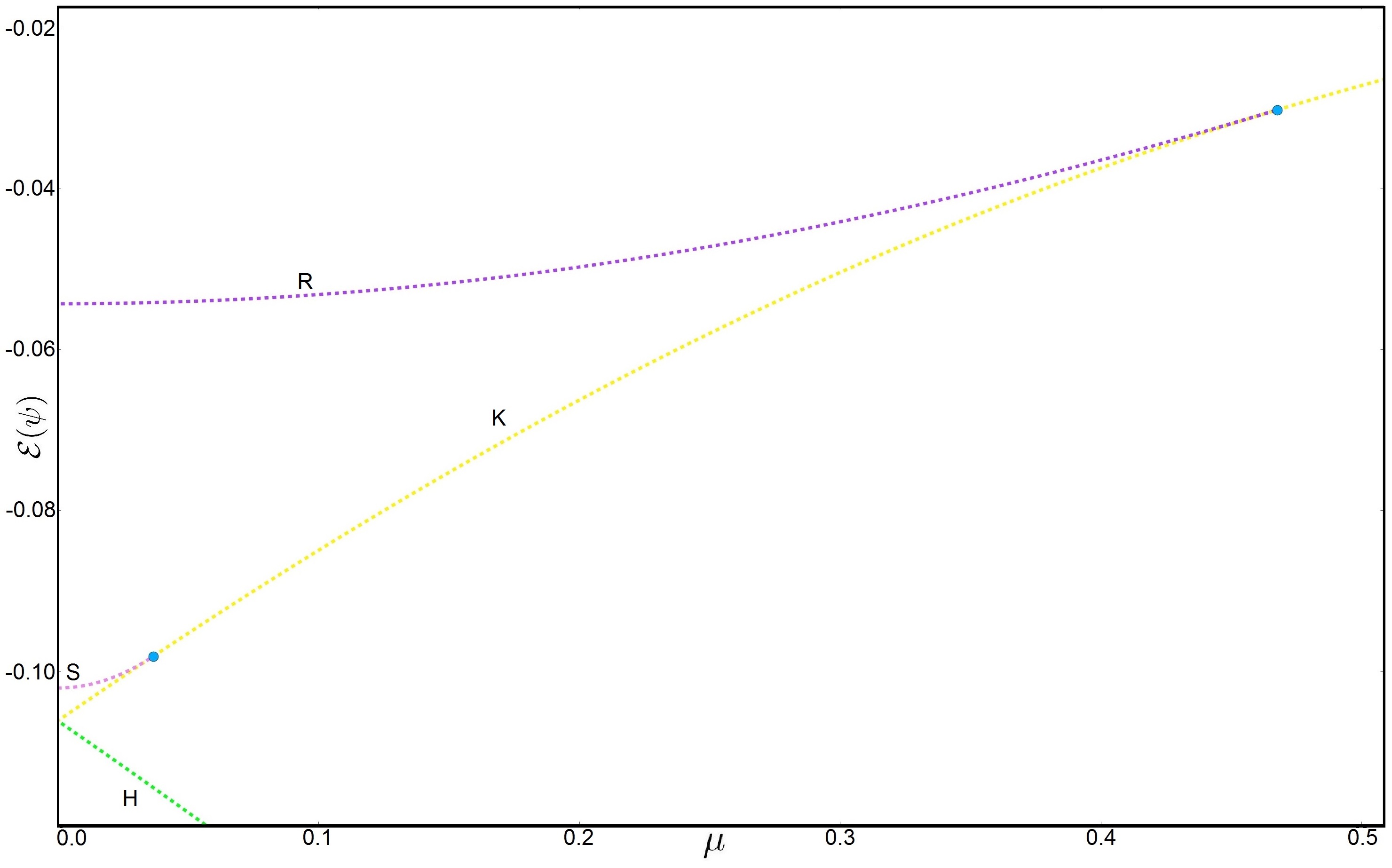}
		\caption{Close-up around branches R and S}
	\end{subfigure}
	\caption{Several close-ups of the bifurcation diagram of figure \ref{fig - bifurcation diagram star}.} \label{fig - bifurcation diagram star zooms}
\end{figure}

Multiple bifurcation points were again discovered during the continuation, this time not all of the branch points corresponded to a Jacobian with a $2$-dimensional kernel. Several ones only contained a kernel consisting of a single null vector, for these points application of algorithm \ref{algorithm Mei} resulted in the (single) new tangent direction. As has been stated in section \ref{sectie tangents}, for these cases the algorithm is equivalent to applying the algebraic branching equation. Other branch points did correspond to a Jacobian with a $2$-dimensional kernel, for these points a single iteration of algorithm \ref{algoritme n=2} ($k=3$) was applied to construct the tangent directions. Note that these directions could not be determined by application of the equivariant branching lemma since the problem only exhibits rotational symmetry.

Contrary to the triangular material, four main branches (with full $C_4$ symmetry) were discovered. These are labelled A, H, J and K in figure \ref{fig - bifurcation diagram star}. Branch A contains the trivial solution $\psi\equiv 1$, $\mu=0$. Increasing the strength of the applied magnetic field from this solution, four zones of low supercurrent density form near the inner corners of the material. It connects to the trivial zero solution $\psi\equiv0$ at $\mu\approx2.23$. The solution branch is stable for values of $\mu$ between $0$ and approximately $1.71$. Other bifurcations are situated at $\mu\approx1.80$, $\mu\approx2.05$ and $\mu\approx2.20$.

All of the solutions of branch H contain a single vortex in the centre of the material. For high values of the field strength ($\mu\gtrsim 2.15$), zones of low supercurrent density start to form near the inner corners as well. The solution branch is stable between values $\mu\approx1.03$ and $\mu\approx2.48$. Two more bifurcations were discovered near $\mu\approx2.65$, immediately before the curve connects to the trivial zero solution $\psi\equiv0$.

For low values of $\mu$, solutions that belong to branch J show a cross pattern of low superconductivity: zones of low supercurrent density appear in the centre and stretch towards the inner corners. Increasing the field strength along this branch, a bifurcation with the same pattern is discovered near $\mu\approx0.33$. The pattern of the solutions changes for $\mu\gtrsim0.95$: a vortex with multiplicity 2 appears in the centre of the material. Further increase of $\mu$ gives rise to more bifurcations at $\mu\approx1.48$, $\mu\approx1.79$ and $\mu\approx2.01$. At this last bifurcation, branch J stabilizes. It remains stable until it connects with the trivial $\psi\equiv0$ solution near $\mu\approx2.90$.

The final main solution branch discovered for this problem is branch K. Similar to branch H, its solutions contain a single vortex in the centre for low values of the field strength. Increasing this strength shows how four vortices, with opposite vorticity to the central one, form near the inner corners ($\mu\gtrsim0.20$), afterwards these vortices move towards the centre ($\mu\gtrsim1.12$) and join with the central one into a giant vortex with multiplicity 3 ($\mu\gtrsim2.00$). The branch eventually connects to the trivial zero solution $\psi\equiv0$ at $\mu\approx2.66$. A total of 6 bifurcations were discovered for solution branch K, at values of $\mu\approx0.04$, $\mu\approx0.47$, $\mu\approx0.91$, $\mu\approx2.08$, $\mu\approx2.62$ and $\mu\approx2.64$. Note that no stable solutions have been found for this branch.

The four main solution branches are again mutually connected by other branches. Branch A and H are connected through branches B and C. Starting from the bifurcation of branch A at $\mu\approx1.71$ and decreasing the field strength, branch B shows how a single vortex moves from one of the inner corners towards the centre of the material, appearing there at the bifurcation $\mu\approx1.03$ of branch H. Branch C shows the same behaviour, with the vortex appearing from one of the outer corners.

Branches D and E connect branch A with J. Decrease of the field strength from the bifurcation at $\mu\approx1.80$ of branch A along solution branch D shows how two vortices move from opposite inner corners towards the centre. At the bifurcation point of branch J for $\mu\approx1.79$ the vortices join together into a giant vortex with multiplicity 2. Note that branch D contains two extra bifurcations, near the ones of branch A. These will be discussed later in this section. Branch E is again similar to branch D: it connects branch A from its bifurcation at $\mu\approx2.05$ to J at the bifurcation for $\mu\approx1.48$. Contrary to branch D, the vortices appear at the outer corners in the solutions of branch E, and branch E does not contain other bifurcation points itself.

Branches A and K are connected through F and G. Both branches start at the bifurcation for $\mu\approx2.20$ of A and lead to the one for $\mu\approx2.08$ of K. Three vortices move from the corners of the material towards the centre, joining into a giant one with vorticity 3. For branch F one of these vortices appears in an outer corner, the other two at the inner ones, for branch G the opposite happens.

Solution branch H is connected to J by both L and M, both connecting these branches between the bifurcations at $\mu\approx2.48$ (on branch H) and $\mu\approx2.01$ (on branch J). For branch L, a vortex appears at an inner corner of the material and moves towards the central one, eventually joining into a giant vortex of multiplicity 2. The same happens in branch M, this time for a vortex appearing at one of the outer corners.

Another connection exists between solution branches H and K. Decrease of the field strength at the first bifurcation near $\mu\approx2.65$ along branch O shows how two vortices form near the inner corners of the material and start to move towards the central one. At the bifurcation for $\mu\approx2.64$ the three vortices join into a giant one with vorticity 3. Solution branch N is similar to O and connects branches H and K through the second bifurcation near $\mu\approx2.65$ (of H) and the one at $\mu\approx2.62$ (of K). For this branch the vortices appear near the outer corners of the material.

Two final connections appear between branches J and K. Starting from the bifurcation of branch J at $\mu\approx0.33$ with the cross pattern, increase of the field strength along branch P shows how the zones of low superconductivity start to stretch towards one of the outer corners. Eventually a new vortex moves from this corner towards the centre, appearing there at the bifurcation of branch K for $\mu\approx0.91$. Branch Q is similar, with the zones of low supercurrent density stretching towards an inner corner.

Both solution branches R and S connect K with its symmetric counterpart for negative $\mu$ values. Starting with a single vortex solution at the bifurcation for $\mu\approx0.04$, decrease of the field strength along branch S shows how the vortex changes into a zone of low superconductivity stretching from one inner corner towards the opposite one. A similar behaviour is observed for branch R, which starts at the bifurcation for $\mu\approx0.47$. For this branch the vortex solution changes into a zone of low supercurrent density that stretches between two opposite outer corners.

The final solution branch I is unique in that it is not a main branch, nor does it connect two main branches. It appears between the two extra bifurcation points of branch D. Starting from a solution consisting of two vortices near opposite inner corners at $\mu\approx1.80$ and increasing $\mu$ along branch I, the vortices start to move in the same direction towards two of the the outer corners. A turning point is encountered near $\mu\approx1.87$. Decreasing $\mu$ from this value, the vortices first continue to move towards the outer corners, but will not reach them: after a while the vortices start to move back towards each other, eventually reaching another turning point near $\mu\approx1.68$. Increase of the field strength shows how the vortices both move back towards the centre, eventually arriving at a configuration where the two vortices lie close to each other near the centre of the material, which coincides with the bifurcation of branch D at $\mu\approx1.79$. Note that the turning points at $\mu\approx 1.87$ and $\mu\approx1.68$ are the only ones that appear in the bifurcation diagram of figure \ref{fig - bifurcation diagram star}.

\subsection*{Square material}
Our final example consists of a material shaped as a square with edges sized 5.5. This problem was also treated in \cite{Schlomer_square}, where a bifurcation diagram with 13 different solution branches was constructed. Application of the methods described in the current paper yields the same bifurcation diagram, with an extra 30 branches, resulting in a total of 43 solution curves. The four main branches (with full $D_4$ symmetry) are shown in figure \ref{fig - main diagram square}. Several close-ups of the bifurcation diagram containing other solution curves are given in figure \ref{fig - bifurcation diagram square zooms}. A schematic representation of the diagram can be found in figure \ref{fig - schematic square}. Finally, representative solutions for each branch are given in figure \ref{fig - representative square}.

\begin{figure}
	\centering
	\begin{subfigure}[t]{1.0\textwidth}
		\includegraphics[trim = 0mm 0mm 0mm 0mm,clip,scale=0.5]{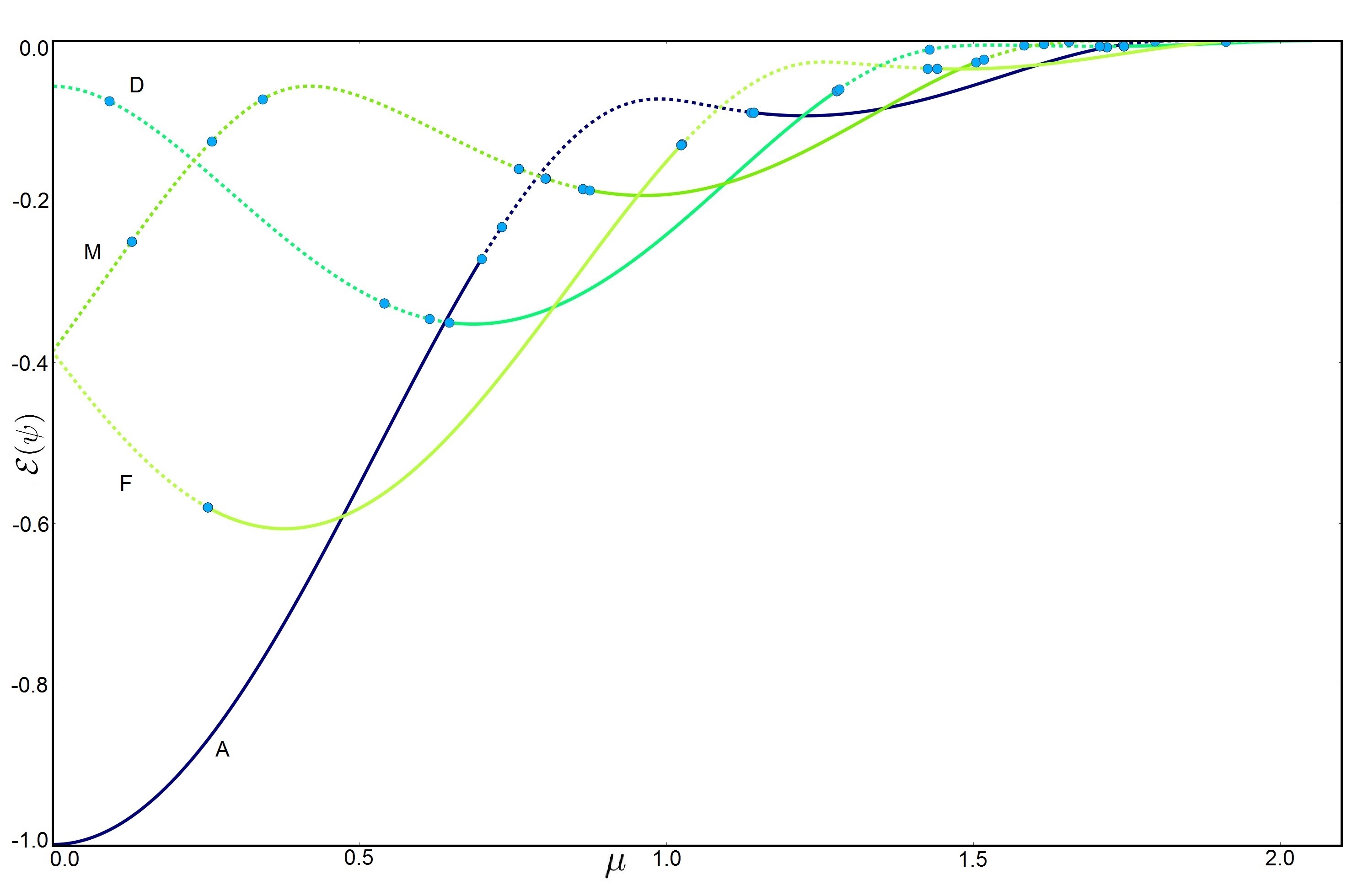}
		\caption{Complete view}
	\end{subfigure}
	
	\begin{subfigure}[t]{1.0\textwidth}
		\includegraphics[trim = 9mm 0mm 0mm 0mm,clip,scale=0.5]{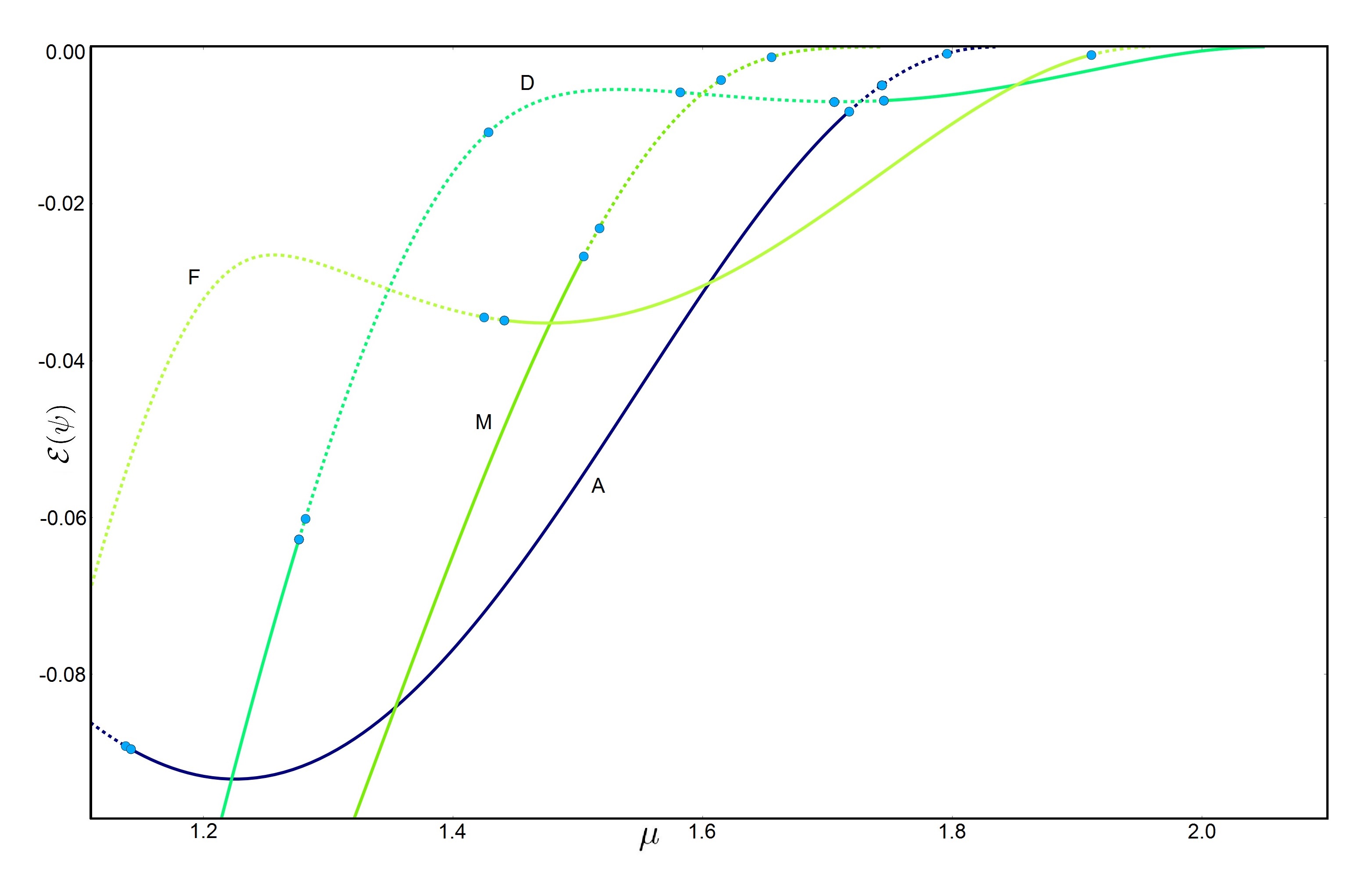}
		\caption{Close-up for $1.1\lesssim \mu \lesssim 2.1$, $-0.1 \lesssim \mathcal{E}(\psi) \leq 0$}
	\end{subfigure}
	\caption{Main solution branches of the bifurcation diagram for the square-shaped material. The free energy $\mathcal{E}(\psi)$ of the solutions is shown as a function of the strength of the applied magnetic field $\mu$. Solid (dashed) lines represent stable (unstable) solutions. Bifurcation points are indicated with a blue dot.} \label{fig - main diagram square}
\end{figure}

\begin{figure}
	\centering
	\includegraphics[trim = 0mm 0mm 0mm 0mm,clip,scale=0.12]{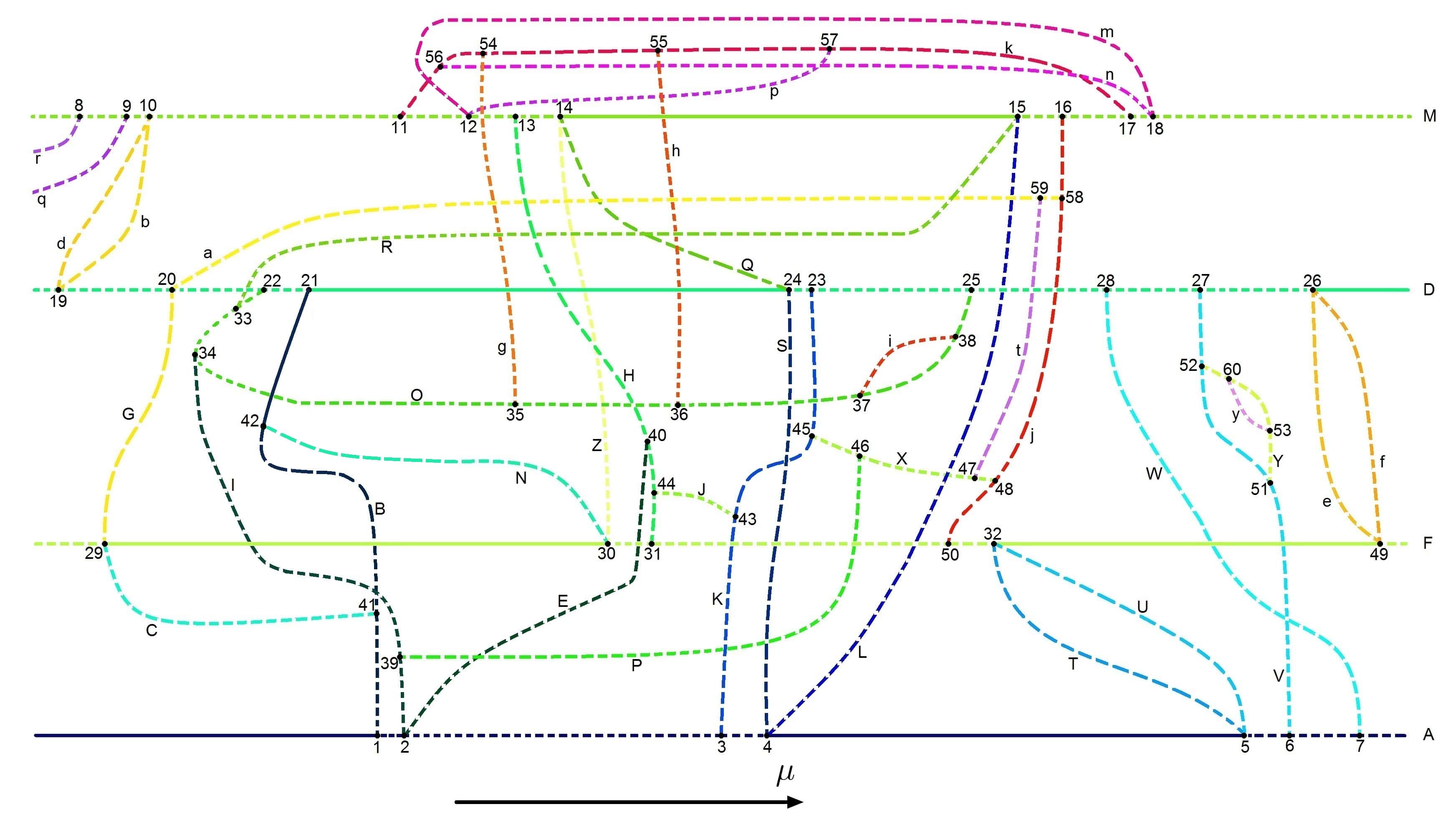}
	\caption{Schematic of solution branches and branch points for a square-shaped material with edge 5.5. Solid (dashed) lines represent stable (unstable) solutions. Representative solutions for each branch are given in figure \ref{fig - representative square}, the $\mu$ values for the branch points are given in table \ref{fig - table}.} \label{fig - schematic square}
\end{figure}

\begin{figure}
	\centering
	\begin{subfigure}[t]{0.49\textwidth}
		\includegraphics[trim = -5mm 0mm 0mm 0mm,clip,scale=0.23]{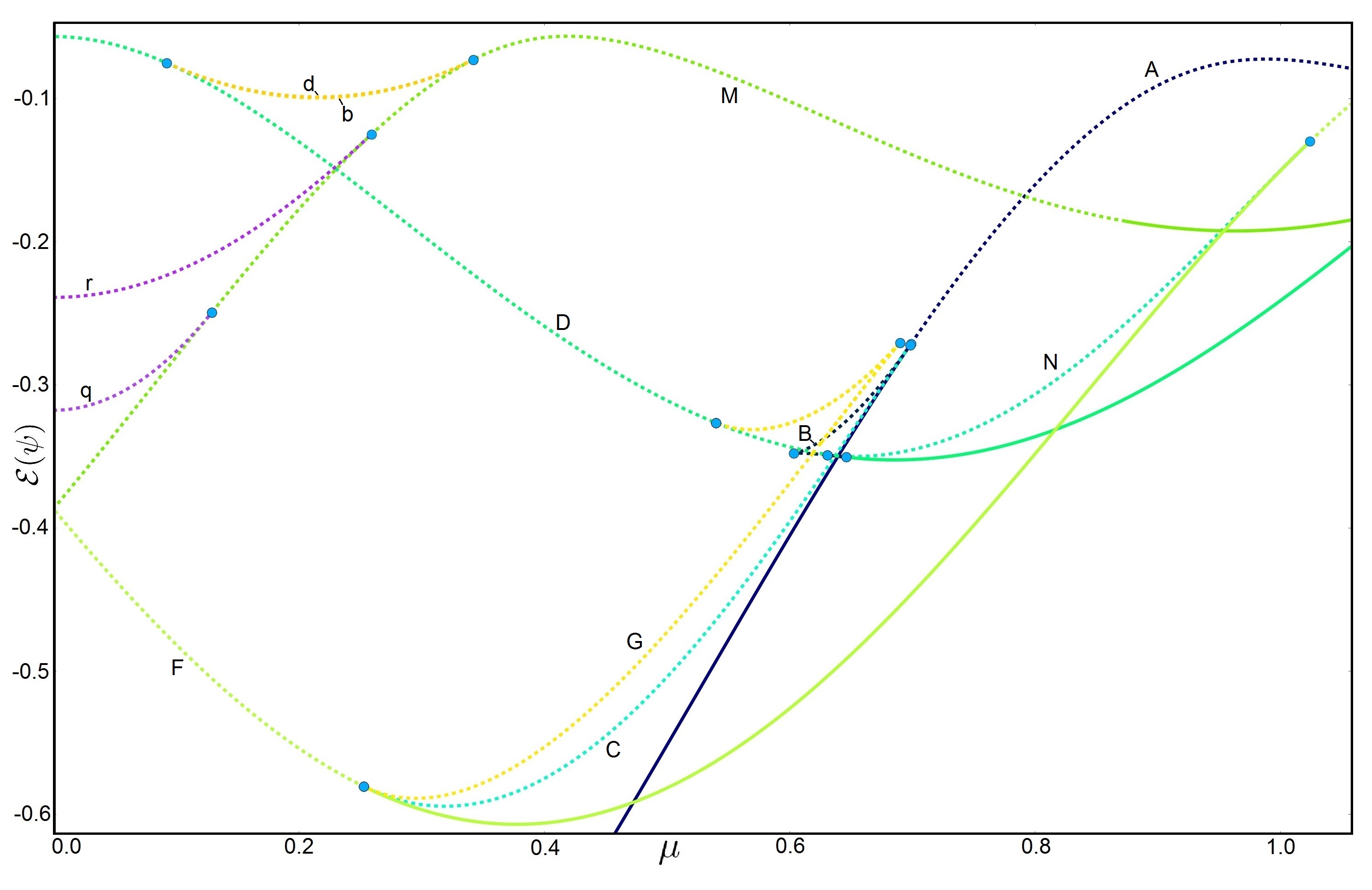}
		\caption{Branches B, C, G, N, b, d, q and r}
	\end{subfigure}
	\hfil
	\begin{subfigure}[t]{0.49\textwidth}
		\includegraphics[trim = -5mm 0mm 0mm 0mm,clip,scale=0.23]{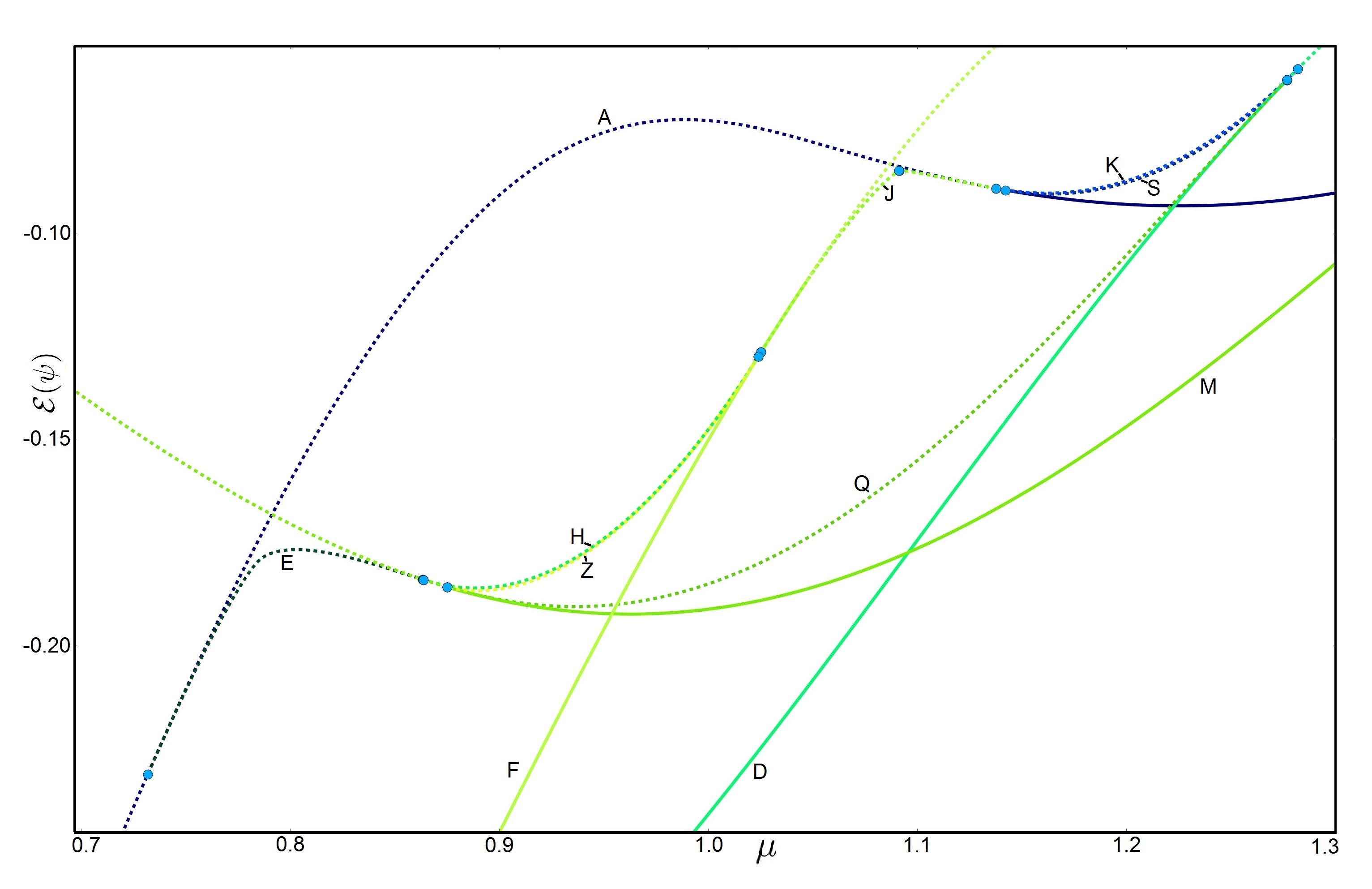}
		\caption{Branches E, H, J, K, Q, S and Z}
	\end{subfigure}
	\begin{subfigure}[t]{0.49\textwidth}
		\includegraphics[trim = 0mm 0mm 0mm 0mm,clip,scale=0.23]{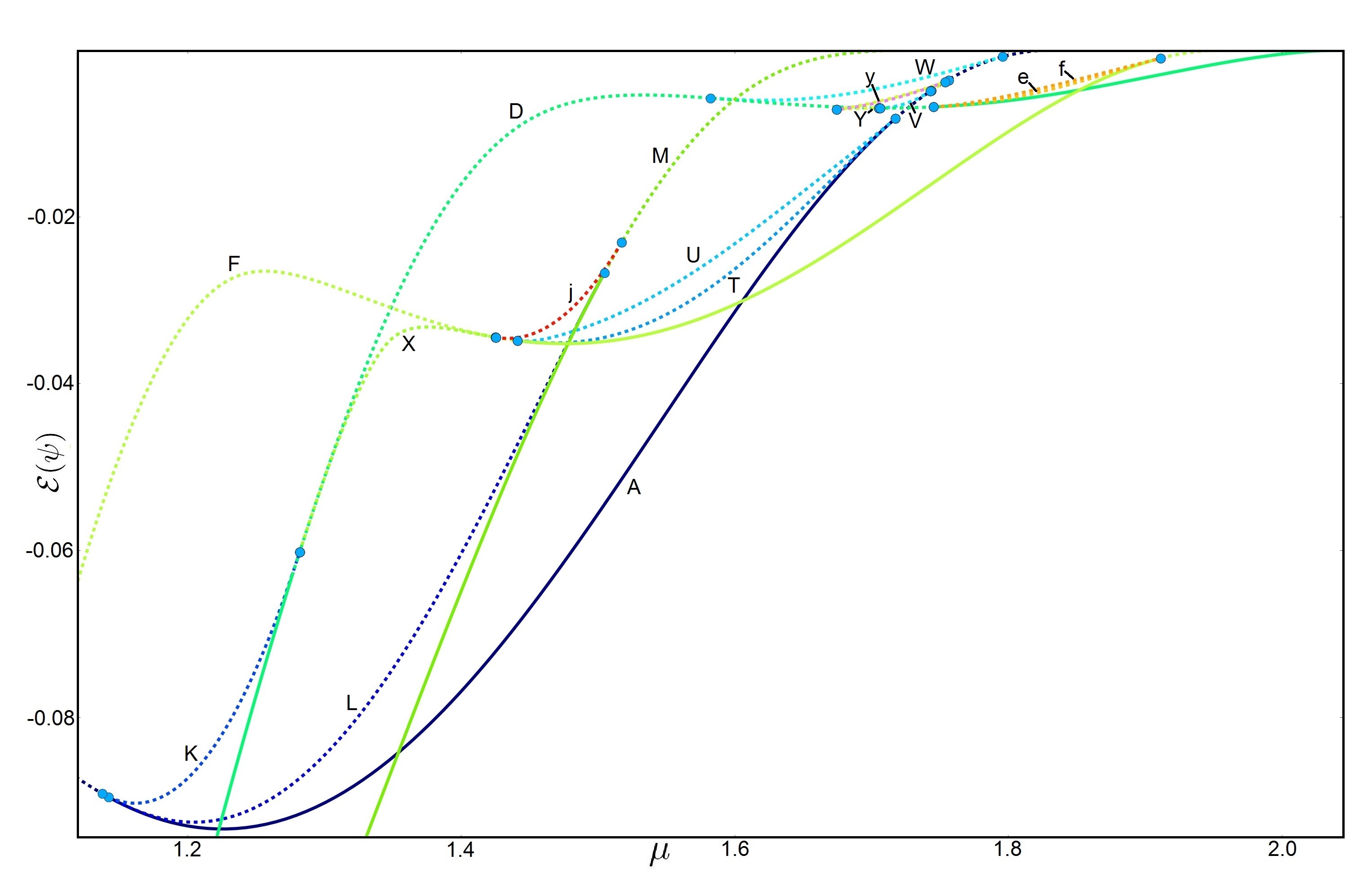}
		\caption{Branches K, L, T, U, V, W, X, Y, e, f, j and y}
	\end{subfigure}
	\hfil
	\begin{subfigure}[t]{0.49\textwidth}
		\includegraphics[trim = -10mm 0mm 0mm 0mm,clip,scale=0.23]{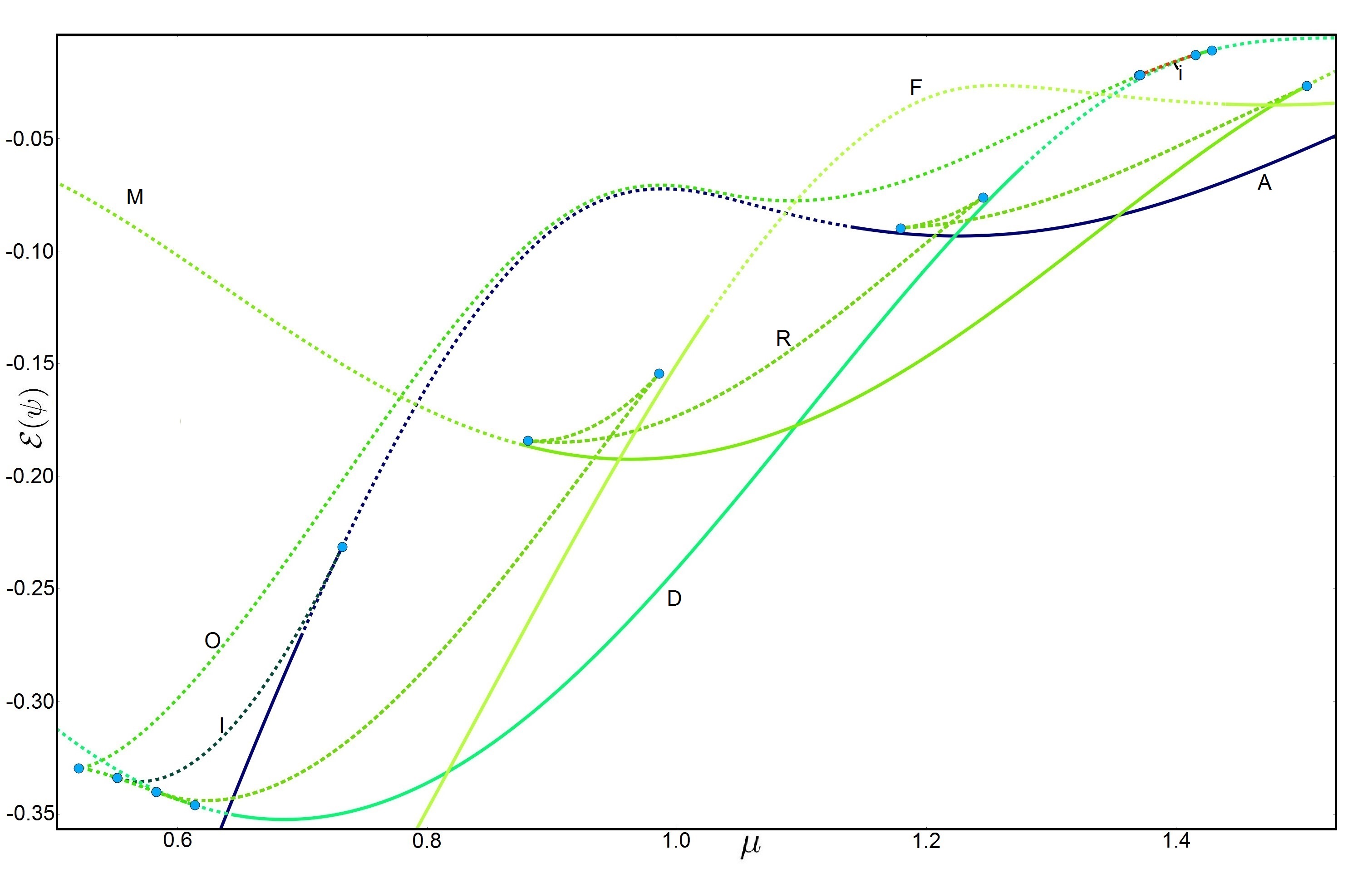}
		\caption{Branches I, O, R and i}
	\end{subfigure}
	\begin{subfigure}[t]{0.49\textwidth}
		\includegraphics[trim = 3mm 0mm 0mm 0mm,clip,scale=0.23]{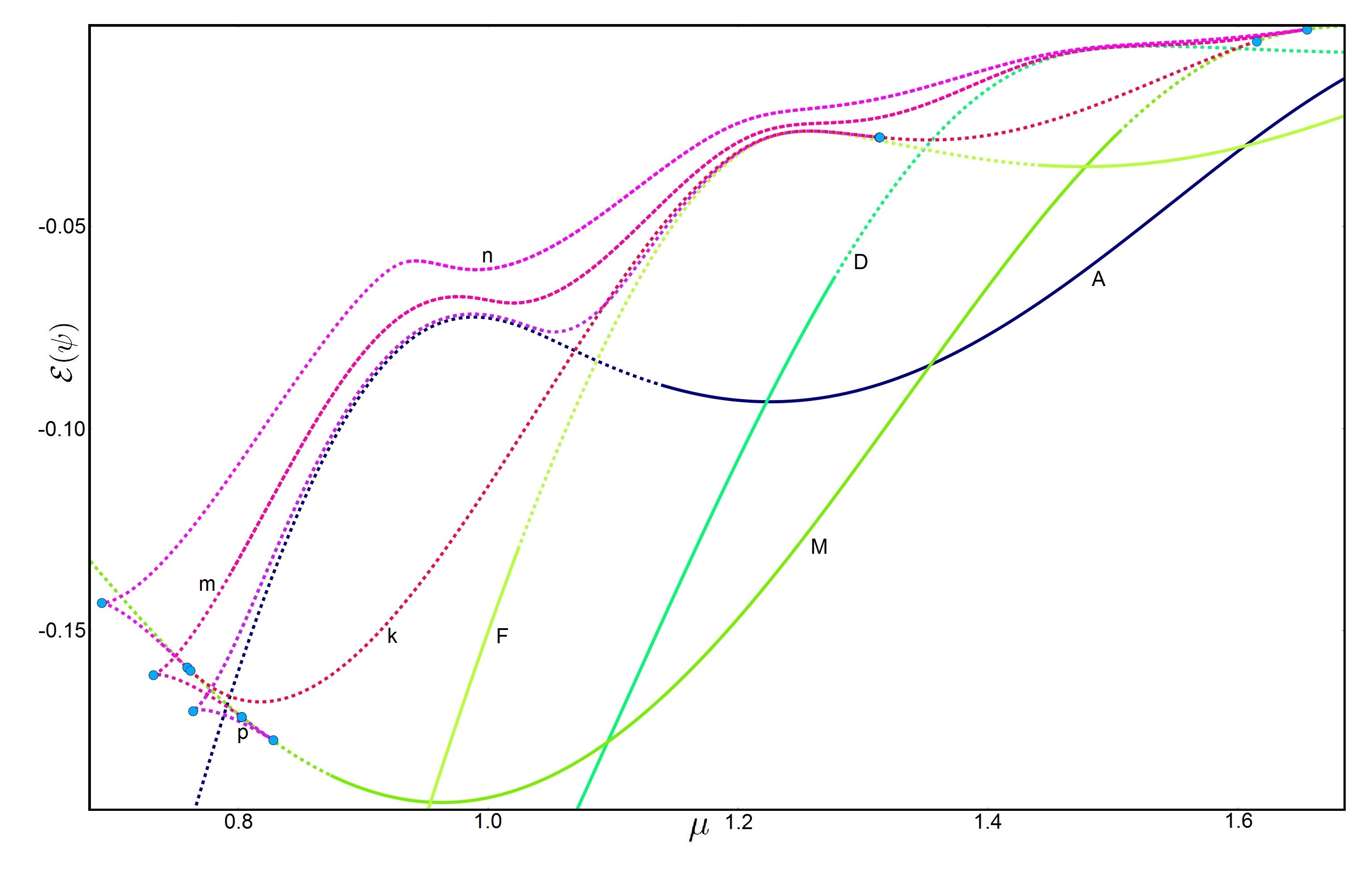}
		\caption{Branches k, m, n and p}
	\end{subfigure}
	\hfil
	\begin{subfigure}[t]{0.49\textwidth}
		\includegraphics[trim = -8mm 0mm 0mm 0mm,clip,scale=0.23]{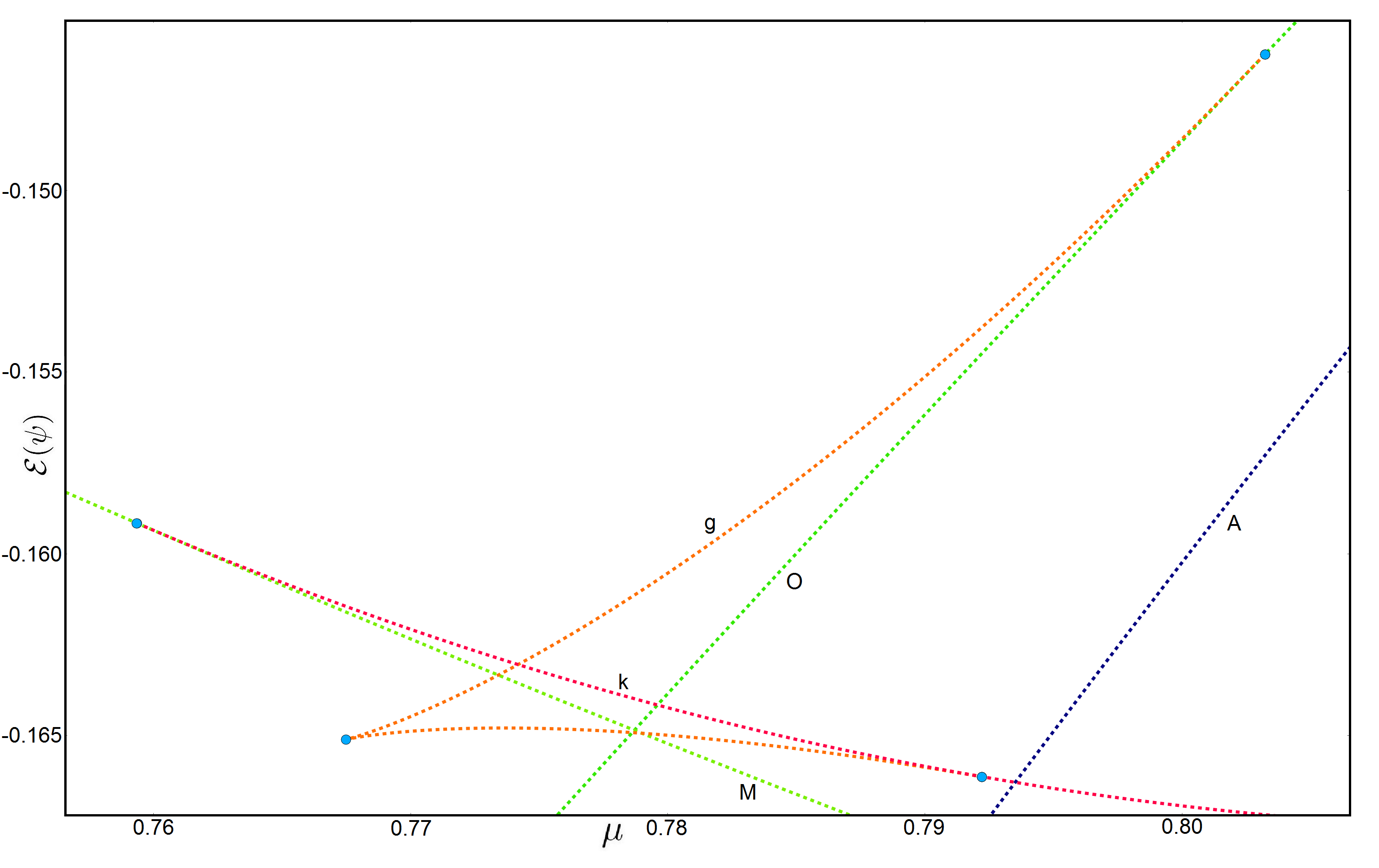}
		\caption{Branch g}
	\end{subfigure}
	\begin{subfigure}[t]{0.49\textwidth}
		\includegraphics[trim = 15mm 0mm 0mm 0mm,clip,scale=0.23]{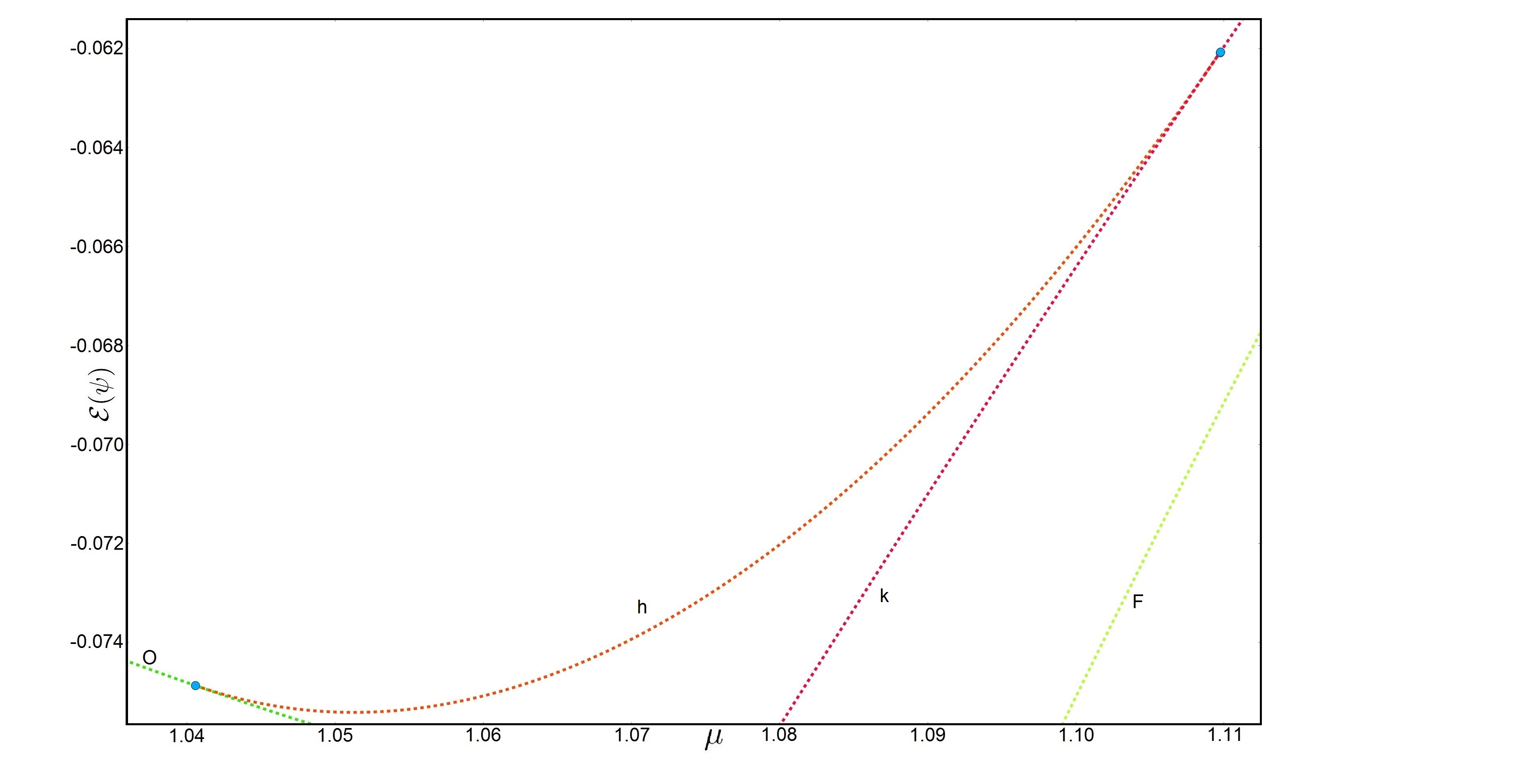}
		\caption{Branch h}
	\end{subfigure}
	\hfil
	\begin{subfigure}[t]{0.49\textwidth}
		\includegraphics[trim = -3mm 0mm 0mm 0mm,clip,scale=0.23]{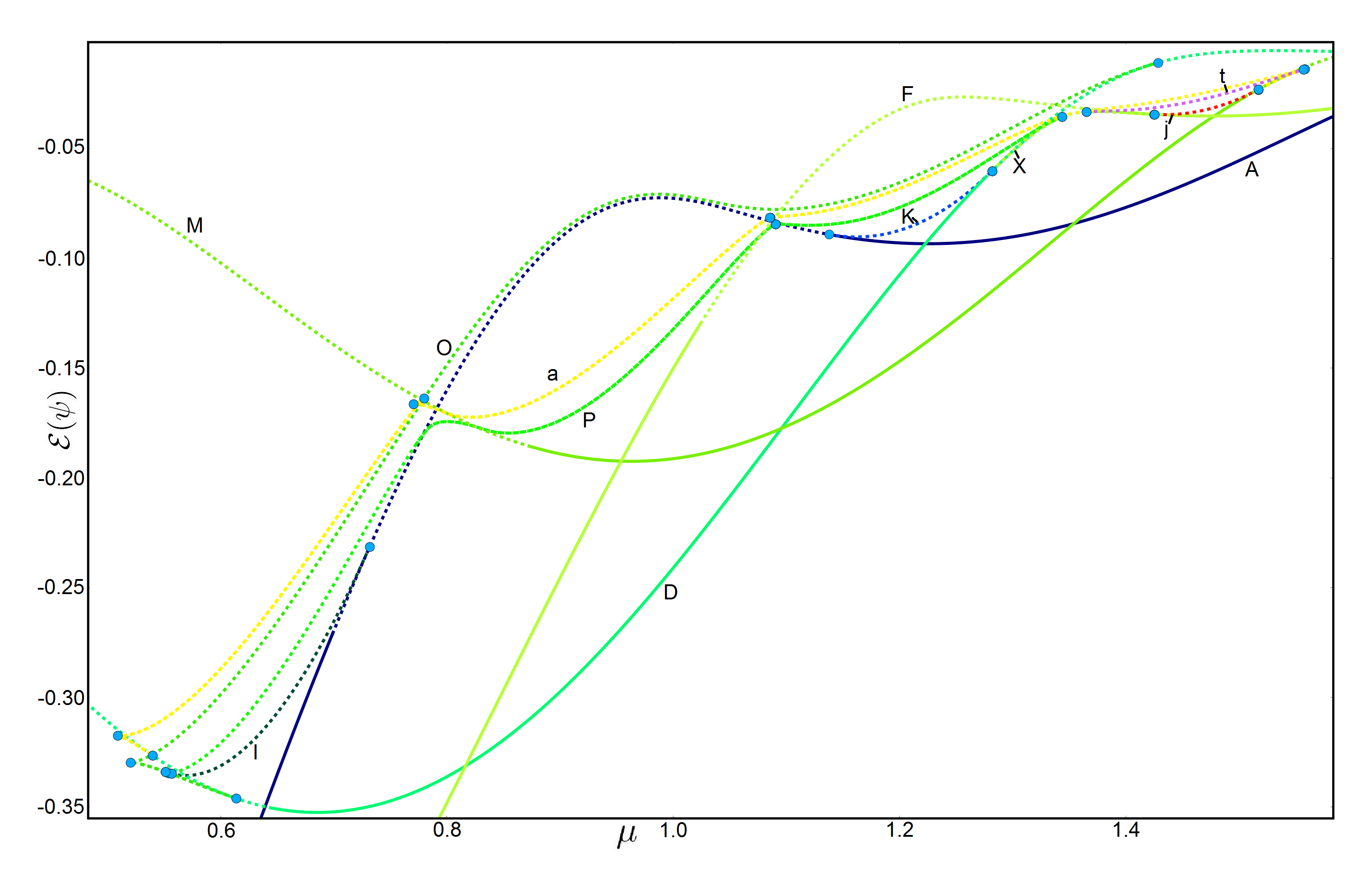}
		\caption{Branches I, K, O, P, X, a, j and t}
	\end{subfigure}
	\caption{Several close-ups of the bifurcation diagram of figure \ref{fig - main diagram square}, including connective solution branches.} \label{fig - bifurcation diagram square zooms}
\end{figure}

\begin{figure}
	\centering
	\includegraphics[trim = 0mm 0mm 0mm 0mm,clip,scale=0.4]{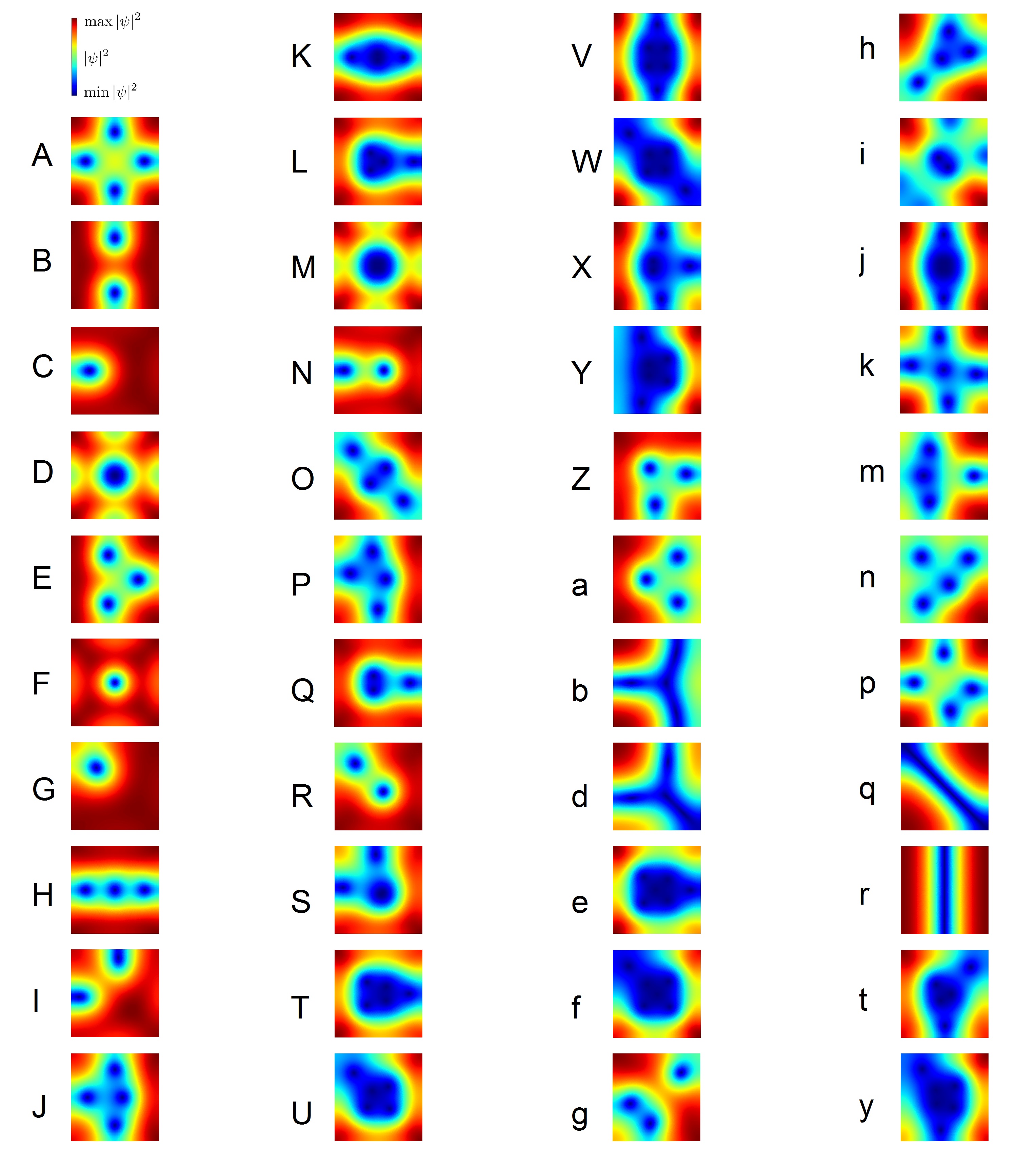}
	\caption{Representative solutions for the solution curves in the bifurcation diagram of a square-shaped material with edge 5.5.} \label{fig - representative square}
\end{figure}

\begin{table}[]
	\centering
	\caption{Approximate values for the applied magnetic fields strength ($\mu$) for the branch points displayed in figure \ref{fig - schematic square}.} \label{fig - table}
	\label{my-label}
	\begin{tabular}{|l|l|l|l|l|l|l|l|}
		\hline
		index & $\mu$ value & index & $\mu$ value & index & $\mu$ value & index & $\mu$ value \\ \hline
		1     & 0.6989      & 16    & 1.51736     & 31    & 1.02539     & 46    & 1.34396     \\ \hline
		2     & 0.7319      & 17    & 1.6149      & 32    & 1.4412      & 47    & 1.3655      \\ \hline
		3     & 1.13777     & 18    & 1.655       & 33    & 0.58275     & 48    & 1.425396    \\ \hline
		4     & 1.1423      & 19    & 0.09241    & 34    & 0.551370    & 49    & 1.9115      \\ \hline
		5     & 1.7176      & 20    & 0.54002     & 35    & 0.80322     & 50    & 1.425179    \\ \hline
		6     & 1.7437      & 21    & 0.646098    & 36    & 1.0406      & 51    & 1.7430      \\ \hline
		7     & 1.7959      & 22    & 0.61367     & 37    & 1.3712      & 52    & 1.7063      \\ \hline
		8     & 0.12905     & 23    & 1.281974    & 38    & 1.41573     & 53    & 1.75378     \\ \hline
		9     & 0.2591      & 24    & 1.27688     & 39    & 0.5566      & 54    & 0.7922136   \\ \hline
		10    & 0.34209     & 25    & 1.42875     & 40    & 0.8639      & 55    & 1.10976     \\ \hline
		11    & 0.7594      & 26    & 1.7453      & 41    & 0.6980      & 56    & 0.761674    \\ \hline
		12    & 0.8029      & 27    & 1.7055      & 42    & 0.6308      & 57    & 1.312923    \\ \hline
		13    & 0.8635      & 28    & 1.5824      & 43    & 1.13779     & 58    & 1.517199    \\ \hline
		14    & 0.875147    & 29    & 0.25285     & 44    & 1.02538     & 59    & 1.557099    \\ \hline
		15    & 1.5048      & 30    & 1.0240      & 45    & 1.28195     & 60    & 1.67462     \\ \hline
	\end{tabular}
\end{table}

A total of $60$ branch points were discovered during the continuation, their values are given in table \ref{fig - table}. Several of these only contained a $1$-dimensional kernel for which algorithm \ref{algorithm Mei} yielded the (single) new tangent direction. Other branch points corresponded to a kernel with dimension $2$. Fore these points the tangent directions were constructed by application of the equivariant branching lemma, described in the last part of section \ref{sectie tangents}.

Note that except for the four main branches (A, D, F and M), solution branch B contains a stable part as well: it is stable for magnetic field strength values between approximately $0.6308$ and $0.646098$, for solutions that consist of two vortices very close to the center, on the horizontal or vertical axis. Branch B is the only non-main branch to contain a stable part, the stability of this region has been verified with a Crank-Nicolson time step method as well.

A thorough description of branches A to M can be found in \cite{Schlomer_square}. Though multiple new solution branches have been discovered in the current paper, their discussion will be omitted.

\section{Conclusion} \label{sectie conclusie}
This paper presents multiple tools for the automatic exploration of
solution landscapes for the extreme type-II Ginzburg-Landau
equation. One of the main challenges is the detection and
determination of bifurcation points. We illustrate that this can be
handled by tracing Ritz values that indicate a nearby bifurcation
point. This triggers a precise determination by an extended system
\eqref{bifurcatiepunt stelsel} that is solved with a Newton-Krylov
algorithm.

The second main challenge is the construction of the tangent
directions of emerging solution branches from these bifurcation points. The
complexity of this problem depends on the dimension of the Jacobian's
kernel associated with the bifurcation. For $1$-dimensional kernels,
algorithm \ref{algorithm Mei} is applied, which is equivalent to
solving the algebraic branching equation. When this kernel is
$2$-dimensional, we provided an adapted version of algorithm
\ref{algorithm Mei}, resulting in algorithm \ref{algoritme n=2}, to
construct the tangent directions. Though this last algorithm is
robust, a faster alternative based on the equivariant branching lemma
is described as well. This last method can, however, only be applied
for $D_m$ ($m\geq 4$) invariant bifurcation points.

The methods described in the current paper are implemented in Python
as part of the package PyNCT. Contrary to other tools, we make use of
sparse linear algebra in the algorithms. The bifurcation diagrams
provided in section \ref{sectie resulaten} are generated by our
methods, showing their robustness. Though we specifically used the
extreme type-II Ginzburg-Landau equations for the derivation of the
algorithms, the generalization to other ($D_m$- or $C_m$-symmetric)
problems is straightforward. For the PyNCT extension general
algorithms were implemented.

In the future, we aim to apply the methods presented to more
complicated superconducting systems. For example, the Ginzburg-Landau
equations with space-dependent coefficients or the one that describes
the states of type 1.5 superconductors. A similar preconditioner as
the one presented in section \ref{sectie numerieke continuatie} will
likely lead to a fast solver, which can be used in combination with
the current algorithms to provide automatically explored bifurcation
diagrams.

Further work is required to determine robust stopping criteria for
determination of bifurcation points and tangent directions. Although,
for the examples presented in section \ref{sectie resulaten}, the
desired tolerances were always achieved, they require a new tuning
when applied to other samples with different discretization
levels. Robust stopping criteria, that work automatically for various
discretizations, would be helpful in the future.

Finally, an extension towards $3$ dimensional materials is one of the
next steps. We expect that the current methods for detection and
determination of bifurcation points provided are straightforward
applicable. However, the construction of the tangent directions to
emerging solution branches will require further efforts.  An extra
case will probably exist for $3$ dimensional materials, where the
Jacobian associated with the bifurcation has a $3$-dimensional
kernel. We believe that a similar analysis as for the derivation of
algorithm \ref{algoritme n=2} is possible for this case, and the
equivariant branching lemma might again be applicable in some cases as
well.

\appendix
\section{Proof of algorithm \ref{algoritme n=2}} \label{app proof}
A sketch of the derivation of algorithm \ref{algoritme n=2} is provided in this section. A more detailed one is available at author's request. First we proof the following lemma:
\begin{lemma} \label{proof - lemma}
	Consider a system of equations with unknowns $\alpha_1$ and $\alpha_2$:
	\begin{equation} \label{proof - lemma general equation}
		\begin{cases}
			\sum_{i=0}^{k}\alpha_1^i\alpha_2^{k-i}a_i + b\alpha_1+ e\alpha_2 = 0, \\
			\sum_{i=0}^{k}\alpha_1^i\alpha_2^{k-i}c_i + d\alpha_2 = 0
		\end{cases}
	\end{equation}
	with $k\geq2$, $a_0,\dots,a_k$, $c_0,\dots,c_k$, $e$ $\in\mathbb{R}$ and $b,d\in\mathbb{R}_0$. This system contains non-isolated solutions if and only if the relations
	\begin{align}
		\begin{aligned}
			& d a_0 = ec_0, \\
			&c_k = 0, \\
			&\forall i=1\dots k: d a_i = b c_{i-1} +ec_i
		\end{aligned} \label{proof - lemma relations}
	\end{align}
	hold. Moreover, in this case the system has a single isolated solution, given by $(\alpha_1,\alpha_2)=(0,0)$.
\end{lemma}
\begin{proof} \text{ } \\
	\textbf{Step 1}: Suppose that the system of equations contains non-isolated solutions. Since this system can be regarded as the intersection of two algebraic curves, Cramer's theorem \cite{Cramer1750} can be applied to show that the curves must be degenerate. The degeneracy implies that
	\begin{align}
		&\exists 1\leq l\leq k-1, f^{(m)}_i (m=0\dots l, i=0\dots m), g^{(m)}_i (m=0\dots k-l, i=0\dots m), \nonumber \\
		&\hspace{10pt}h^{(m)}_i (m=0\dots k-l, i=0\dots m): \nonumber \\
		&\sum_{i=0}^{k}a_i\alpha_1^i\alpha_2^{k-i}+b\alpha_1+e\alpha_2 = \left(\sum_{m=0}^{l}\sum_{i=0}^{m}f^{(m)}_i\alpha_1^i\alpha_2^{m-i}\right)\left(\sum_{m=0}^{k-l}\sum_{i=0}^{m}g^{(m)}_i\alpha_1^i\alpha_2^{m-i}\right), \label{proof - curve 1} \\
		&\sum_{i=0}^{k}c_i\alpha_1^i\alpha_2^{k-i}+d\alpha_2 = \left(\sum_{m=0}^{l}\sum_{i=0}^{m}f^{(m)}_i\alpha_1^i\alpha_2^{m-i}\right)\left(\sum_{m=0}^{k-l}\sum_{i=0}^{m}h^{(m)}_i\alpha_1^i\alpha_2^{m-i}\right). \label{proof - curve 2}
	\end{align}
	
	\noindent Comparing the coefficients of $\alpha_1$ and $\alpha_2$ in \eqref{proof - curve 1} and \eqref{proof - curve 2}, it follows that $f^{(0)}_0\neq0$. Similar comparisons of other coefficients can be combined with an induction argument to prove the statements
	\begin{align}
		&\forall m=0\dots k-l: dg^{(m)}_0 = eh^{(m)}_0, \label{proof -step 3a}\\
		&\forall m = 0\dots k-l: h^{(m)}_m=0, \label{proof -step 3b}\\
		&\forall m =1\dots k-l, i=1\dots m: dg^{(m)}_i = bh^{(m)}_{i-1}+eh^{(m)}_i. \label{proof -step 3c}
	\end{align}
	
	\noindent Using \eqref{proof -step 3a}, we get
	$$
	da_0 = df^{(l)}_0g^{(k-l)}_0 = ef^{(l)}_0h^{(k-l)}_0 = ec_0.
	$$
	Next, \eqref{proof -step 3b} implies that
	$$
	c_k = f^{(l)}_lh^{(k-l)}_{k-l} = 0.
	$$
	For ease of notation, define
	\begin{align*}
		&\forall l+1\leq i\leq k: f^{(l)}_i = 0, \\
		&\forall k-l+1\leq i\leq k: g^{(k-l)}_i = h^{(k-l)}_i = 0.
	\end{align*}
	Since $h^{(k-l)}_{k-l}=0$, we have
	$$
	\forall i=1\dots k: dg^{(k-l)}_i = bh^{(k-l)}_{i-1}+eh^{(k-l)}_i.
	$$
	Together with \eqref{proof -step 3a},\eqref{proof -step 3b} and \eqref{proof -step 3c}, this implies ($\forall i=1\dots k$)
	\begin{align*}
		da_i &= d\sum_{j=0}^{i}f^{(l)}_jg^{(k-l)}_{i-j} = b\sum_{j=0}^{i-1}f^{(l)}_jh^{(k-l)}_{i-1-j} + e\sum_{j=0}^{i}f^{(l)}_jh^{(k-l)}_{i-j} = bc_{i-1}+ec_i.
	\end{align*}
	
	\noindent \textbf{Step 2}: We have yet to show that
	\begin{align*}
		&da_0=ec_0, c_k=0, \forall i=1\dots k: da_i=bc_{i-1}+ec_i \\
		&\Rightarrow \text{The system of equations \eqref{proof - lemma general equation} contains non-isolated solutions and} \\
		& \hspace{20pt}\text{ a single isolated solution given by $(\alpha_1,\alpha_2)=(0,0)$}.
	\end{align*}
	We first rewrite the equations for the algebraic curves. The one for the first curve becomes
	\begin{align*}
		&\sum_{i=0}^{k}a_i\alpha_1^i\alpha_2^{k-i}+b\alpha_1+e\alpha_2 = 0
		\iff \left(\sum_{i=0}^{k-1}c_i\alpha_1^i\alpha_2^{k-1-i}+d\right)(b\alpha_1+e\alpha_2) = 0.
	\end{align*}
	The equation for the second curve can be rewritten as
	\begin{align*}
		\sum_{i=0}^{k}c_i\alpha_1^i\alpha_2^{k-i}+d\alpha_2 = 0
		\iff \left(\sum_{i=0}^{k-1}c_i\alpha_1^i\alpha_2^{k-1-i}+d\right)\alpha_2 = 0.
	\end{align*}
	Solutions $(\alpha_1,\alpha_2)$ of system \eqref{proof - lemma general equation} satisfy
	\begin{displaymath}
		b\alpha_1+e\alpha_2=0 \text{ and } \alpha_2=0 \text{ or } \sum_{i=0}^{k-1}c_i\alpha_1^i\alpha_2^{k-1-i}+d=0.
	\end{displaymath}
	The equation $\sum_{i=0}^{k-1}c_i\alpha_1^i\alpha_2^{k-1-i}+d=0$ again describes an algebraic curve. The solutions that satisfy this equation hence form a continuous curve and are non-isolated. The solution for which
	$$
	b\alpha_1+e\alpha_2=0 \text{ and } \alpha_2=0
	$$
	satisfies $(\alpha_1,\alpha_2)=(0,0)$. This solution is isolated since $(\alpha_1,\alpha_2)=(0,0)$ does not belong to the curve described by $\sum_{i=0}^{k-1}c_i\alpha_1^i\alpha_2^{k-1-i}+d=0$ (because $d\neq0$).
\end{proof}

The three equations 
\begin{align}
	&\mathcal{J}_\psi^{(b)}w_k+\mathcal{J}_\mu^{(b)}\beta_k = -r_k, \label{proof - vergelijkingen van paper Mei 1} \\
	&\forall j=1,2: \langle\phi_j,w_k\rangle_\mathbb{R}=0, \label{proof - vergelijkingen van paper Mei 2} \\
	&\forall j=1,2: \langle\phi_j^*,r_k\rangle_\mathbb{R}=0, \label{proof - vergelijkingen van paper Mei 3}
\end{align}
with $k=1\dots l$, are crucial for the derivation of algorithm \ref{algoritme n=2}. These equations were derived in section \ref{sectie tangents}.
The terms $w_k$ and $\beta_k$ appear in the order $l$ Taylor expansion of $\psi(s)$ and $\mu(s)$ \eqref{Taylor expansie psi en mu}. $\phi_1$ and $\phi_2$ are null vectors of $\mathcal{J}_\psi^{(b)}$, $\phi_1^*$ and $\phi_2^*$ of $\mathcal{J}_\psi^{(b)^*}$. The terms $r_k$ are given by the formula
\begin{align}
	&r_k = \sum_{j=2}^{k}\frac{1}{j!}\hspace{20pt}\sum_{\mathclap{\substack{\sum_{p=1}^{j}k_p=k \\ \forall p=1,\dots, j: \\ k_p\in\{1,\dots, k-j+1\}}}}\hspace{10pt}D^j\mathcal{GL}^{(b)}
	\begin{pmatrix}
		x_{k_1} \\
		\beta_{k_1}
	\end{pmatrix}\begin{pmatrix}
		x_{k_2} \\
		\beta_{k_2}
	\end{pmatrix}\dots\begin{pmatrix}
		x_{k_j} \\
		\beta_{k_j}
	\end{pmatrix} \label{proof - formule rk}\\
	&\text{with } x_1=\sum_{i=1}^{2}\alpha_i\phi_i+w_1 \qquad \forall p=2,\dots,k-1: x_p=w_p. \nonumber
\end{align}
The terms $v^{(0)}$, $t_1$, $t_2$, $t_3$, $b^{(1)}$, $b^{(2)}$, $y_i^{(k)}$, $a_i^{(k,1)}$, $a_i^{(k,2)}$, $\kappa_i^{(k)}$, $z_i^{(k)}$ and $q_i^{(k)}$ ($k=1\dots l, i=0\dots k$) that appear in the proof are defined as in algorithm \ref{algoritme n=2}.

\begin{proof}
	\textbf{Initial}: Together with $r_1=0$, equation \eqref{proof - vergelijkingen van paper Mei 1} implies
	\begin{equation}
		\mathcal{J}_\psi^{(b)}w_1 = -\mathcal{J}_\mu^{(b)}\beta_1.
	\end{equation}
	Define $v^{(0)}$ as
	\begin{displaymath}
		\mathcal{J}_\psi^{(b)}v^{(0)} = -\mathcal{J}_\mu^{(b)} \qquad v^{(0)}\in\text{im}\left(\mathcal{J}_\psi^{(b)^*}\right),
	\end{displaymath}
	this leads to the following expressions for $w_1$ and $x_1$:
	\begin{align}
		&w_1 = \beta_1v^{(0)}, \\
		&x_1= \alpha_1\phi_1+\alpha_2\phi_2+\beta_1v^{(0)}. \label{proof - uitdrukking x1}
	\end{align}
	Note that the condition $v^{(0)}\in\text{im}\left(\mathcal{J}_\psi^{(b)^*}\right)$ is necessary to satisfy \eqref{proof - vergelijkingen van paper Mei 2}. Formula \eqref{proof - formule rk} for $k=2$ yields
	\begin{displaymath}
		r_2 = \frac{1}{2}D^2\mathcal{GL}^{(b)}\begin{pmatrix}
			x_1 \\
			\beta_1
		\end{pmatrix}\begin{pmatrix}
			x_1 \\
			\beta_1
		\end{pmatrix}.
	\end{displaymath}
	Substitution of $x_1$ and bilinearity of the Hessian lead to the expression
	\begin{equation}
		r_2 = \sum_{i=0}^{2}\alpha_1^i\alpha_2^{2-i}y_i^{(2)} + \beta_1\sum_{i=1}^{2}\alpha_i t_i + \beta_1^2t_3, \label{proof - uitdrukking r2}
	\end{equation}
	with $y_0^{(2)}$, $y_1^{(2)}$, $y_2^{(2)}$, $t_1$, $t_2$ and $t_3$ defined in algorithm \ref{algoritme n=2}.
	Substitution of this expression in equation \eqref{proof - vergelijkingen van paper Mei 3} leads to the system
	\begin{equation}
		\begin{cases}
			\sum_{i=0}^{2}\alpha_1^i\alpha_2^{2-i}\langle \phi_1^*,y_i^{(2)}\rangle + \beta_1\sum_{i=1}^{2}\alpha_i\langle \phi_1^*,t_i\rangle + \beta_1^2 \langle \phi_1^*,t_3\rangle = 0, \\
			\sum_{i=0}^{2}\alpha_1^i\alpha_2^{2-i}\langle \phi_2^*,y_i^{(2)}\rangle + \beta_1\sum_{i=1}^{2}\alpha_i\langle \phi_2^*,t_i\rangle + \beta_1^2 \langle \phi_2^*,t_3\rangle = 0.
		\end{cases} \label{proof - initial reduced system 1}
	\end{equation}
	Together with the fundamental theorem of linear algebra, assumptions \eqref{extra aanname in LS reductie 1}, \eqref{extra aanname in LS reductie 1.5}, \eqref{extra aanname in LS reductie 1.8} and \eqref{extra aanname in LS reductie 2} imply that we can choose $\phi_1^*$ and $\phi_2^*$ such that
	\begin{align*}
		&\langle\phi_1^*,t_1\rangle \neq 0, &&  && \langle\phi_1^*,t_3\rangle = 0, \\
		&\langle\phi_2^*,t_1\rangle = 0, && \langle\phi_2^*,t_2\rangle \neq 0, && \langle\phi_2^*,t_3\rangle = 0. \\
	\end{align*}
	With these choices, \eqref{proof - initial reduced system 1} becomes
	\begin{equation}
		\begin{cases}
			\sum_{i=0}^{2}\alpha_1^i\alpha_2^{2-i}a_i^{(2,1)} +\beta_{1}\alpha_1b^{(1)} +\beta_{1}\alpha_2b^{(3)} = 0, \\
			\sum_{i=0}^{2}\alpha_1^i\alpha_2^{2-i}a_i^{(2,2)} +\beta_{1}\alpha_2b^{(2)} = 0,
		\end{cases} \label{proof - initial reduced system 2}
	\end{equation}
	where $b^{(1)}$, $b^{(2)}$, $b^{(3)}$, $a_0^{(2,1)}$, $a_1^{(2,1)}$, $a_2^{(2,1)}$, $a_0^{(2,2)}$, $a_1^{(2,2)}$ and $a_2^{(2,2)}$ are defined as in algorithm \ref{algoritme n=2}. Lemma \ref{proof - lemma} can now be applied to check this system for non-isolated solutions. If only isolated solutions exist, the algorithm is stopped. In the other case, we have 
	\begin{align*}
		b^{(2)}a_0^{(2,1)} = b^{(3)}a_0^{(2,2)}, && a_2^{(2,2)}=0 &&\forall i=1;2: b^{(2)}a_i^{(2,1)} = b^{(1)}a_{i-1}^{(2,2)} + b^{(3)}a_i^{(2,2)}
	\end{align*}
	and continue by solving \eqref{proof - initial reduced system 2} for $\beta_1$:
	\begin{align*}
		\sum_{i=0}^{1}\alpha_1^i\alpha_2^{1-i}a_i^{(2,2)} + \beta_1b^{(2)} = 0 \iff \beta_1 = \sum_{i=0}^{1}\alpha_1^i\alpha_2^{1-i}\kappa_i^{(1)},
	\end{align*}
	with $\kappa_0^{(1)}$ and $\kappa_1^{(1)}$ defined in algorithm \ref{algoritme n=2}. Substitution of $\beta_1$ in \eqref{proof - uitdrukking r2} and \eqref{proof - uitdrukking x1} eventually leads to the following expressions for $r_2$ and $x_1$:
	\begin{align*}
		&r_2 = \sum_{i=0}^{2}\alpha_1^i\alpha_2^{2-i}z_i^{(2)}, && x_1 = \sum_{i=0}^{1}\alpha_1^i\alpha_2^{1-i}q_i^{(1)},
	\end{align*}
	with $z_0^{(2)}$, $z_1^{(2)}$, $z_2^{(2)}$, $q_0^{(1)}$, $q_1^{(1)}$ defined in algorithm \ref{algoritme n=2}.
	
	\textbf{Iteration}:
	Let $k\geq 3$ and assume terms $z_i^{(k-1)}\in\mathbb{C}^n$ (for $i=0\dots k-1$), $q_i^{(j)}\in\mathbb{C}^n$ (for $j=1\dots k-2,i=0\dots j$) and $\kappa_i^{(j)}\in\mathbb{R}$ (for $j=1\dots k-2,i=0\dots j$) have been derived such that
	\begin{align}
		&r_{k-1} = \sum_{i=0}^{k-1}\alpha_1^i\alpha_2^{k-1-i}z_i^{(k-1)}, \label{proof - induction hypothesis 1} \\
		&\forall j=1\dots k-2: x_j = \sum_{i=0}^{j}\alpha_1^i\alpha_2^{j-i}q_i^{(j)}, \label{proof - induction hypothesis 2} \\
		&\forall j=1\dots k-2: \beta_j = \sum_{i=0}^{j}\alpha_1^i\alpha_2^{j-i}\kappa_i^{(j)}. \label{proof - induction hypothesis 3}
	\end{align}
	Equation \eqref{proof - vergelijkingen van paper Mei 1} implies
	\begin{equation}
		\mathcal{J}_\psi^{(b)}w_{k-1} = -\sum_{i=0}^{k-1}\alpha_1^i\alpha_2^{k-1-i}z_i^{(k-1)} -\mathcal{J}_\mu^{(b)}\beta_{k-1}.
	\end{equation}
	Define $v_0^{(k-1)}$\dots $v_{k-1}^{(k-1)}$ as
	\begin{displaymath}
		\forall i=0\dots k-1: \mathcal{J}_\psi^{(b)}v_i^{(k-1)} = -z_i^{(k-1)}  \qquad v_i^{(k-1)}\in\text{im}\left(\mathcal{J}_\psi^{(b)^*}\right),
	\end{displaymath}
	this leads to the following expression for $w_{k-1}$ ($=x_{k-1}$):
	\begin{align}
		&x_{k-1} = w_{k-1} = \sum_{i=0}^{k-1}\alpha_1^i\alpha_2^{k-1-i}v_i^{(k-1)} + \beta_{k-1}v^{(0)}. \label{proof - uitdrukking xk}
	\end{align}
	Rewriting formula \eqref{proof - formule rk} after substitution of \eqref{proof - induction hypothesis 2}, \eqref{proof - induction hypothesis 3} and \eqref{proof - uitdrukking xk}, and defining the terms $y_0^{(k)}$,\dots,$y_k^{(k)}$ as in algorithm \ref{algoritme n=2}, one can show that
	\begin{align*}
		r_k =& \sum_{i=0}^{k}\alpha_1^i\alpha_2^{k-i}y_i^{(k)} +\beta_{k-1}\alpha_1\mathcal{H}^{(b)}\begin{pmatrix}
			q_1^{(1)} \\
			\kappa_1^{(1)}
		\end{pmatrix} \begin{pmatrix}
			v^{(0)} \\
			1
		\end{pmatrix} + \beta_{k-1}\alpha_2\mathcal{H}^{(b)}\begin{pmatrix}
			q_0^{(1)} \\
			\kappa_0^{(1)}
		\end{pmatrix} \begin{pmatrix}
			v^{(0)} \\
			1
		\end{pmatrix}.
	\end{align*}
	The coefficients of $\beta_{k-1}\alpha_1$ and $\beta_{k-1}\alpha_2$ can respectively be rewritten as $t_1+2\kappa_1^{(1)}t_3$ and $t_2+2\kappa_0^{(1)}t_3$. This leads to the following expression for $r_k$:
	\begin{equation}
		r_k = \sum_{i=0}^{k}\alpha_1^i\alpha_2^{k-i}y_i^{(k)} + \beta_{k-1}\alpha_1\left(t_1 + 2\kappa_1^{(1)}t_3\right) + \beta_{k-1}\alpha_2\left(t_2 + 2\kappa_0^{(1)}t_3\right). \label{proof - uitdrukking rk}
	\end{equation}
	Substitution of this expression in equation \eqref{proof - vergelijkingen van paper Mei 3} leads to the system
	\begin{equation}\begin{cases}
			\sum_{i=0}^{k}\alpha_1^i\alpha_2^{k-i}\langle\phi_1^*,y_i^{(k)}\rangle +\beta_{k-1}\alpha_1\left(\langle\phi_1^*,t_1\rangle+2\kappa_1^{(1)}\langle\phi_1^*,t_3\rangle\right) \\
			\hspace{150pt}+\beta_{k-1}\alpha_2\left(\langle\phi_1^*,t_2\rangle+2\kappa_0^{(1)}\langle\phi_1^*,t_3\rangle\right) = 0, \\
			\sum_{i=0}^{k}\alpha_1^i\alpha_2^{k-i}\langle\phi_2^*,y_i^{(k)}\rangle +\beta_{k-1}\alpha_1\left(\langle\phi_2^*,t_1\rangle+2\kappa_1^{(1)}\langle\phi_2^*,t_3\rangle\right) \\
			\hspace{150pt}+\beta_{k-1}\alpha_2\left(\langle\phi_2^*,t_2\rangle+2\kappa_0^{(1)}\langle\phi_2^*,t_3\rangle\right) = 0.
		\end{cases} \label{proof - iteration reduced system 1}
	\end{equation}
	Remember that
	\begin{align*}
		&b^{(1)} = \langle\phi_1^*,t_1\rangle \neq 0, && b^{(3)} = \langle\phi_1^*,t_2\rangle, && \langle\phi_1^*,t_3\rangle = 0, \\
		&\langle\phi_2^*,t_1\rangle = 0, && b^{(2)}=\langle\phi_2^*,t_2\rangle\neq0, && \langle\phi_2^*,t_3\rangle = 0. \\
	\end{align*}
	Equation \eqref{proof - iteration reduced system 1} hence becomes
	\begin{equation}
		\begin{cases}
			\sum_{i=0}^{k}\alpha_1^i\alpha_2^{k-i}a_i^{(k,1)} +\beta_{k-1}\alpha_1b^{(1)} + \beta_{k-1}\alpha_2b^{(3)} = 0, \\
			\sum_{i=0}^{k}\alpha_1^i\alpha_2^{k-i}a_i^{(k,2)} +\beta_{k-1}\alpha_2b^{(2)} = 0,
		\end{cases} \label{proof - iteration reduced system 2}
	\end{equation}
	where $a_0^{(k,1)}$,\dots, $a_k^{(k,1)}$ and $a_0^{(k,2)}$,\dots, $a_k^{(k,2)}$ are defined as in algorithm \ref{algoritme n=2}. Lemma \ref{proof - lemma} can again be applied to check this system for non-isolated solutions. If only isolated solutions exist, the algorithm is stopped. In the other case, we have
	\begin{align*}
		b^{(2)}a_0^{(k,1)} = b^{(3)}a_0^{(k,2)}, && a_k^{(k,2)}=0, &&\forall i=1,\dots,k: b^{(2)}a_i^{(k,1)} = b^{(1)}a_{i-1}^{(k,2)} + b^{(3)}a_i^{(k,2)}
	\end{align*}
	and continue by solving \eqref{proof - iteration reduced system 2} for $\beta_{k-1}$:
	\begin{align*}
		\sum_{i=0}^{k-1}\alpha_1^i\alpha_2^{k-1-i}a_i^{(k,2)} + \beta_{k-1}b^{(2)} = 0 \iff \beta_{k-1} = \sum_{i=0}^{k-1}\alpha_1^i\alpha_2^{k-1-i}\kappa_i^{(k-1)},
	\end{align*}
	with $\kappa_0^{(k-1)}$,\dots, $\kappa_{k-1}^{(k-1)}$ defined in algorithm \ref{algoritme n=2}. Substitution of $\beta_{k-1}$ in \eqref{proof - uitdrukking rk} and \eqref{proof - uitdrukking xk} eventually leads to the following expressions for $r_k$ and $x_{k-1}$:
	\begin{align*}
		&r_k = \sum_{i=0}^{k}\alpha_1^i\alpha_2^{k-i}z_i^{(k)}, && x_{k-1} = w_{k-1} = \sum_{i=0}^{k-1}\alpha_1^i\alpha_2^{k-1-i}q_i^{(k-1)},
	\end{align*}
	with $z_0^{(k)}$, \dots, $z_k^{(k)}$ and $q_0^{(k-1)}$,\dots, $q_{k-1}^{(k-1)}$ defined in algorithm \ref{algoritme n=2}. With the new expressions for $r_{k}$, $x_{k-1}$ and $\beta_{k-1}$ now available, the iteration can be continued for $k\rightarrow k+1$.
\end{proof}

\section{Calculation of the $y_i^{(k)}$ terms in algorithm \ref{algoritme n=2}} \label{app calc}

\noindent An important step in the execution of algorithm \ref{algoritme n=2} is the calculation of the terms

\begin{align*}
	&y_i^{(k)} = \sum_{\mathclap{\substack{i_1+i_2=i \\ i_1\in\{0;1\} \\ i_2\in\{0,\dots, k-1\}}}} \mathcal{H}^{(b)}
	\begin{pmatrix}
		q^{(1)}_{i_1} \\
		\kappa^{(1)}_{i_1}
	\end{pmatrix}\begin{pmatrix}
		v^{(k-1)}_{i_2} \\
		0
	\end{pmatrix} +\frac{1}{2} \hspace{20pt} \sum_{\mathclap{\substack{k_1+k_2=k \\ k_1,k_2\in\{2,\dots, k-2\}}}} \hspace{40pt} \sum_{\mathclap{\substack{i_1+i_2=i \\ i_1\in\{0,\dots, k_1\} \\ i_2\in\{0,\dots, k_2\}}}} \mathcal{H}^{(b)}
	\begin{pmatrix}
		q^{(k_1)}_{i_1} \\
		\kappa^{(k_1)}_{i_1}
	\end{pmatrix}\begin{pmatrix}
		q^{(k_2)}_{i_2} \\
		\kappa^{(k_2)}_{i_2}
	\end{pmatrix} \\
	&\hspace{30pt}+\sum_{j=3}^{k}\frac{1}{j!} \hspace{20pt} \sum_{\mathclap{\substack{\sum_{p=1}^{j}k_p=k  \\ \forall p=1,\dots, j: \\ k_p\in\{1,\dots, k-j+1\}}}} \hspace{40pt} \sum_{\mathclap{\substack{\sum_{p=1}^{j}i_p=i  \\ \forall p=1\dots j: \\ i_p\in\{0,\dots, k_p\}}}} D^j\mathcal{GL}^{(b)}
	\begin{pmatrix}
		q^{(k_1)}_{i_1} \\
		\kappa^{(k_1)}_{i_1}
	\end{pmatrix}\begin{pmatrix}
		q^{(k_2)}_{i_2} \\
		\kappa^{(k_2)}_{i_2}
	\end{pmatrix}\dots\begin{pmatrix}
		q^{(k_j)}_{i_j} \\
		\kappa^{(k_j)}_{i_j}
	\end{pmatrix}
\end{align*}
for a given $k\geq 3$, $0\leq i \leq k$. The terms $q^{(m)}_l$, $\kappa^{(m)}_l$ and $v^{(k-1)}_l$ ($\forall m=1,\dots,k-2$, $l=0,\dots,m$) have been calculated in previous steps of the algorithm.
In this appendix we derive an efficient way to calculate these $y_i^{(k)}$ terms. For simplicity, we define ($\forall m=1,\dots,k-2$)
$$
q_{-1}^{(m)} = q_{m+1}^{(m)} = 0, \qquad \kappa_{-1}^{(m)} = \kappa_{m+1}^{(m)} = 0, \qquad v_{-1}^{(k-1)} = v_{k}^{(k-1)} = 0.
$$
Given a choice for $k$ and $j$, algorithm \ref{algoritme k indices} constructs a list $\mathcal{K}$ of possible index sets $K=(k_1,\dots,k_j)$ that satisfy
\begin{equation} \label{app - k indices}
	\sum_{p=1}^{j}k_p = k \qquad \forall p=1,\dots,j: k_p\in\{1,\dots,k-j+1\}
\end{equation}
The list is constructed in such a way that no two index sets are permutations of one another.

\begin{algorithm}
	\caption{Calculation of list that satisfies \eqref{app - k indices}} 
	\label{algoritme k indices}
	\textnormal{INPUT}: $k,j\in\mathbb{N}$ \\
	\textnormal{OUTPUT}: List $\mathcal{K}$ of index sets that satisfy \eqref{app - k indices}
	\begin{tabbing}
		\textnormal{1}: \hspace{2pt}\= Initialize the list of possible index sets: $\mathcal{K} = [\text{ }]$. \\
		\textnormal{2}: \> Initialize an index set: $K = (\text{ })$ \\
		\textnormal{3}: \> $m = length(K)$ \\
		\textnormal{4}: \> If $m<j-1$: \\
		\textnormal{5}: \> \hspace{10pt} \= If $m\neq 0$: \\
		\textnormal{6}: \> \> \hspace{10pt} \= $p_0=K_m$ \\
		\textnormal{7}: \> \> Else: \\
		\textnormal{8}: \> \> \> $p_0=1$ \\
		\textnormal{9}: \> \> End If \\
		\textnormal{10}: \> \> For $p=p_0,\dots,k-\sum_{l=1}^{m}K_l-j+m+1$: \\
		\textnormal{11}: \> \> \> Create a copy $K'$ of the set $K$. \\
		\textnormal{12}: \> \> \> $K' = (K', p)$ \\
		\textnormal{13}: \> \> \> Go to step $3$, using $K'$ instead of $K$. \\
		\textnormal{14}: \> \> End For \\
		\textnormal{15}: \> Else: \\
		\textnormal{16}: \> \> $p = k-\sum_{l=1}^{m}K_l$ \\
		\textnormal{17}: \> \> If $K_{j-1}\leq p$: \\
		\textnormal{18}: \> \> \> $K = (K,p)$ \\
		\textnormal{19}: \> \> \> Update list: $\mathcal{K} = [\mathcal{K},K]$ \\
		\textnormal{20}: \> \> End If \\
		\textnormal{21}: \> End If
	\end{tabbing}
\end{algorithm}

Algorithm \ref{algoritme i indices} is similar to algorithm \ref{algoritme k indices}. Given a choice for $j$ and $i$, and a set of indices $K=(k_1,\dots,k_j)$, it constructs a list $\mathcal{I}$ of index sets $I=(i_1,\dots,i_j)$ that satisfy
\begin{equation} \label{app - i indices}
	\sum_{p=1}^{j}i_p = i \qquad \forall p=1,\dots,j: i_p\in\{0,\dots,k_p\}
\end{equation}
This time the list is constructed such that no two sets of pairs of the form ($(k_1,i_1),\dots,(k_j,i_j)$) are permutations.

\begin{algorithm}
	\caption{Calculation of list that satisfies \eqref{app - i indices}} 
	\label{algoritme i indices}
	\textnormal{INPUT}: $i,j\in\mathbb{N}$, index set $K$ that satisfies \eqref{app - k indices}\\
	\textnormal{OUTPUT}: List $\mathcal{I}$ of index sets that satisfy \eqref{app - i indices}
	\begin{tabbing}
		\textnormal{1}: \hspace{2pt}\= Initialize the list of possible index sets: $\mathcal{I} = [\text{ }]$. \\
		\textnormal{2}: \> Initialize an index set: $I = (\text{ })$ \\
		\textnormal{3}: \> $m = length(I)$ \\
		\textnormal{4}: \> If $m<j-1$: \\
		\textnormal{5}: \> \hspace{10pt} \= If $m\neq 0$ and $K_{m+1}=K_m$: \\
		\textnormal{6}: \> \> \hspace{10pt} \= $p_0=I_m$ \\
		\textnormal{7}: \> \> Else: \\
		\textnormal{8}: \> \> \> $p_0=0$ \\
		\textnormal{9}: \> \> End If \\
		\textnormal{10}: \> \> For $p=p_0,\dots,\min(K_{m+1},i-\sum_{l=1}^{m}I_l)$: \\
		\textnormal{11}: \> \> \> Create a copy $I'$ of the set $I$. \\
		\textnormal{12}: \> \> \> $I' = (I', p)$ \\
		\textnormal{13}: \> \> \> Go to step $3$, using $I'$ instead of $I$. \\
		\textnormal{14}: \> \> End For \\
		\textnormal{15}: \> Else: \\
		\textnormal{16}: \> \> $p = i-\sum_{l=1}^{m}I_l$ \\
		\textnormal{17}: \> \> If $p\leq K_j$: \\
		\textnormal{18}: \> \> \> $I = (I,p)$ \\
		\textnormal{19}: \> \> \> If $K_{j-1}\neq K_j$ or $I_{j-1}\leq I_j$: \\
		\textnormal{20}: \> \> \> \hspace{10pt} \= Update list: $\mathcal{I} = [\mathcal{I},I]$ \\
		\textnormal{21}: \> \> \> End If \\
		\textnormal{22}: \> \> End If \\
		\textnormal{23}: \> End If
	\end{tabbing}
\end{algorithm}

Given the choice for $k$ and $i$, algorithm \ref{algoritme y} describes a possible way to calculate the term $y_i^{(k)}$. Algorithms \ref{algoritme k indices} and \ref{algoritme i indices} are used to construct lists of index sets required for the summation of higher-order derivative applications. It is possible to provide a bound on the order of these derivatives to speed up the process.

\begin{algorithm} 
	\caption{Calculation of $y_i^{(k)}$ in algorithm \ref{algoritme n=2}}
	\label{algoritme y}
	\textnormal{INPUT}: $i,k\in\mathbb{N}$, upper bound $b\in\mathbb{N}$ on order of derivatives\\
	\textnormal{OUTPUT}: The term $y_i^{(k)}$
	\begin{tabbing}
		\textnormal{1}: \hspace{2pt}\= Set $y_i^{(k)} = \mathcal{H}^{(b)}
		\begin{pmatrix}
		q^{(1)}_{0} \\
		\kappa^{(1)}_{0}
		\end{pmatrix}
		\begin{pmatrix}
		v^{(k-1)}_{i} \\
		0
		\end{pmatrix}
		+
		\mathcal{H}^{(b)}
		\begin{pmatrix}
		q^{(1)}_{1} \\
		\kappa^{(1)}_{1}
		\end{pmatrix}
		\begin{pmatrix}
		v^{(k-1)}_{i-1} \\
		0
		\end{pmatrix}
		$ \\
		\textnormal{2}: \> For $k_1=2,\dots,k-2$: \\
		\textnormal{3}: \> \hspace{10pt} \= $k_2 = k-k_1$ \\
		\textnormal{4}: \> \> For $i_1=0,\dots,\min(k_1,i)$: \\
		\textnormal{5}: \> \> \hspace{10pt} \= $i_2 = i-i_1$ \\
		\textnormal{6}: \> \> \> If $i_2\leq k_2$ and $k_1\leq k_2$ and ($k_1\neq k_2$ or $i_1\leq i_2$): \\
		\textnormal{7}: \> \> \> \hspace{10pt} \= $\tilde{y} = \mathcal{H}^{(b)}
		\begin{pmatrix}
		q^{(k_1)}_{i_1} \\
		\kappa^{(k_1)}_{i_1}
		\end{pmatrix}\begin{pmatrix}
		q^{(k_2)}_{i_2} \\
		\kappa^{(k_2)}_{i_2}
		\end{pmatrix}
		$ \\
		\textnormal{8}: \> \> \> \> $y_i^{(k)} = y_i^{(k)} + c \tilde{y}$ with $c=\frac{1}{2}$ if both $k_1=k_2$ and $i_1=i_2$, $c=1$ otherwise \\
		\textnormal{9}: \> \> \> End If \\
		\textnormal{10}: \>\> End For \\
		\textnormal{11}: \> End For \\
		\textnormal{12}: \> For $j=3,\dots,\min(k,b)$: \\
		\textnormal{13}: \>\> Execute algorithm \ref{algoritme k indices} to construct the list $\mathcal{K}$ \\
		\textnormal{14}: \>\> For each set $K$ of indices in $\mathcal{K}$: \\
		\textnormal{15}: \>\>\> Execute algorithm \ref{algoritme i indices} to construct the list $\mathcal{I}$ \\
		\textnormal{16}: \>\>\> For each set $I$ of indices in $\mathcal{I}$: \\
		\textnormal{17}: \>\>\>\> Calculate the number $c$ of possible permutations of the set $((K_1,I_1),\dots,(K_j,I_j))$ \\
		\textnormal{18}: \>\>\>\> $y_i^{(k)} = y_i^{(k)} + \frac{c}{j!}D^j\mathcal{GL}^{(b)}
		\begin{pmatrix}
		q^{(K_1)}_{I_1} \\
		\kappa^{(K_1)}_{I_1}
		\end{pmatrix}\begin{pmatrix}
		q^{(K_2)}_{I_2} \\
		\kappa^{(K_2)}_{I_2}
		\end{pmatrix}\dots\begin{pmatrix}
		q^{(K_j)}_{I_j} \\
		\kappa^{(K_j)}_{I_j}
		\end{pmatrix}$ \\
		\textnormal{19}: \>\>\> End For \\
		\textnormal{20}: \>\> End For \\
		\textnormal{21}: \> End For
	\end{tabbing}
\end{algorithm}

\section*{Acknowledgments}
The authors would like to thank Jacques Tempere, Wout Van Alphen, Nico Schl\"{o}mer, Przemyslaw Klosiewicz and Delphine Draelants for their fruitful discussions.

\bibliographystyle{siamplain}
\bibliography{mybibfile}

\begin{thebibliography}{10}

\bibitem{Abrikosov1957}
{\sc A.~Abrikosov}, {\em On the magnetic properties of superconductors of the
  second group}, Sov. Phys. JETP, 5 (1957), pp.~1174--1182.

\bibitem{Baelus2002}
{\sc B.~Baelus and F.~Peeters}, {\em Dependence of the vortex configuration on
  the geometry of mesoscopic flat samples}, Phys. Rev. B, 65 (2002), p.~104515.

\bibitem{Beyn2002}
{\sc W.~Beyn, A.~Champneys, E.~Doedel, W.~Govaerts, Y.~Kuznetsov, and
  B.~Sandstede}, {\em Numerical continuation, and computation of normal forms},
  in Handbook of dynamical systems, vol.~2, Elsevier, 2002, pp.~149--219.

\bibitem{Chan1986}
{\sc T.~Chang and D.~C. Resasco}, {\em Generalized deflated block-elimination},
  SIAM J. Numer. Anal., 23 (1986), pp.~913--924.

\bibitem{Cordoba2013}
{\sc R.~C{\'o}rdoba et~al.}, {\em Magnetic field-induced dissipation-free state
  in superconducting nanostructures}, Nature communications, 4 (2013).

\bibitem{Cramer1750}
{\sc G.~Cramer}, {\em Introduction \`{a} l'analyse des lignes courbes
  alg\'{e}briques}, chez les fr\`{e}res Cramer et C. Philibert, 1750.

\bibitem{Demmel1997}
{\sc J.~Demmel}, {\em Applied numerical linear algebra}, SIAM, 1997.

\bibitem{MatCont}
{\sc A.~Dhooge, W.~Govaerts, and Y.~Kuznetsov}, {\em Matcont: a matlab package
  for numerical bifurcation analysis of {ODE}s}, {ACM} Transactions on
  Mathematical Software, 29 (2003), pp.~141--164.

\bibitem{Auto}
{\sc E.~Doedel}, {\em Auto: A program for the automatic bifurcation analysis of
  autonomous systems}, Cong. Numer., 30 (1981), pp.~265--284.

\bibitem{Dorojevets2015}
{\sc M.~{Dorojevets} and Z.~{Chen}}, {\em Fast pipelined storage for
  high-performance energy-efficient computing with superconductor technology},
  in 2015 12th International Conference Expo on Emerging Technologies for a
  Smarter World (CEWIT), Oct 2015, pp.~1--6,
  \url{https://doi.org/10.1109/CEWIT.2015.7338159}.

\bibitem{Draelants2015}
{\sc D.~Draelants, P.~Klosiewicz, J.~Broeckhove, and W.~Vanroose}, {\em Solving
  general auxin transport models with a numerical continuation toolbox in
  {P}ython: {P}y{NCT}}, Lecture notes in computer science, 9271 (2015),
  pp.~211--225.

\bibitem{Du1992}
{\sc Q.~Du, M.~Gunzburger, and J.~Peterson}, {\em Analysis and approximation of
  the {G}inzburg-{L}andau model of superconductivity}, SIAM Rev., 34 (1992),
  pp.~54--81.

\bibitem{Embon2017}
{\sc L.~Embon et~al.}, {\em Imaging of super-fast dynamics and flow
  instabilities of superconducting vortices}, Nature communications, 8 (2017).

\bibitem{Filippov2012}
{\sc T.~Filippov et~al.}, {\em 20 ghz operation of an asynchronous
  wave-pipelined rsfq arithmetic-logic unit}, Physics Procedia, 36 (2012),
  pp.~59--65.

\bibitem{Fornberg1988}
{\sc B.~Fornberg}, {\em Generation of finite difference formulas on arbitrarily
  spaced grids}, Mathematics of Computation, 51 (1988), pp.~699--706.

\bibitem{Gaul2011}
{\sc A.~Gaul, M.~Gutknecht, J.~Liesen, and R.~Nabben}, {\em Deflated and
  augmented {K}rylov subspace methods: {B}asic facts and a breakdown-free
  deflated {MINRES}}, in {MINRES, MATHEON P}reprint 759, {T}echnical
  {U}niversity {B}erlin, 2011.

\bibitem{Gaul2013}
{\sc A.~Gaul, M.~Gutknecht, J.~Liesen, and R.~Nabben}, {\em A framework for
  deflated and augmented {K}rylov subspace methods}, SIAM J. Matrix Anal.
  Appl., 34 (2013), pp.~495--518.

\bibitem{pynosh}
{\sc A.~Gaul and N.~Schl{\"o}mer}, {\em Py{N}osh: {P}ython framework for
  nonlinear {S}chr{\"o}dinger equations}.
\newblock https://github.com/nschloe/pynosh, August 2013.

\bibitem{Gaul2015}
{\sc A.~Gaul and N.~Schl{\"o}mer}, {\em Preconditioned recycling {K}rylov
  subspace methods for self-adjoint problems}.
\newblock 2015.

\bibitem{Golubitsky2002}
{\sc M.~Golubitsky and I.~Stewart}, {\em The symmetry perspective},
  Birkh{\"a}user, 2002.

\bibitem{Goodman1966}
{\sc B.~Goodman}, {\em Type {II} superconductors}, Reports on progress in
  physics, 29 (1966), pp.~445--483.

\bibitem{Govaerts1991}
{\sc W.~Govaerts}, {\em Stable solvers and block elimination for bordered
  systems}, SIAM J. Matrix Anal. Appl., 12 (1991), pp.~469--483.

\bibitem{Holmes2013}
{\sc D.~S. {Holmes} et~al.}, {\em Energy-efficient superconducting
  computing—power budgets and requirements}, IEEE Transactions on Applied
  Superconductivity, 23 (2013),
  \url{https://doi.org/10.1109/TASC.2013.2244634}.

\bibitem{Hoyle2006}
{\sc R.~Hoyle}, {\em Pattern formation}, Cambridge University Press, 2006.

\bibitem{Keller1986}
{\sc H.~Keller}, {\em Lectures on numerical methods in bifurcation problems},
  Springer-Verlag, 1986.

\bibitem{Kleiner2004}
{\sc R.~Kleiner et~al.}, {\em Superconducting quantum interference devices:
  State of the art and applications}, Proceedings of the IEEE, 92 (2004),
  pp.~1534--1548.

\bibitem{Mei2000}
{\sc Z.~Mei}, {\em Numerical bifurcation analysis for reaction-diffusion
  equations}, Springer, 2000.

\bibitem{Mei1996}
{\sc Z.~Mei and A.~Schwarzer}, {\em Scaling solution branches of one-parameter
  bifurcation problems}, Journal of Mathematical Analysis and Applications, 204
  (1996), pp.~102--123.

\bibitem{Murphy2017}
{\sc A.~Murphy et~al.}, {\em Nanoscale superconducting memory based on the
  kinetic inductance of asymmetric nanowire loops}, New Journal of Physics, 19
  (2017).

\bibitem{Loca}
{\sc A.~Salinger et~al.}, {\em Loca 1.1 library of continuation algorithms:
  theory and implementation manual}, Tech. Report SAND2002-0396, Sandia
  National Laboratories, Albuquercue, NM, 2002.

\bibitem{Schlomer_square}
{\sc N.~Schl{\"o}mer, D.~Avitabile, and W.~Vanroose}, {\em Numerical
  bifurcation study of superconducting patterns on a square}, SIAM J. Appl.
  Dyn. Syst., 11 (2012), pp.~447--477.

\bibitem{Schlomer_disc}
{\sc N.~Schl{\"o}mer, M.~Milosevic, B.~Partoens, and W.~Vanroose}, {\em
  Exploration of stable and unstable vortex patterns in a superconductor under
  a magnetic disc}.
\newblock 2013.

\bibitem{Schlomer_solver}
{\sc N.~Schl{\"o}mer and W.~Vanroose}, {\em An optimal linear solver for the
  {J}acobian system of the extreme type-{II} {G}inzburg-{L}andau problem},
  Journal of Computational Physics, 234 (2013), pp.~560--572.

\end{thebibliography}
\end{document}